\documentclass[12pt]{amsart}
\usepackage{}

\usepackage{amsmath}
\usepackage{amsfonts}
\usepackage{amssymb}
\usepackage[all]{xy}           

\usepackage{bbding}
\usepackage{txfonts}
\usepackage{amscd}

\usepackage[shortlabels]{enumitem}
\usepackage{ifpdf}
\ifpdf
  \usepackage[colorlinks,final,
  hyperindex]{hyperref}
\else
  \usepackage[colorlinks,final,
  hyperindex]{hyperref}
\fi
\usepackage{tikz}
\usepackage[active]{srcltx}

\numberwithin{equation}{section}

\makeatletter

\newtheorem{thm}{Theorem}[section]
\newtheorem{lem}[thm]{Lemma}
\newtheorem{cor}[thm]{Corollary}
\newtheorem{pro}[thm]{Proposition}
\newtheorem{ex}[thm]{Example}

\newtheorem{defi}[thm]{Definition}

\newcommand{\g}{\mathfrak g}
\newcommand{\kt}{\mathfrak t}

\newcommand{\kl}{\mathfrak l}
\newcommand{\kr}{\mathfrak r}

\newcommand{\bk}{\mathbf{k}}
\newcommand{\bz}{\mathbb Z}

\newcommand{\End}{\mathrm{End}}
\newcommand{\ad}{\mathrm{ad}}

\newcommand{\fl}{\mathbf l}
\newcommand{\fr}{\mathbf r}
\newcommand{\fu}{\mathbf u}
\def\id{\mathop {\fam0 id}\nolimits}
\newcommand{\ZYBE}{\mathrm{ZYBE}}
\newcommand{\CYBE}{\mathrm{CYBE}}
\newcommand{\AYBE}{\mathrm{AYBE}}
\newcommand{\PPYBE}{\mathrm{PPYBE}}
\newcommand{\PLYBE}{\mathrm{PLYBE}}
\newcommand{\PYBE}{\mathrm{PYBE}}

\begin{document}

\title[Affinization of Zinbiel and pre-Poisson bialgebras, Poisson bialgebras]
{Affinization of Zinbiel bialgebras and pre-Poisson bialgebras, infinite-dimensional
Poisson bialgebras}

\author{Yanhong Guo}
\address{School of Mathematics and Statistics, Henan University, Kaifeng 475004,
China}
\email{guoaihong1983@163.com}

\author{Bo Hou}
\address{School of Mathematics and Statistics, Henan University, Kaifeng 475004,
China}
\email{houbo@henu.edu.cn, bohou1981@163.com}

\vspace{-10mm}

\begin{abstract}
The purpose of this paper is to construct infinite-dimensional Poisson bialgebras by the
affinization of pre-Poisson algebras. There is a natural Poisson algebra structure on the
tensor product of a pre-Poisson algebra and a perm algebra, and the Poisson algebra
structure on the tensor product of a pre-Poisson algebra and a special perm algebra
characterizes the pre-Poisson algebra. We extend such correspondences to the context
of bialgebras, that is, there is a Poisson bialgebra structure on the tensor product
of a pre-Poisson bialgebra and a quadratic $\bz$-graded perm algebra.
In this process, we provide the affinization of Zinbiel bialgebras, and give a
correspondence between symmetric solutions of the Yang-Baxter equation in pre-Poisson
algebras and certain skew-symmetric solutions of the Yang-Baxter equation in the induced
infinite-dimensional Poisson algebras. The similar correspondences for the related
triangular bialgebra structures and $\mathcal{O}$-operators are given.
\end{abstract}

\makeatletter
\@namedef{subjclassname@2020}{\textup{2020} Mathematics Subject Classification}
\makeatother

\maketitle

\vspace{-12mm}
\tableofcontents 



\vspace{-12mm}

\section{Introduction}\label{sec:intr}
The purpose of this paper is to give a construction of infinite-dimensional Poisson
bialgebras via the affinization of pre-Poisson bialgebras.  A pre-Poisson
algebra is both a Zinbiel algebra and a pre-Lie algebra which are compatible in a
certain sense. We use the affinization of Zinbiel bialgebras and pre-Lie bialgebras
to achieve the affinization of pre-Poisson bialgebras.

\smallskip\noindent
1.1. {\bf Poisson algebras and pre-Poisson algebras.}
Poisson algebras were originally coming from the study of the Hamiltonian mechanics.
A Poisson algebra is the algebraic structure corresponding to a Poisson manifold.
Poisson algebras serve as fundamental structures in various areas of mathematics
and mathematical physics, including Poisson geometry \cite{Vai,Wei},
algebraic geometry \cite{GK,Pol}, classical and quantum mechanics \cite{Arn,Dir},
quantum groups \cite{Dri} and quantization theory \cite{Hue,Kon}.

The notion of a pre-Poisson algebra was first introduced by Aguiar \cite{Agu},
which combines a Zinbiel algebra and a pre-Lie algebra on the same vector space
satisfying some compatibility conditions. The notion of Zinbiel algebras
(also called dual Leibniz algebras) was introduced by Loday in \cite{Lod1}, and further
studied in \cite{Liv,Lod2}. Pre-Lie algebras (also called left-symmetric algebras,
quasi-associative algebras, Vinberg algebras and so on) are a class of nonassociative
algebras coming from the study of convex homogeneous cones, affine manifolds and
affine structures on Lie groups, and cohomologies of associative algebras \cite{Bur,Baa}.
They also appeared in many fields in mathematics and mathematical physics, such as complex
and symplectic structures on Lie groups and Lie algebras, Poisson brackets and infinite
dimensional Lie algebras, vertex algebras, integrable systems and quantum field theory.
A pre-Poisson algebra naturally gives rise to its sub-adjacent Poisson algebra through the
anti-commutator of Zinbiel algebras and the commutator of pre-Lie algebras.
On the other hand, a Rota-Baxter operator of weight zero on a Poisson algebra can
give rise to a pre-Poisson algebra \cite{Agu}.

\smallskip\noindent
1.2. {\bf Bialgebras theory and classical Yang-Baxter equation.}
A bialgebra structure on a given algebra structure is obtained as a coalgebra
structure together which gives the same algebra structure on the dual space with
a set of compatibility conditions between the products and coproducts.
Two famous examples of bialgebras are Lie bialgebras \cite{Dri} and antisymmetric
infinitesimal bialgebras \cite{Zhe,Bai}. Lie bialgebras can be viewed as a linearization
of Poisson-Lie groups. It were first introduced by Drinfeld in the context of the
theory of Yang-Baxter equations and quantum groups.
Infinitesimal bialgebras first appeared in the work of Joni and Rota to give an
algebraic framework for the calculus of divided differences \cite{JR}.

The bialgebra theories of various algebra structures have been extensively developed,
such as pre-Lie bialgebras \cite{Bai1}, Jordan bialgebras \cite{Zhe}, Novikov
bialgebras \cite{HBG}, perm bialgebras \cite{Hou,LZB}, and so on. The Poisson
bialgebras have been studied in \cite{NB}. Recently, Wang and Sheng have given a
bialgebra theory of pre-Poisson algebras and shown that there is a one-to-one
correspondence between quadratic Rota-Baxter pre-Poisson algebras and factorizable
pre-Poisson bialgebras \cite{WS}. In \cite{LL}, Lin and Lu have studied the
quasi-triangular and factorizable Poisson bialgebras. In this paper, we construct
an infinite-dimensional Poisson bialgebras by the affinization of pre-Poisson bialgebras.
More precisely, we show that there is a Poisson bialgebra structure on the tensor product
of a pre-Poisson bialgebra and a quadratic perm algebra, and a bialgebra structure
for a pre-Poisson algebra which could be characterized by the fact that its affinization
by a special quadratic perm algebra gives an infinite-dimensional Poisson bialgebra.

The classical Yang-Baxter equation (for short $\CYBE$) arose from the study of
inverse scattering theory in the 1980s and was recognized as the semi-classical limit
of the quantum Yang-Baxter equation \cite{Yang,Bax}. The study of the $\CYBE$
in a Lie algebra has substantial ramifications and applications in the areas
of symplectic geometry, quantum groups, integrable systems, and quantum field theory
\cite{BD,Sto}. The classical Yang-Baxter equations in a Poisson algebra and a
pre-Poisson algebra have been studied in \cite{NB} and \cite{WS}. Similar to
the results for Lie algebra, we can obtain that a skew-symmetric solution of the
Yang-Baxter equation in a Poisson algebra (resp, a symmetric solution of the Yang-Baxter
equation in a pre-Poisson algebra) induces a triangular Poisson bialgebra structure (resp.
a triangular pre-Poisson bialgebra structure). In this paper, we show that any symmetric
solution of the Yang-Baxter equation in a pre-Poisson algebra can be elevated to a
skew-symmetric solution of the Yang-Baxter equation in the Poisson algebra obtained by
the affinization, thus obtaining that the Poisson bialgebra obtained by the affinization
is triangular if the pre-Poisson bialgebra is triangular.

\smallskip\noindent
1.3. {\bf Affinization of pre-Poisson bialgebras.}
An important construction of infinite-dimensional algebras is a process of
affinization which equips an algebra structure on the tensor product, over a
field $\bk$, $A\otimes\bk[\kt, \kt^{-1}]$, of a certain algebra $A$ of finite dimension
with another algebra structure on the space of Laurent polynomials $\bk[\kt, \kt^{-1}]$
\cite{BN}. The classical instance of the affinization is the Lie algebra affinization
$L\otimes\bk[\kt, \kt^{-1}]$, where $L$ is a finite-dimensional Lie algebra and
$\bk[\kt, \kt^{-1}]$ is the commutative associative algebra \cite{Kac}.
Recently, further research has been conducted on the affinization of the bialgebra structure.
Since the operad of (left) Novikov algebras and the operad of right Novikov algebras
are Koszul dual, in \cite{HBG}, Hong, Bai and Guo have proposed a method for
constructing infinite-dimensional Lie bialgebras using the affinization of Novikov
bialgebras. Similarly, since the operads of perm algebras and pre-Lie algebras are
the Koszul dual each other, Lin, Zhou and Bai have constructed infinite-dimensional
Lie bialgebras by using the pre-Lie bialgebras and perm bialgebras, respectively \cite{LZB}.

Roughly speaking, the affinization of a given algebra structure is to define an
algebra structure on the vector space of Laurent polynomials and obtain another algebra
structure on the tensor product, which could resolve the given algebra structure in turn.
In this paper, we consider the affinization of Zinbiel bialgebras and pre-Poisson
bialgebras. We first show that the tensor product of a Zinbiel algebra and a perm algebra
is a commutative associative algebra and when the perm algebra is a special
perm algebra on the vector space of Laurent polynomials, the tensor product characterizes
the Zinbiel algebra as a property of the associative algebra. Next, we consider
the dual version of the Zinbiel algebra affinization for Zinbiel coalgebras.
We introduce the notion of completed cocommutative coassociative coalgebras
and prove that there is a natural completed cocommutative coassociative coalgebra
structure on the tensor product of a Zinbiel coalgebra and the completed
perm coalgebra. Therefore, we use the quadratic $\bz$-graded perm algebra lifts the
affinization characterizations of Zinbiel algebras and Zinbiel coalgebras to
the level of Zinbiel bialgebras. Based on these results and the affinization of
pre-Lie bialgebras given in \cite{LZB}, we present the affinization of pre-Lie bialgebras.

\smallskip\noindent
{\bf Theorem.} {\it Let $(A, \ast, \circ, \vartheta, \theta)$ be a finite-dimensional
pre-Poisson bialgebra, $(B=\oplus_{i\in\bz}B_{i}, \diamond, \omega)$ be a quadratic
$\bz$-graded perm algebra and $(A\otimes B, \cdot, [-,-])$ be the induced $\bz$-graded
Poisson algebra from $(A, \ast, \circ)$ by $(B=\oplus_{i\in\bz}B_{i}, \diamond)$.
Define two linear maps $\Delta, \delta: A\otimes B\rightarrow(A\otimes B)\otimes
(A\otimes B)$ by Eqs. \eqref{Zcoass} and \eqref{PLcoLie} for $\nu=\nu_{\omega}$.
Then $(A\otimes B, \cdot, [-,-], \Delta, \delta)$ is a completed Poisson bialgebra.

Moreover, if $(B=\oplus_{i\in\bz}B_{i}, \diamond, \omega)$ is the quadratic $\bz$-graded
perm algebra given in Example \ref{ex:qu-perm}, then $(A\otimes B, \cdot, [-,-],
\Delta, \delta)$ is a completed Poisson bialgebra if and only if
$(A, \ast, \circ, \vartheta, \theta)$ is a pre-Poisson bialgebra.}

\smallskip\noindent
Moreover, we also consider the affinization of quasi-Frobenius pre-Poisson algebras.

With the development of bialgebra theory, the relationship between different types
of bialgebras has attracted the attention of scholars \cite{HBG,LZB,LLB,HBG1,CH}.
It is worth noting that although a pre-Poisson algebra naturally has a Poisson algebra
structure (by the anti-commutator of Zinbiel algebras and the commutator of pre-Lie algebras),
there is generally no Poisson bialgebra structure on a pre-Poisson bialgebra unless strong
conditions are added. Here, we establish the connection between Poisson bialgebras and
pre-Poisson bialgebras using the affinization.

\smallskip\noindent
1.4. {\bf Outline of the paper.}
This paper is organized as follows. In Section \ref{sec:poi-alg}, we recall the
notions of pre-Poisson algebras, Poisson algebras perm algebras and representations of
pre-Poisson algebras and Poisson algebras. Then we show in Proposition \ref{pro:tensor-ass}
that there are Poisson algebra structures on the tensor products of perm algebras
and pre-Poisson algebras. In Section \ref{sec:zbialg}, we give the affinization of
Zinbiel bialgebras by quadratic $\bz$-graded perm algebras. We introduce the completed
tensor product, completed coproduct, and show in Proposition \ref{pro:Zco-coass}
that there exists a completed cocommutative coassociative coalgebra structure on the
tensor product of a Zinbiel coalgebra and a completed perm coalgebra, which could
give a characterization of the Zinbiel coalgebra by its affinization. Therefore,
we show in Theorem \ref{thm:Z-perm-ass} that there is a natural completed infinitesimal
bialgebra structure on the tensor product of a perm bialgebra and a quadratic
$\bz$-graded perm algebra. In Section \ref{sec:pbialg}, based on the affinization of
Zinbiel bialgebras and pre-Lie bialgebras, we provide the affinization of pre-Poisson
bialgebras in Theorem \ref{thm:PP-Poisbia}. Moreover, we present a natural construction
of skew-symmetric solutions of the $\PYBE$ in the Poisson algebra $(A\otimes B, \cdot,
[-,-])$ from the symmetric solutions of the $\PPYBE$ in pre-Poisson algebra $(A, \ast,
\circ)$. Finally, a construction of quasi-Frobenius Poisson algebras from quasi-Frobenius
pre-Poisson algebras corresponding to a class of symmetric solutions of the $\PPYBE$ is given.

Throughout this paper, we fix $\bk$ as a field of characteristic zero.
All the vector spaces, algebras are over $\bk$ and are finite-dimensional
unless otherwise specified, and all tensor products are also over $\bk$.
For any vector space $V$, we denote $V^{\ast}$ the dual space of $V$. We denote the
identity map by $\id$. For vector spaces $A$ and $V$, we fix the following notations:
\begin{enumerate}
\item[$(i)$] Let $f: V\rightarrow V$ be a linear map. Define a linear map $f^{\ast}:
      V^{\ast}\rightarrow V^{\ast}$ by
      $\langle f^{\ast}(\xi),\; v\rangle=\langle\xi,\; f(v)\rangle$,
      for any $v\in V$ and $\xi\in V^{\ast}$.
\item[$(ii)$] Let $\mu: A\rightarrow\End_{\bk}(V)$ be a linear map. Define a linear
      map $\mu^{\ast}: A\rightarrow\End_{\bk}(V^{\ast})$ by
      $\langle\mu^{\ast}(a)(\xi),\; v\rangle=-\langle\xi,\; \mu(a)(v)\rangle$,
      for any $a\in A$, $v\in V$ and $\xi\in V^{\ast}$.
\item[$(iii)$] For any $r\in A\otimes A$, we define a linear map
      $r^{\sharp}: A^{\ast}\rightarrow A$ by
      $\langle\xi_{2},\; r^{\sharp}(\xi_{1})\rangle=\langle\xi_{1}\otimes\xi_{2},\;
      r\rangle$, for any $\xi_{1}, \xi_{2}\in A^{\ast}$.
\end{enumerate}

\section{Poisson bialgebras and pre-Poisson bialgebras}\label{sec:poi-alg}
In this section, we recall the notions of Poisson algebras, pre-Poisson algebras
and the bialgebra theory of Poisson algebras and pre-Poisson algebras.
For the details, see \cite{NB,Agu,WS}.
First we recall the notion of Poisson algebras and pre-Poisson algebras.

\begin{defi}\label{def:Palg}
A {\bf Poisson algebra} is a triple $(P, \cdot, [-,-])$, where $(P, \cdot)$ is a
commutative associative algebra and $(P, [-,-])$ is a Lie algebra, such that the
Leibniz rule holds:
$$
[p_{1},\; p_{2}p_{3}]=[p_{1}, p_{2}]p_{3}+p_{2}[p_{1}, p_{3}],
$$
for any $p_{1}, p_{2}, p_{3}\in P$, where $p_{1}p_{2}:=p_{1}\cdot p_{2}$.
\end{defi}

Let $(A, \cdot)$ be a commutative associative algebra. Recall that a representation of
$(A, \cdot)$ is a pair $(V, \mu)$, where $V$ is a vector space and $\mu:
A\rightarrow\End_{\bk}(V)$ is a linear map such that $\mu(a_{1})\mu(a_{2})=\mu(a_{1}a_{2})$
for any $a_{1}, a_{2}\in A$. Let $(\g, [-,-])$ be a Lie algebra. Recall that a
representation of $(\g, [-,-])$ is a pair $(V, \rho)$, where $V$ is a vector space
and $\mu: \g\rightarrow\End_{\bk}(V)$ is a linear map such that $\rho(g_{1}g_{2})=
\rho(a_{1})\rho(a_{2})-\rho(a_{2})\rho(a_{1})$ for any $g_{1}, g_{2}\in\g$.
Let $(P, \cdot, [-,-])$ be a Poisson algebra, $V$ be a vector space and
$\mu, \rho: P\rightarrow\End_{\bk}(V)$ be two linear maps. Recall that $(V, \mu, \rho)$ is
called a {\bf representation of} $(P, \cdot, [-,-])$, if $(V, \rho)$ is a
representation of the Lie algebra $(P, [-,-])$ and $(V, \mu)$ is a representation
of the commutative associative algebra $(P, \cdot)$ and $\mu, \rho$ satisfy
the following conditions:
\begin{align*}
\rho(p_{1}p_{2})&=\mu(p_{2})\rho(p_{1})+\mu(p_{1})\rho(p_{2}),\\
\mu([p_{1}, p_{2}])&=\rho(p_{1})\mu(p_{2})-\mu(p_{2})\rho(p_{1}),
\end{align*}
for any $p_{1}, p_{2}\in P$. In particular, $(P, \fu_{P}, \ad_{P})$ is a
representation of the Poisson algebra $(P, \cdot, [-,-])$, which is called the
{\bf regular representation} of $(P, \cdot, [-,-])$, where $\fu_{P}(p_{1})(p_{2})
=p_{1}p_{2}$ and $\ad_{P}(p_{1})(p_{2})$ $=[p_{1}, p_{2}]$ for any $p_{1}, p_{2}\in P$.
Moreover, $(P^{\ast}, -\fu^{\ast}_{P}, \ad^{\ast}_{P})$ is also a representation of
the Poisson algebra $(P, \cdot, [-,-])$, which is called the {\bf coregular
representation} of $(P, \cdot, [-,-])$.

\begin{defi}\label{def:Zalg}
A {\bf Zinbiel algebra} $(A, \ast)$ is a vector space $A$ together with a bilinear
operation $\ast: A\otimes A\rightarrow A$ such that
$$
a_{1}\ast(a_{2}\ast a_{3})=(a_{1}\ast a_{2})\ast a_{3}+(a_{2}\ast a_{1})\ast a_{3},
$$
for any $a_{1}, a_{2}, a_{3}\in A$.
\end{defi}

In a Zinbiel algebra $(A, \ast)$, it is easy to see that $a_{1}\ast(a_{2}\ast a_{3})
=a_{2}\ast(a_{1}\ast a_{3})$ for any $a_{1}, a_{2}, a_{3}\in A$. Moreover, if we define
$a_{1}a_{2}=a_{1}\ast a_{2}+a_{2}\ast a_{1}$, then $(A, \cdot)$ is a commutative
associative algebra.

\begin{defi}\label{def:PLalg}
A {\bf pre-Lie algebra } $(A, \circ)$ is a vector space $A$ equipped with a bilinear
multiplication $\circ: A\otimes A\rightarrow A$ such that for any $a_{1}, a_{2}, a_{3}\in A$,
$$
(a_{1}\circ a_{2})\circ a_{3}-a_{1}\circ(a_{2}\circ a_{3})
=(a_{2}\circ a_{1})\circ a_{3}-a_{2}\circ(a_{1}\circ a_{3}).
$$
\end{defi}

Let $(A, \circ)$ be a pre-Lie algebra. Then the commutator $[a_{1}, a_{2}]:=
a_{1}\circ a_{2}-a_{2}\circ a_{1}$ for any $a_{1}, a_{2}\in A$ defines a Lie algebra
$(A, [-,-])$, which is called the sub-adjacent Lie algebra of $(A, \circ)$.

\begin{defi}[\cite{Agu}]\label{def:PPalg}
A {\bf pre-Poisson algebra} is a triple $(A, \ast, \circ)$, where $(P, \ast)$ is a
Zinbiel algebra and $(P, \circ)$ is a pre-Lie algebra, satisfying the following
compatibility conditions:
\begin{align}
(a_{1}\circ a_{2}-a_{2}\circ a_{1})\ast a_{3}
&=a_{1}\circ(a_{2}\ast a_{3})-a_{2}\ast(a_{1}\circ a_{3}), \label{PPalg1}\\
(a_{1}\ast a_{2}+a_{2}\ast a_{1})\circ a_{3}
&=a_{1}\ast(a_{2}\circ a_{3})+a_{2}\ast(a_{1}\circ a_{3}), \label{PPalg2}
\end{align}
for any $a_{1}, a_{2}, a_{3}\in A$.
\end{defi}

\begin{ex}\label{ex:bialia}
Let $A$ be a two-dimensional vector space with a basis $\{e_{1}, e_{2}\}$.
If we define two linear maps $\ast, \circ: A\otimes A\rightarrow A$ by
$e_{1}\ast e_{1}=ae_{2}$, $e_{1}\circ e_{1}=be_{1}+ce_{2}$, $e_{1}\circ e_{2}=be_{2}
=e_{2}\circ e_{1}$ and $e_{1}\cdot e_{2}=e_{2}\cdot e_{1}=e_{2}\cdot e_{2}
=e_{2}\circ e_{2}=0$, then $(A, \ast, \circ)$ is a pre-Poisson algebra.
\end{ex}

Let $(A, \ast, \circ)$ be a pre-Poisson algebra. Then $(A, \cdot, [-,-])$ is a
Poisson algebra, where the bracket $[-,-]$ and the multiplication $\cdot$
are given by $[a_{1}, a_{2}]=a_{1}\circ a_{2}-a_{2}\circ a_{1}$ and
$a_{1}a_{2}=a_{1}\ast a_{2}+a_{2}\ast a_{1}$ for any $a_{1}, a_{2}\in A$.
On the other hand, recall that a {\bf Rota-Baxter operator} on a Poisson algebra
$(P, \cdot, [-,-])$ is a linear map $\mathcal{R}: P\rightarrow P$ satisfying
\begin{align*}
\mathcal{R}(p_{1})\mathcal{R}(p_{2})&=\mathcal{R}\big(\mathcal{R}(p_{1})p_{2}
+p_{1}\mathcal{R}(p_{2})\big),\\
[\mathcal{R}(p_{1}),\; \mathcal{R}(p_{2})]&=\mathcal{R}\big([\mathcal{R}(p_{1}),\, p_{2}]
+[p_{1},\, \mathcal{R}(p_{2})]\big),
\end{align*}
for any $p_{1}, p_{2}\in P$. Let $\mathcal{R}: P\rightarrow P$ be a Rota-Baxter
operator of a Poisson algebra $(P, \cdot, [-,-])$. Define new operations
$\ast_{\mathcal{R}}, \circ_{\mathcal{R}}: P\otimes P\rightarrow P$ by
$$
p_{1}\ast_{\mathcal{R}}p_{2}=\mathcal{R}(p_{1})p_{2}, \qquad
p_{1}\circ_{\mathcal{R}}p_{2}=[\mathcal{R}(p_{1}),\; p_{2}],
$$
for any $p_{1}, p_{2}\in P$. Then $(P, \ast_{\mathcal{R}}, \circ_{\mathcal{R}})$
is a pre-Poisson algebra \cite{Agu}.

Let us recall the notions of representations of Zinbiel algebras and pre-Lie algebras.
Then we introduce the representations of pre-Poisson algebras.
A {\bf representation of a Zinbiel algebra} $(A, \ast)$ is a triple $(V, \kl, \kr)$,
where $V$ is a vector space and $\kl, \kr: A\rightarrow\End_{\bk}(V)$ are linear maps
such that the following equations hold for all $a_{1}, a_{2}\in A$,
\begin{align*}
\kl(a_{1})\kl(a_{2})&=\kl(a_{1}\ast a_{2})+\kl(a_{2}\ast a_{1}), \\
\kr(a_{1}\ast a_{2})&=\kl(a_{1})\kr(a_{2})=\kr(a_{2})\kl(a_{1})+\kr(a_{2})\kr(a_{1}).
\end{align*}
A {\bf representation of a pre-Lie algebra} $(A, \circ)$ is a triple $(V, \hat{\kl},
\hat{\kr})$, where $V$ is a vector space, $\hat{\kl}, \hat{\kr}: A\rightarrow\End_{\bk}(V)$
are linear maps such that the following equations hold for all $a_{1}, a_{2}\in A$,
\begin{align*}
\hat{\kl}(a_{1})\hat{\kl}(a_{2})-\hat{\kl}(a_{1}\circ a_{2})
&=\hat{\kl}(a_{2})\hat{\kl}(a_{1})-\hat{\kl}(a_{2}\circ a_{1}),\\
\hat{\kl}(a_{1})\hat{\kr}(a_{2})-\hat{\kr}(a_{2})\hat{\kl}(a_{1})
&=\hat{\kr}(a_{1}\circ a_{2})-\hat{\kr}(a_{2})\kr(a_{1}).
\end{align*}

\begin{defi}
A {\bf representation of a pre-Poisson algebra} $(A, \ast, \circ)$ is a quintuple
$(V, \kl, \kr, \hat{\kl}, \hat{\kr})$, where $(V, \kl, \kr)$ is a representation of
the Zinbiel algebra $(A, \ast)$, $(V, \hat{\kl}, \hat{\kr})$ is a representation
of the pre-Lie algebra $(A, \circ)$ and for any $a_{1}, a_{2}\in A$,
\begin{align*}
\kl(a_{1}\circ a_{2}-a_{2}\circ a_{1})
&=\hat{\kl}(a_{1})\kl(a_{2})-\kl(a_{2})\hat{\kl}(a_{1}),\\
\kr(a_{1}\circ_P a_{2})
&=\kr(a_{2})\hat{\kr}(a_{1})-\kr(a_{2})\hat{\kl}(a_{1})+\hat{\kl}(a_{1})\kr(a_{2}),\\
\kr(a_{1}\circ_P a_{2})
&=\hat{\kr}(a_{2})\kl(a_{1})+\hat{\kr}(a_{2})\kr(a_{1})-\kl(a_{1})\hat{\kr}(a_{2}),\\
\hat{\kl}(a_{1}\ast a_{2}+a_{2}\ast a_{1})
&=\kl(a_{1})\hat{\kl}(a_{2})+\kl(a_{2})\hat{\kl}(a_{1}),\\
\hat{\kr}(a_{1}\ast a_{2})
&=\kl(a_{1})\hat{\kr}(a_{2})+\kr(a_{2})\hat{\kr}(a_{1})-\kr(a_{2})\hat{\kl}(a_{1}).
\end{align*}
\end{defi}

Let $(A, \ast, \circ)$ be a pre-Poisson algebra. If we define linear maps $\fl_{A},
\fr_{A}, \hat{\fl}_{A}, \hat{\fr}_{A}: A\otimes\End_{\bk}(A)$ by $\fl_{A}(a_{1})(a_{2})
=a_{1}\ast a_{2}$, $\fr_{A}(a_{1})(a_{2})=a_{2}\ast a_{1}$, $\hat{\fl}_{A}(a_{1})(a_{2})
=a_{1}\circ a_{2}$ and $\hat{\fr}_{A}(a_{1})(a_{2})=a_{2}\circ a_{1}$ for any $a_{1},
a_{2}\in A$, then $(A, \fl_{A}, \fr_{A}, \hat{\fl}_{A}, \hat{\fr}_{A})$ is a
representation of the pre-Poisson algebra $(A, \ast, \circ)$, which is called the
{\bf regular representation} of $(A, \ast, \circ)$, and $(A^{\ast}, -\fl^{\ast}_{A}
-\fr^{\ast}_{A}, \fr^{\ast}_{A}, \hat{\fl}^{\ast}_{A}-\hat{\fr}^{\ast}_{A},
-\hat{\fr}^{\ast}_{A})$ is also a representation of the pre-Poisson algebra $(A,
\ast, \circ)$, which is called the {\bf coregular representation} of $(A, \ast, \circ)$.

Next, we consider the bialgebra theory of Poisson algebras and pre-Poisson algebras.
Recall that a {\bf coassociative coalgebra} $(A, \Delta)$ is a vector space $A$
with a linear map $\Delta: A\rightarrow A\otimes A$ satisfying the coassociative law:
$$
(\Delta\otimes\id)\Delta=(\id\otimes\Delta)\Delta.
$$
A coassociative coalgebra $(A, \Delta)$ is called {\bf cocommutative} if
$\Delta=\tau\Delta$, where $\tau: A\otimes A\rightarrow A\otimes A$ is the
flip operator defined by $\tau(a_{1}\otimes a_{2}):=a_{2}\otimes a_{1}$ for
all $a_{1}, a_{2}\in A$. An {\bf infinitesimal bialgebra} is a triple $(A, \cdot, \Delta)$,
where $(A, \cdot)$ is a commutative associative algebra, $(A, \Delta)$ is a cocommutative
coassociative coalgebra and for any $a_{1}, a_{2}\in A$,
$$
\Delta(a_{1}a_{2})=(\fu_{A}(a_{2})\otimes\id)\Delta(a_{1})
+(\id\otimes\,\fu_{A}(a_{1}))\Delta(a_{2}).
$$
Recall that  {\bf Lie coalgebra} $(\g, \delta)$
is a vector space $\g$ with a linear map $\delta: \g\rightarrow\g\otimes\g$,
such that $\tau\delta=-\delta$ and $(\id\otimes\delta)\delta
-(\tau\otimes\id)(\id\otimes\delta)\delta=(\delta\otimes\id)\delta$.
A {\bf Lie bialgebra} is a triple $(\g, [-,-], \delta)$ such that
$(\g, [-,-])$ is a Lie algebra, $(\g, \delta)$ is a Lie coalgebra, and the following
compatibility condition holds:
$$
\delta([g_{1}, g_{2}])=(\ad_{\g}(g_{1})\otimes\id+\id\otimes\ad_{\g}(g_{1}))
(\delta(g_{2}))-(\ad_{\g}(g_{2})\otimes\id+\id\otimes\ad_{\g}(g_{2}))(\delta(g_{1})),
$$
for any $g_{1}, g_{2}\in\g$. Let $(P, \Delta)$ be a cocommutative coassociative coalgebra
and $(P, \delta)$ be a Lie coalgebra. If $(\id\otimes\Delta)\delta=(\delta\otimes\id)\Delta
+(\tau\otimes\id)(\id\otimes\delta)\Delta$, then $(P, \Delta, \delta)$ is called a
{\bf Poisson coalgebra}. It is easy to see that $(P, \Delta, \delta)$ is a Poisson
coalgebra if and only if $(P^{\ast}, \Delta^{\ast}, \delta^{\ast})$ is a Poisson algebra.

\begin{defi}\label{def:Pbialg}
Let $(P, \cdot, [-,-])$ be a Poisson algebra and $(P, \Delta, \delta)$ be a Poisson
coalgebra. Then the quintuple $(P, \cdot, [-,-], \Delta, \delta)$ is called a
{\bf Poisson bialgebra} if
\begin{enumerate}
\item[$(i)$] $(P, \cdot, \Delta)$ is an infinitesimal bialgebra,
\item[$(ii)$] $(P, [-,-], \delta)$ is a Lie bialgebra, 		
\item[$(iii)$] $\Delta$ and $\delta$ are compatible in the following sense:
     for any $p_{1}, p_{2}\in P$,
\begin{align*}
&\Delta([p_{1}, p_{2}])=\big(\ad_{P}(p_{1})\otimes\id+\id\otimes\ad_{P}(p_{1})
\big)\Delta(p_{2})+\big(\fu_{P}(p_{2})\otimes\id
-\id\otimes\fu_{P}(p_{2})\big)\delta(p_{1}),\\
&\qquad\quad \delta(p_{1}p_{2})=(\fu_{P}(p_{1})\otimes\id)\delta(p_{2})
+(\fu_{P}(p_{2})\otimes\id)\delta(p_{1})\\[-1mm]
&\qquad\qquad\qquad\qquad+(\id\otimes\ad_{P}(p_{1}))\Delta(p_{2})
+(\id\otimes\ad_{P}(p_{2}))\Delta(p_{1}).
\end{align*}		
\end{enumerate}
\end{defi}

Recall that a {\bf Zinbiel coalgebra} is a vector space $A$ with a linear map
$\vartheta: A \rightarrow A\otimes A$ such that
\begin{align}
(\vartheta\otimes\id)\vartheta+(\tau\otimes\id)(\vartheta\otimes\id)\vartheta
=(\id\otimes\vartheta)\vartheta.                \label{zinbco}
\end{align}
A {\bf Zinbiel bialgebra} is a triple $(A, \ast, \vartheta)$, where $(A, \ast)$
is a Zinbiel algebra and $(A, \vartheta)$ is a Zinbiel coalgebra such that for
all $a_{1}, a_{2}\in A$, the following compatible conditions hold:
\begin{align}
\vartheta(a_{1}\ast a_{2})+\tau(\vartheta(a_{1}\ast a_{2}))
&=(\id\otimes\fr_{A}(a_{2}))\vartheta(a_{1})+(\fl_{A}(a_{1})\otimes\id)\vartheta(a_{2})
+\tau((\id\otimes\fl_{A}(a_{1}))\vartheta(a_{2})),  \label{Zbialg1}\\
\vartheta(a_{1}\ast a_{2}+a_{2}\ast a_{1})
&=(\fl_{A}(a_{1})\otimes\id)\vartheta(a_{2})+(\id\otimes\fl_{A}(a_{2}))\vartheta(a_{1})
+(\id\otimes\fr_{A}(a_{2}))\vartheta(a_{1}).           \label{Zbialg2}
\end{align}
Recall that a {\bf pre-Lie coalgebra} is a pair $(A, \theta)$, where $A$ is a vector
space and $\theta: A\rightarrow A\otimes A$ is a linear map satisfying
$$
(\id\otimes\theta)\theta-(\tau\otimes\id)(\id\otimes\theta)\theta
=(\theta\otimes\id)\theta-(\tau\otimes\id)(\theta\otimes\id)\theta,
$$
for any $a_{1}, a_{2}\in A$. Let $(A, \circ)$ be a pre-Lie algebra and
$(A, \theta)$ be a pre-Lie coalgebra. If for any $a_{1}, a_{2}\in A$, the
following compatibility conditions are satisfied:
\begin{align*}
&(\theta-\tau\theta)(a_{1}\circ a_{2})
-(\hat{\fl}_{A}(a_{1})\otimes\id)((\theta-\tau\theta)(a_{2})) \\[-1mm]
&\quad-(\id\otimes\hat{\fl}_{A}(a_{1}))((\theta-\tau\theta)(a_{2}))
-(\id\otimes\hat{\fr}_{A}(a_{2}))(\theta(a_{1}))
+(\hat{\fr}_{A}(a_{2})\otimes\id)(\tau(\theta(a_{1})))=0,\\
&\theta(a_{1}\circ a_{2}-a_{2}\circ a_{1})-(\id\otimes(\hat{\fr}_{A}(a_{2})
-\hat{\fl}_{A}(a_{2})))(\theta(a_{1})) \\[-1mm]
&\quad-(\id\otimes(\hat{\fl}_{A}(a_{1})-\hat{\fr}_{A}(a_{1})))(\theta(a_{2}))
-(\hat{\fl}_{A}(a_{1})\otimes\id)(\theta(a_{2}))
+(\hat{\fl}_{A}(a_{2})\otimes\id)(\theta(a_{1}))=0,
\end{align*}
then we call $(A, \circ, \theta)$ a {\bf pre-Lie bialgebra}.
Let $(A, \vartheta)$ be a Zinbiel coalgebra and $(A, \theta)$ be a pre-Lie coalgebra.
If
\begin{align}
&((\theta-\tau\theta)\otimes\id)\vartheta
=(\id\otimes\vartheta)\theta-(\tau\otimes\id)(\id\otimes\theta)\vartheta, \label{PPco1}\\
&((\vartheta+\tau\vartheta)\otimes\id)\theta
=(\id\otimes\theta)\vartheta+(\tau\otimes\id)(\id\otimes\theta)\vartheta, \label{PPco2}
\end{align}
then $(A, \vartheta, \theta)$ is called a {\bf pre-Poisson coalgebra}. It is easy to
see that $(A, \vartheta, \theta)$ is a pre-Poisson coalgebra if and only if $(A^{\ast},
\vartheta^{\ast}, \theta^{\ast})$ is a pre-Poisson algebra.

\begin{defi}[\cite{WS}]\label{def:PLbialg}
Let $(A, \ast, \circ)$ be a pre-Poisson algebra and $(A, \vartheta, \theta)$ be a
pre-Poisson coalgebra. Then the quintuple $(A, \ast, \circ, \vartheta, \theta)$ is
called a {\bf pre-Poisson bialgebra} if
\begin{enumerate}
\item[$(i)$] $(A, \ast, \vartheta)$ is a Zinbiel bialgebra,
\item[$(ii)$] $(A, \circ, \theta)$ is a pre-Lie bialgebra, 		
\item[$(iii)$] $\vartheta$ and $\theta$ are compatible in the following sense:
     for any $a_{1}, a_{2}\in A$,
\begin{align}
\theta(a_{1}\ast a_{2}+a_{2}\ast a_{1})
&=(\id\otimes(\fl_{A}+\fr_{A})(a_{2}))\theta(a_{1})
+(\id\otimes(\fl_{A}+\fr_{A})(a_{1}))\theta(a_{2}) \label{PPbialg1}\\[-1mm]
&\qquad -(\hat{\fl}_{A}(a_{1})\otimes\id)\vartheta(a_{2})
-(\hat{\fl}_{A}(a_{2})\otimes\id)\vartheta(a_{1}),  \nonumber\\
\vartheta(a_{1}\circ a_{2}-a_{2}\circ a_{1})
&=(\fl_{A}(a_{2})\otimes\id-\id\otimes(\fl_{A}+\fr_{A})(a_{2}))
\theta(a_{1})                               \label{PPbialg2}\\[-1mm]
&\qquad +(\hat{\fl}_{A}(a_{1})\otimes\id+\id\otimes(\hat{\fl}_{A}
-\hat{\fr}_{A})(a_{1}))\vartheta(a_{2}),          \nonumber\\
(\vartheta+\tau\vartheta)(a_{1}\circ a_{2})
&=-(\id\otimes\fr_{A}(a_{2}))\theta(a_{1})
-(\fr_{A}(a_{2})\otimes\id)(\tau(\theta(a_{1})))   \label{PPbialg3}\\[-1mm]
&\qquad +(\hat{\fl}_{A}(a_{1})\otimes\id
+\id\otimes\hat{\fl}_{A}(a_{1}))((\vartheta+\tau\vartheta)(a_{2})),    \nonumber\\
(\theta-\tau\theta)(a_{1}\ast a_{2})
&=(\id\otimes\fr_{A}(a_{2}))\theta(a_{1})+(\id\otimes\fl_{A}(a_{1}))
((\theta-\tau\theta)(a_{2}))                   \label{PPbialg4}\\[-1mm]
&\qquad +(\hat{\fr}_{A}(a_{2})\otimes\id)(\tau(\vartheta(a_{1})))
-(\hat{\fl}_{A}(a_{1})\otimes \id)((\vartheta+\tau\vartheta)(a_{2})).    \nonumber
\end{align}		
\end{enumerate}
\end{defi}

In \cite{NB}, Ni and Bai have studied the bialgebra theory and Yang-Baxter equations for
Poisson algebras. A Poisson bialgebra can be characterized by a Manin triple or
a matched pair of Poisson algebras. Recently, the quasi-triangular Poisson bialgebras
and quasi-triangular pre-Poisson bialgebras have been studied in \cite{LL} and \cite{WS}
respectively. Here we mainly consider the close relationship between (triangular)
Poisson bialgebras and (triangular) pre-Poisson bialgebras.
A pre-Lie algebra $(A, \circ)$ naturally induces a Lie algebra $(A, [-,-])$,
where $[a_{1}, a_{2}]=a_{1}\circ a_{2}-a_{2}\circ a_{1}$ for any $a_{1}, a_{2}\in A$.
A pre-Lie coalgebra $(A, \vartheta)$ natural induces a Lie coalgebra $(A, \delta)$,
where $\delta=\vartheta-\tau\vartheta$. But a pre-Lie bialgebra $(A, \circ, \vartheta)$
does not induce a Lie coalgebra $(A, [-,-], \delta)$ in general \cite{Bai}.
Thus, we cannot use this method to construct Poisson bialgebras using pre-Poisson bialgebras.
In \cite{LZB}, Lin, Zhou and Bai have constructed Lie bialgebras by
the tensor product of pre-Lie bialgebras and perm algebras. Recall that a
{\bf perm algebra} is a pair $(B, \diamond)$, where $B$ is a vector space and
$\diamond: B\otimes B\rightarrow B$ is a bilinear operator such that for
any $b_{1}, b_{2}, b_{3}\in B$,
$$
b_{1}\diamond(b_{2}\diamond b_{3})=(b_{1}\diamond b_{2})\diamond b_{3}
=(b_{2}\diamond b_{1})\diamond b_{3}.
$$
In a perm algebra $(B, \diamond)$, wa also have $b_{1}\diamond(b_{2}\diamond b_{3})
=b_{2}\diamond(b_{1}\diamond b_{3})$ for any $b_{1}, b_{2}, b_{3}\in B$.

\begin{pro}[\cite{LZB}]\label{pro:tensor-lie}
Let $(A, \circ)$ be a pre-Lie algebra and $(B, \diamond)$ be a perm algebra.
Define a binary bracket $[-,-]$ on $A\otimes B$ by
\begin{align}
[a_{1}\otimes b_{1},\; a_{2}\otimes b_{2}]
=(a_{1}\circ a_{2})\otimes(b_{1}\diamond b_{2})
-(a_{2}\circ a_{1})\otimes(b_{2}\diamond b_{1}),  \label{ind-lie}
\end{align}
for any $a_{1}, a_{2}\in A$ and $b_{1}, b_{2}\in B$. Then $(A\otimes B, \ast)$ is a
Lie algebra, which is called {\bf the Lie algebra induced from $(A, \circ)$
by $(B, \diamond)$}.
\end{pro}

Similar to this proposition, for the tensor product of Zinbiel algebras and perm algebras,
we have:

\begin{pro}\label{pro:tensor-ass}
Let $(A, \ast)$ be a Zinbiel algebra and $(B, \diamond)$ be a perm algebra.
Define a binary operator $\cdot$ on $A\otimes B$ by
\begin{align}
(a_{1}\otimes b_{1})(a_{2}\otimes b_{2})
=(a_{1}\ast a_{2})\otimes(b_{1}\diamond b_{2})
+(a_{2}\ast a_{1})\otimes(b_{2}\diamond b_{1}), \label{ind-ass}
\end{align}
for any $a_{1}, a_{2}\in A$ and $b_{1}, b_{2}\in B$. Then $(A\otimes B, \cdot)$ is a
commutative associative algebra, which is called {\bf the commutative associative
algebra induced from $(A, \ast)$ by $(B, \diamond)$}.
\end{pro}

\begin{proof}
Clearly, the operator $\cdot$ is commutative.
For any $a_{1}, a_{2}, a_{3}\in A$, $b_{1}, b_{2}, b_{3}\in B$, since
$(A, \ast)$ is a Zinbiel algebra and $(B, \diamond)$ is a perm algebra, we have
\begin{align*}
&\;(a_{1}\otimes b_{1})((a_{2}\otimes b_{2})(a_{3}\otimes b_{3}))
-((a_{1}\otimes b_{1})(a_{2}\otimes b_{2}))(a_{3}\otimes b_{3})\\
=&\;(a_{1}\ast(a_{2}\ast a_{3}))\otimes(b_{1}\diamond(b_{2}\diamond b_{3}))
+((a_{2}\ast a_{3})\ast a_{1})\otimes((b_{2}\diamond b_{3})\diamond b_{1})\\[-1mm]
&\quad+(a_{1}\ast(a_{3}\ast a_{2}))\otimes(b_{1}\diamond(b_{3}\diamond b_{2}))
+((a_{3}\ast a_{2})\ast a_{1})\otimes((b_{3}\diamond b_{2})\diamond b_{1})\\[-1mm]
&\quad-((a_{1}\ast a_{2})\ast a_{3})\otimes((b_{1}\diamond b_{2})\diamond b_{3})
-(a_{3}\ast(a_{1}\ast a_{2}))\otimes(b_{3}\diamond(b_{1}\diamond b_{2}))\\[-1mm]
&\quad-((a_{2}\ast a_{1})\ast a_{3})\otimes((b_{2}\diamond b_{1})\diamond b_{3})
-(a_{3}\ast(a_{2}\ast a_{1}))\otimes(b_{3}\diamond(b_{2}\diamond b_{1}))\\
\end{align*}
\begin{align*}
=&\;\Big(a_{1}\ast(a_{2}\ast a_{3})-(a_{1}\ast a_{2})\ast a_{3}
-(a_{2}\ast a_{1})\ast a_{3}\Big)\otimes(b_{1}\diamond(b_{2}\diamond b_{3}))\\[-1mm]
&\ \ +\Big((a_{2}\ast a_{3})\ast a_{1}+(a_{3}\ast a_{2})\ast a_{1}
-a_{3}\ast(a_{2}\ast a_{1})\Big)\otimes(b_{2}\diamond(b_{3}\diamond b_{1}))\qquad\;\\[-1mm]
&\ \ +\Big(a_{1}\ast(a_{3}\ast a_{2})-a_{3}\ast(a_{1}\ast a_{2})
\Big)\otimes(b_{3}\diamond(b_{1}\diamond b_{2}))\\
=&\; 0.
\end{align*}
That is, $(A, \cdot)$ is a commutative associative algebra.
\end{proof}

Not only that, we can get a Poisson algebra by the tensor product of a pre-Poisson
algebra and a perm algebra.

\begin{pro}\label{pro:tensor-poi}
Let $(A, \ast, \circ)$ be a pre-Poisson algebra and $(B, \diamond)$ be a perm algebra.
Define two binary operators $[-,-], \cdot: (A\otimes B)\otimes(A\otimes B)
\rightarrow A\otimes B$ by Eqs. \eqref{ind-lie} and \eqref{ind-ass} respectively. Then
$(A\otimes B, \cdot, [-,-])$ is a Poisson algebra, which is called {\bf the
Poisson algebra induced from $(A, \ast, \circ)$ by $(B, \diamond)$}.
\end{pro}

\begin{proof}
By Propositions \ref{pro:tensor-lie} and \ref{pro:tensor-ass}, we need
to show that the Leibniz rule holds. For any $a_{1}, a_{2}, a_{3}\in A$,
$b_{1}, b_{2}, b_{3}\in B$, since $(A, \ast, \circ)$ is a pre-Poisson algebra
and $(B, \diamond)$ is a perm algebra, we have
\begin{align*}
&\;[a_{1}\otimes b_{1},\; a_{2}\otimes b_{2}](a_{3}\otimes b_{3})
+(a_{2}\otimes b_{2})[a_{1}\otimes b_{1},\; a_{3}\otimes b_{3}]\\
=&\;\Big((a_{1}\circ a_{2})\ast a_{3}-(a_{2}\circ a_{1})\ast a_{3}
+a_{2}\ast(a_{1}\circ a_{3})\Big)\otimes(b_{1}\diamond(b_{2}\diamond b_{3}))\\[-1mm]
&\ \ +\Big(-a_{3}\ast(a_{2}\circ a_{1})-a_{2}\ast(a_{3}\circ a_{1})
\Big)\otimes(b_{2}\diamond(b_{3}\diamond b_{1}))\\[-1mm]
&\ \ +\Big(a_{3}\ast(a_{1}\circ a_{2})+(a_{1}\circ a_{3})\ast a_{2}
-(a_{3}\circ a_{1})\ast a_{2}\Big)\otimes(b_{3}\diamond(b_{1}\diamond b_{2}))\\
=&\; (a_{1}\circ(a_{2}\ast a_{3}))\otimes(b_{1}\diamond(b_{2}\diamond b_{3}))
-((a_{2}\ast a_{3})\circ a_{1}\otimes((b_{2}\diamond b_{3})\diamond b_{1})\\[-1mm]
&\quad+(a_{1}\circ(a_{3}\ast a_{2}))\otimes(b_{1}\diamond(b_{3}\diamond b_{2}))
-((a_{3}\ast a_{2})\circ a_{1}\otimes((b_{3}\diamond b_{2})\diamond b_{1})\\
=&\; [(a_{1}\otimes b_{1}),\ \ (a_{2}\otimes b_{2})(a_{3}\otimes b_{3})].
\end{align*}
Thus, $(A\otimes B, \cdot, [-,-])$ is a Poisson algebra.
\end{proof}

\begin{ex}\label{ex:preP-P}
Let $A=\bk\{e_{1}, e_{2}\}$ be a two dimensional vector space. Define two
bilinear operators $\ast, \circ: A\otimes A\rightarrow A$ by $e_{1}\ast e_{1}=e_{2}$,
$e_{1}\circ e_{1}=e_{1}$, $e_{1}\circ e_{2}=e_{2}\circ e_{1}=e_{2}$ and others are
all zero. Then $(A, \ast, \circ)$ is a pre-Poisson algebra. Consider the
two dimensional perm algebra $(B=\bk\{x_{1}, x_{2}\}, \diamond)$, where the nonzero
products are given by $x_{2}\diamond x_{1}=x_{1}$ and $x_{2}\diamond x_{2}=x_{2}$.
Then the bracket and the dot product defined by Eqs. \eqref{ind-lie} and \eqref{ind-ass}
give a Poisson algebra structure on $A\otimes B$ as follows:
\begin{align*}
&\qquad\qquad\qquad (e_{1}\otimes x_{1})(e_{1}\otimes x_{2})=e_{2}\otimes x_{1},\qquad\qquad
(e_{1}\otimes x_{2})(e_{1}\otimes x_{2})=e_{2}\otimes x_{2},\\
&[e_{1}\otimes x_{1},\; e_{1}\otimes x_{2}]=-e_{1}\otimes x_{1},\qquad
[e_{1}\otimes x_{1},\; e_{2}\otimes x_{2}]=-e_{2}\otimes x_{1},\qquad
[e_{1}\otimes x_{2},\; e_{2}\otimes x_{1}]=e_{2}\otimes x_{1}.
\end{align*}
\end{ex}

This paper mainly elevates this conclusion to the level of bialgebras by the affinization
of pre-Poisson bialgebras. A pre-Poisson algebra is a Zinbiel algebra and a pre-Lie algebra
on the same vector space satisfying some compatibility conditions.
The affinization of pre-Lie bialgebras has been considered in \cite{LZB}. In next
section, we start with the affinization of Zinbiel bialgebras.

\section{Infinite-dimensional infinitesimal bialgebras from Zinbiel bialgebras}
\label{sec:zbialg}
In this section, we recall the notion of a quadratic $\bz$-graded perm algebra, as
a $\bz$-graded perm algebra equipped with an invariant bilinear form. We show that
the tensor product of a finite-dimensional Zinbiel bialgebra and a quadratic
$\bz$-graded perm algebra can be naturally endowed with a completed infinitesimal bialgebra.
The converse of this result also holds when the quadratic $\bz$-graded perm algebra
is special, giving the desired characterization of the Zinbiel bialgebra
that its affinization is a completed infinitesimal bialgebra.

\subsection{Affinization of Zinbiel bialgebras}\label{subsec:affZbia}
To provide the affinization characterization of Zinbiel algebra, we first recall
the notion of $\bz$-graded perm algebra.

\begin{defi}\label{def:zgrad-alg}
A {\bf $\bz$-graded commutative associative algebra} (resp. {\bf $\bz$-graded perm algebra})
is a commutative associative algebra $(A, \ast)$ (resp. a perm algebra $(B, \diamond)$) with a
linear decomposition $A=\oplus_{i\in\bz}A_{i}$ (resp. $B=\oplus_{i\in\bz}B_{i}$) such that
each $A_{i}$ (resp. $B_{i}$) is finite-dimensional and $A_{i}\ast A_{j}\subseteq A_{i+j}$
(resp. $B_{i}\diamond B_{j}\subseteq B_{i+j}$) for all $i, j\in\bz$.
\end{defi}

\begin{ex}[\cite{LZB}]\label{ex:grperm}
Let $B=\{f_{1}\partial_{1}+f_{2}\partial_{2}\mid f_{1}, f_{2}\in\bk[x_{1}^{\pm},
x_{2}^{\pm}]\}$ and define a binary operation $\diamond: B\otimes B\rightarrow B$ by
\begin{align*}
(x_{1}^{i_{1}}x_{2}^{i_{2}}\partial_{s})\diamond(x_{1}^{j_{1}}x_{2}^{j_{2}}\partial_{t})
:=\delta_{s,1}x_{1}^{i_{1}+j_{1}+1}x_{2}^{i_{2}+j_{2}}\partial_{t}
+\delta_{s,2}x_{1}^{i_{1}+j_{1}}x_{2}^{i_{2}+j_{2}+1}\partial_{t},
\end{align*}
for any $i_{1}, i_{2}, j_{1}, j_{2}\in\bz$ and $t\in\{1, 2\}$.
Then $(B, \diamond)$ is a $\bz$-graded perm algebra with the linear decomposition
$B=\oplus_{i\in\bz}B_{i}$, where
$$
B_{i}=\Big\{\sum_{k=1}^{2}f_{k}\partial_{k}\mid f_{k}\text{ is a homogeneous
polynomial with }\deg(f_{k})=i-1,\; k=1,2\Big\}
$$
for all $i\in\bz$.
\end{ex}

By Proposition \ref{pro:tensor-ass}, we have the following proposition.

\begin{pro}\label{pro:aff-Zalg}
Let $(A, \ast)$ be a finite-dimensional Zinbiel algebra and $(B=\oplus_{i\in\bz}B_{i},
\diamond)$ be a $\bz$-graded perm algebra. If we define a binary operation $\cdot$ on
$A\otimes B$ by Eq. \eqref{ind-ass}, then $(A\otimes B, \cdot)$ is a $\bz$-graded
commutative associative algebra.
Moreover, if $(B=\oplus_{i\in\bz}B_{i}, \diamond)$ is the $\bz$-graded perm algebra given in
Example \ref{ex:grperm}, then $(A\otimes B, \cdot)$ is a $\bz$-graded commutative
associative algebra if and only if $(A, \ast)$ is a Zinbiel algebra.
\end{pro}

\begin{proof}
By Proposition \ref{pro:tensor-ass}, $(A\otimes B, \cdot)$ is a commutative associative
algebra. Since $(B=\oplus_{i\in\bz}B_{i}, \diamond)$ is $\bz$-graded, $(A\otimes B, \cdot)$
is a $\bz$-graded commutative associative algebra. If $(B=\oplus_{i\in\bz}B_{i}, \diamond)$
is the $\bz$-graded perm algebra given in Example \ref{ex:grperm}, then we have
\begin{align*}
\Big((a_{1}\otimes\partial_{1})(a_{2}\otimes\partial_{2})\Big)
(a_{3}\otimes\partial_{1})&=\Big((a_{1}\ast a_{2})\otimes x_{1}\partial_{2})
+(a_{2}\ast a_{1})\otimes x_{2}\partial_{1}\Big)(a_{3}\otimes\partial_{1})\\
&=((a_{1}\ast a_{2})\ast a_{3})\otimes x_{1}x_{2}\partial_{1}
+(a_{3}\ast(a_{1}\ast a_{2}))\otimes x_{1}^{2}\partial_{2}\\[-1mm]
&\quad+((a_{2}\ast a_{1})\ast a_{3})\otimes x_{1}x_{2}\partial_{1}
+(a_{3}\ast(a_{2}\ast a_{1}))\otimes x_{1}x_{2}\partial_{1},
\end{align*}
and
\begin{align*}
(a_{1}\otimes\partial_{1})\Big((a_{2}\otimes\partial_{2})
(a_{3}\otimes\partial_{1})\Big)&=(a_{1}\otimes\partial_{1})
\Big((a_{2}\ast a_{3})\otimes x_{2}\partial_{1}
+(a_{2}\ast a_{3})\otimes x_{1}\partial_{2}\Big)\\
&=(a_{1}\ast(a_{2}\ast a_{3}))\otimes x_{1}x_{2}\partial_{1}
+((a_{2}\ast a_{3})\ast a_{1})\otimes x_{1}x_{2}\partial_{1}\\[-1mm]
&\quad+(a_{1}\ast(a_{3}\ast a_{2}))\otimes x_{1}^{2}\partial_{2}
+((a_{3}\ast a_{2})\ast a_{1})\otimes x_{1}x_{2}\partial_{1}.
\end{align*}
Comparing the coefficients of $x_{1}^{2}\partial_{2}$, we get $a_{3}\ast(a_{1}\ast a_{2})
=a_{1}\ast(a_{3}\ast a_{2})$. Similarly, comparing the coefficients of
$x_{1}^{2}\partial_{2}$ in equations $((a_{1}\otimes\partial_{1})
(a_{2}\otimes\partial_{1}))(a_{3}\otimes\partial_{2})=(a_{1}\otimes\partial_{1})
((a_{2}\otimes\partial_{1})(a_{3}\otimes\partial_{2}))$, we get $a_{3}\ast(a_{2}\ast a_{1})
=(a_{3}\ast a_{2})\ast a_{1}+(a_{2}\ast a_{3})\ast a_{1}$.
Thus, $(A, \ast)$ is a Zinbiel algebra.
\end{proof}

To carry out the Zinbiel coalgebra affinization, we need to extend the codomain
of the comultiplication $\vartheta$ to allow infinite sums.
Let $U=\oplus_{i\in\bz}U_{i}$ and $V=\oplus_{j\in\bz}V_{j}$ be $\bz$-graded vector spaces.
We call the {\bf completed tensor product} of $U$ and $V$ to be the vector space
$$
U\,\hat{\otimes}\,V:=\prod_{i,j\in\bz}U_{i}\otimes V_{j}.
$$
If $U$ and $V$ are finite-dimensional, then $U\,\hat{\otimes}\,V$ is just the usual
tensor product $U\otimes V$. In general, an element in $U\,\hat{\otimes}\,V$ is an
infinite formal sum $\sum_{i,j\in\bz}X_{ij} $ with $X_{ij}\in U_{i}\otimes V_{j}$.
So $X_{ij}=\sum_{\alpha} u_{i, \alpha}\otimes v_{j, \alpha}$ for pure tensors
$u_{i, \alpha}\otimes v_{j, \alpha}\in U_{i}\otimes V_{j}$ with $\alpha$ in a finite
index set. Thus a general term of $U\,\hat{\otimes}\,V$ is a possibly infinite sum
$\sum_{i,j,\alpha}u_{i\alpha}\otimes v_{j\alpha}$, where $i, j\in\bz$ and $\alpha$ is
in a finite index set (which might depend on $i, j$). With these notations, for linear
maps $f: U\rightarrow U'$ and $g: V\rightarrow V'$, define
$$
f\,\hat{\otimes}\,g: U\,\hat{\otimes}\,V\rightarrow U'\,\hat{\otimes}\,V',
\qquad \sum_{i,j,\alpha}u_{i,\alpha}\otimes v_{j, \alpha}\mapsto
\sum_{i,j,\alpha} f(u_{i, \alpha})\otimes g(v_{j, \alpha}).
$$
Also the twist map $\tau$ has its completion
$\hat{\tau}: V\,\hat{\otimes}\,V\rightarrow V\,\hat{\otimes}\,V$,
$\sum_{i,j,\alpha}u_{i, \alpha}\otimes v_{j, \alpha}\mapsto
\sum_{i,j,\alpha}v_{j, \alpha}\otimes u_{i, \alpha}$.
Moreover, we define a (completed) comultiplication to be a linear map
$\vartheta: V\rightarrow V\,\hat{\otimes}\,V$,
$\vartheta(v):=\sum_{i, j, \alpha}v_{1, i, \alpha}\otimes v_{2, j, \alpha}$.
Then we have the well-defined map
$$
(\vartheta\,\hat{\otimes}\,\id)\vartheta(v)=(\vartheta\,\hat{\otimes}\,\id)
\Big(\sum_{i,j,\alpha}v_{1, i, \alpha}\otimes v_{2, j, \alpha}\Big)
:=\sum_{i,j,\alpha}\vartheta(v_{1, i, \alpha})\otimes v_{2, j, \alpha}
\in V\,\hat{\otimes}\,V\,\hat{\otimes}\,V.
$$

\begin{defi}\label{def:ali-coa}
\begin{enumerate}
\item[$(i)$] A {\bf completed cocommutative coassociative coalgebra} is a pair
    $(A, \Delta)$ where $A=\oplus_{i\in\bz}A_{i}$ is a $\bz$-graded vector space and
    $\Delta: A\rightarrow A\,\hat{\otimes}\, A$ is a linear map satisfying
    $$
    \hat{\tau}\Delta=\Delta,\qquad\qquad (\Delta\,\hat{\otimes}\,\id)\Delta
    =(\id\,\hat{\otimes}\,\Delta)\Delta.
    $$
\item[$(ii)$] A {\bf completed perm coalgebra} is a pair $(B, \nu)$, where
    $B=\oplus_{i\in\bz}B_{i}$ is a $\bz$-graded vector space and $\nu: B\rightarrow
    B\,\hat{\otimes}\,B$ is a linear map satisfying
    $$
    (\nu\,\hat{\otimes}\,\id)\nu=(\id\,\hat{\otimes}\,\nu)\nu
    =(\tau\,\hat{\otimes}\,\id)(\nu\,\hat{\otimes}\,\id)\nu.
    $$
\end{enumerate}
\end{defi}

\begin{ex}[\cite{LZB}]\label{ex:grpermco}
Consider the $\bz$-graded vector space $B=\{f_{1}\partial_{1}+f_{2}\partial_{2}\mid
f_{1}, f_{2}\in\bk[x_{1}^{\pm}, x_{2}^{\pm}]\}=\oplus_{i\in\bz}B_{i}$ given in
Example \ref{ex:grperm}. Define a linear map $\nu: B\rightarrow B\,\hat{\otimes}\,B$ by
\begin{align*}
\nu(x_{1}^{m}x_{2}^{n}\partial_{1})&=\sum_{i_{1}, i_{2}\in\bz}
\Big(x_{1}^{i_{1}}x_{2}^{i_{2}}\partial_{1}\otimes x_{1}^{m-i_{1}}
x_{2}^{n-i_{2}+1}\partial_{1} -x_{1}^{i_{1}}x_{2}^{i_{2}}\partial_{2}
\otimes x_{1}^{m-i_{1}+1}x_{2}^{n-i_{2}}\partial_{1}\Big), \\[-2mm]
\nu(x_{1}^{m}x_{2}^{n}\partial_{2})&=\sum_{i_{1}, i_{2}\in\bz}
\Big(x_{1}^{i_{1}}x_{2}^{i_{2}}\partial_{1}\otimes x_{1}^{m-i_{1}}
x_{2}^{n-i_{2}+1}\partial_{2}-x_{1}^{i_{1}}x_{2}^{i_{2}}\partial_{2}
\otimes x_{1}^{m-i_{1}+1}x_{2}^{n-i_{2}}\partial_{2}\Big),
\end{align*}
for any $m, n \in\bz$. Then $(B=\oplus_{i\in\bz}B_{i}, \nu)$ is a completed perm coalgebra.
\end{ex}

Now, we consider the dual version of the Zinbiel algebra affinization,
for Zinbiel coalgebras. We give the dual version of Proposition \ref{pro:aff-Zalg}.

\begin{pro}\label{pro:Zco-coass}
Let $(A, \vartheta)$ be a finite-dimensional Zinbiel coalgebra
and $(B=\oplus_{i\in\bz}B_{i}, \nu)$ be a completed perm coalgebra.
Define a linear map $\Delta: A\otimes B\rightarrow(A\otimes B)\,\hat{\otimes}\,(A\otimes B)$ by
\begin{align}
\Delta(a\otimes b)
&=(\id\otimes\id+\hat{\tau})\big(\vartheta(a)\bullet\nu(b)\big)     \label{Zcoass} \\
&:=\sum_{(a)}\sum_{i,j,\alpha}\Big((a_{(1)}\otimes b_{1,i,\alpha})
\otimes(a_{(2)}\otimes b_{2,j,\alpha})+(a_{(2)}\otimes b_{2,j,\alpha})
\otimes(a_{(1)}\otimes b_{1,i,\alpha})\Big),   \nonumber
\end{align}
for any $a\in A$ and $b\in B$, where $\vartheta(a)=\sum_{(a)}a_{(1)}\otimes a_{(2)}$
in the Sweedler notation and $\nu(b)=\sum_{i,j,\alpha}b_{1,i,\alpha}\otimes b_{2,j,\alpha}$.
Then $(A\otimes B, \Delta)$ is a completed cocommutative coassociative coalgebra.
Moreover, if $(B=\oplus_{i\in\bz}B_{i}, \nu)$ is the completed perm coalgebra given in
Example \ref{ex:grpermco}, then $(A\otimes B, \Delta)$ is a completed cocommutative
coassociative coalgebra if and only if $(A, \vartheta)$ is a Zinbiel coalgebra.
\end{pro}

\begin{proof}
For any $\sum_{l}a'_{l}\otimes a''_{l}\otimes a'''_{l}\in A\otimes A\otimes A$ and
$\sum_{i,j,k,\alpha}b'_{i,\alpha}\otimes b''_{j,\alpha}
\otimes b'''_{k,\alpha}\in B\,\hat{\otimes}\,B\,\hat{\otimes}\,B$, we denote
$$
\Big(\sum_{l}a'_{l}\otimes a''_{l}\otimes a'''_{l}\Big)\bullet
\Big(\sum_{i,j,k,\alpha}b'_{i,\alpha}\otimes b''_{j,\alpha}\otimes b'''_{k,\alpha}\Big)
=\sum_{l}\sum_{i,j,k,\alpha}(a'_{l}\otimes b'_{i,\alpha})\otimes
(a''_{l}\otimes b''_{j,\alpha})\otimes(a'''_{l}\otimes b'''_{k,\alpha}).
$$
Then, by using the above notations, since $(B, \nu)$ is a perm coalgebra,
$(A, \vartheta)$ is a Zinbiel coalgebra and
$(\id\,\hat{\otimes}\,\hat{\tau})(\hat{\tau}\,\hat{\otimes}\,\id)
(\id\,\hat{\otimes}\,\hat{\tau})=(\hat{\tau}\,\hat{\otimes}\,\id)
(\id\,\hat{\otimes}\,\hat{\tau})(\hat{\tau}\,\hat{\otimes}\,\id)$,
for any $a\otimes b\in A\otimes B$, we have
\begin{align*}
&\;(\id\,\hat{\otimes}\,\Delta)(\Delta(a\otimes b))\\
=&\;(\id\otimes\vartheta)(\vartheta(a))\bullet(\id\,\hat{\otimes}\,\nu)(\nu(b))
+(\id\otimes\tau)((\id\otimes\vartheta)(\vartheta(a)))\bullet
(\id\,\hat{\otimes}\,\hat{\tau})((\id\,\hat{\otimes}\,\nu)(\nu(b)))\qquad\quad
\end{align*}
\begin{align*}
&\ \ +(\tau\otimes\id)((\id\otimes\tau)((\vartheta\otimes\id)(\vartheta(a))))\bullet
(\hat{\tau}\,\hat{\otimes}\,\id)((\id\,\hat{\otimes}\,\hat{\tau})((\id\,\hat{\otimes}\,\nu)
(\nu(b))))\\[-1mm]
&\ \ +(\id\otimes\tau)((\tau\otimes\id)((\id\otimes\tau)((\vartheta\otimes\id)(\vartheta(a)))))
\bullet(\id\,\hat{\otimes}\,\hat{\tau})((\hat{\tau}\,\hat{\otimes}\,\id)
((\id\,\hat{\otimes}\,\hat{\tau})((\id\,\hat{\otimes}\,\nu)(\nu(b)))))\\
=&\;(\vartheta\otimes\id)(\vartheta(a))\bullet(\nu\,\hat{\otimes}\,\id)(\nu(b))
+(\tau\otimes\id)((\vartheta\otimes\id)(\vartheta(a)))
\bullet(\id\,\hat{\otimes}\,\nu)(\nu(b))\\[-1mm]
&\ \ +(\id\otimes\tau)((\tau\otimes\id)((\id\otimes\vartheta)(\vartheta(a))))\bullet
(\id\,\hat{\otimes}\,\hat{\tau})((\hat{\tau}\,\hat{\otimes}\,\id)
((\id\,\hat{\otimes}\,\nu)(\nu(b))))\\[-1mm]
&\ \ +(\tau\otimes\id)((\id\otimes\tau)((\vartheta\otimes\id)(\vartheta(a))))\bullet
(\hat{\tau}\,\hat{\otimes}\,\id)((\id\,\hat{\otimes}\,\hat{\tau})((\id\,\hat{\otimes}\,\nu)
(\nu(b))))\\[-1mm]
&\ \ +(\tau\otimes\id)((\id\otimes\tau)((\id\otimes\vartheta)(\vartheta(a))))
\bullet(\hat{\tau}\,\hat{\otimes}\,\id)((\id\,\hat{\otimes}\,\hat{\tau})
((\nu\,\hat{\otimes}\,\id)(\nu(b))))\\[-1mm]
&\ \ -(\tau\otimes\id)((\id\otimes\tau)((\vartheta\otimes\id)(\vartheta(a))))
\bullet(\hat{\tau}\,\hat{\otimes}\,\id)((\id\,\hat{\otimes}\,\hat{\tau})
((\nu\,\hat{\otimes}\,\id)(\nu(b))))\\
=&\;(\vartheta\otimes\id)(\vartheta(a))\bullet(\nu\,\hat{\otimes}\,\id)(\nu(b))
+(\tau\otimes\id)((\vartheta\otimes\id)(\vartheta(a)))\bullet
(\hat{\tau}\,\hat{\otimes}\,\id)((\nu\,\hat{\otimes}\,\id)(\nu(b)))\\[-1mm]
&\ \ +(\id\otimes\tau)((\tau\otimes\id)((\id\otimes\vartheta)(\vartheta(a))))\bullet
(\id\,\hat{\otimes}\,\hat{\tau})((\hat{\tau}\,\hat{\otimes}\,\id)((\nu\,\hat{\otimes}\,\id)
(\nu(b))))\\[-1mm]
&\ \ +(\tau\otimes\id)((\id\otimes\tau)((\tau\otimes\id)((\id\otimes\vartheta)(\vartheta(a)))))
\bullet(\hat{\tau}\,\hat{\otimes}\,\id)((\id\,\hat{\otimes}\,\hat{\tau})
((\hat{\tau}\,\hat{\otimes}\,\id)((\nu\,\hat{\otimes}\,\id)(\nu(b)))))\\
=&\;(\Delta\,\hat{\otimes}\,\id)(\Delta(a\otimes b)).
\end{align*}
Moreover, it is easy to see that $\Delta=\hat{\tau}\Delta$.
Thus, $(A\otimes B, \Delta)$ is a completed cocommutative coassociative coalgebra.

Conversely, if $(B=\oplus_{i\in\bz}B_{i}, \nu)$ is the $\bz$-graded perm algebra
given in Example \ref{ex:grpermco}, then we have
\begin{align*}
&\;(\id\,\hat{\otimes}\,\nu)(\nu(\partial_{1}))\\
=&\;\sum_{i_{1}, i_{2}\in\bz}\sum_{j_{1}, j_{2}\in\bz}
\Big(x_{1}^{i_{1}}x_{2}^{i_{2}}\partial_{1}
\otimes\Big(x_{1}^{j_{1}}x_{2}^{j_{2}}\partial_{1}\otimes
x_{1}^{-i_{1}-j_{1}}x_{2}^{1-i_{2}-j_{2}+1}\partial_{1}
-x_{1}^{j_{1}}x_{2}^{j_{2}}\partial_{2}\otimes
x_{1}^{-i_{1}-j_{1}+1}x_{2}^{1-i_{2}-j_{2}}\partial_{1}\Big)\\[-4mm]
&\qquad\qquad\quad-x_{1}^{i_{1}}x_{2}^{i_{2}}\partial_{2}\otimes\Big(
x_{1}^{j_{1}}x_{2}^{j_{2}}\partial_{1}\otimes
x_{1}^{1-i_{1}-j_{1}}x_{2}^{-i_{2}-j_{2}+1}\partial_{1}
-x_{1}^{j_{1}}x_{2}^{j_{2}}\partial_{2}\otimes
x_{1}^{1-i_{1}-j_{1}+1}x_{2}^{-i_{2}-j_{2}}\partial_{1}\Big)\Big)
\end{align*}
and
\begin{align*}
&\;(\nu\,\hat{\otimes}\,\id)(\nu(\partial_{1}))\\
=&\;\sum_{i_{1}, i_{2}\in\bz}\sum_{j_{1}, j_{2}\in\bz}
\Big(\Big(x_{1}^{j_{1}}x_{2}^{j_{2}}\partial_{1}\otimes
x_{1}^{i_{1}-j_{1}}x_{2}^{i_{2}-j_{2}+1}\partial_{1}
-x_{1}^{j_{1}}x_{2}^{j_{2}}\partial_{2}\otimes
x_{1}^{i_{1}-j_{1}+1}x_{2}^{i_{2}-j_{2}}\partial_{1}\Big)
\otimes x_{1}^{-i_{1}}x_{2}^{1-i_{2}}\partial_{1}\\[-4mm]
&\qquad\qquad\quad-\Big(x_{1}^{j_{1}}x_{2}^{j_{2}}\partial_{1}\otimes
x_{1}^{i_{1}-j_{1}}x_{2}^{i_{2}-j_{2}+1}\partial_{2}
-x_{1}^{j_{1}}x_{2}^{j_{2}}\partial_{2}\otimes
x_{1}^{i_{1}-j_{1}+1}x_{2}^{i_{2}-j_{2}}\partial_{2}\Big)
\otimes x_{1}^{1-i_{1}}x_{2}^{-i_{2}}\partial_{1}\Big).
\end{align*}
Comparing the coefficients of $x_{1}^{i_{1}}x_{2}^{i_{2}}\partial_{1}
\otimes x_{1}^{j_{1}}x_{2}^{j_{2}}\partial_{2}\otimes
x_{1}^{1-i_{1}-j_{1}}x_{2}^{1-i_{2}-j_{2}}\partial_{1}$
in equation $(\id\,\hat{\otimes}\,\Delta)(\Delta(a\otimes\partial_{1}))
=(\Delta\,\hat{\otimes}\,\id)(\Delta(a\otimes\partial_{1}))$, we get
$(\id\otimes\vartheta)\vartheta=(\vartheta\otimes\id)\vartheta+(\tau\otimes\id)
(\vartheta\otimes\id)\vartheta$. Thus, $(A, \vartheta)$ is a Zinbiel coalgebra.
\end{proof}

Next, we extend the Zinbiel algebra affinization and the Zinbiel coalgebra
affinization to Zinbiel bialgebras. First, we recall the definition of quadratic
$\bz$-graded perm algebra.

\begin{defi}\label{def:quad}
Let $\omega(-, -)$ be a bilinear form on a $\bz$-graded perm algebra
$(B=\oplus_{i\in\bz}B_{i}, \diamond)$.
\begin{enumerate}\itemsep=0pt
\item[$(i)$] $\omega(-, -)$ is called {\bf invariant}, if $\omega(b_{1}\diamond b_{2},\;
     b_{3})=\omega(b_{1},\; b_{2}\diamond b_{3}-b_{3}\diamond b_{2})$ for any $b_{1},
     b_{2}, b_{3}\in B$;
\item[$(ii)$] $\omega(-, -)$ is called {\bf graded}, if there exists some $m\in\bz$
        such that $\omega(B_{i}, B_{j})=0$ when $i+j+m\neq0$.
\end{enumerate}
A {\bf quadratic $\bz$-graded perm algebra}, denoted by $(B=\oplus_{i\in\bz}B_{i},
\diamond, \omega)$, is a $\bz$-graded perm algebra together with an antisymmetric
invariant nondegenerate graded bilinear form.
In particular, if $B=B_{0}$, it is just the quadratic perm algebra.
\end{defi}

\begin{ex}[\cite{LZB}]\label{ex:qu-perm}
Let $(B=\oplus_{i\in\bz}B_{i}, \diamond)$ be the $\bz$-graded perm algebra given in
Example \ref{ex:grperm}, where $B=\{f_{1}\partial_{1}+f_{2}\partial_{2}\mid f_{1},
f_{2}\in\bk[x_{1}^{\pm}, x_{2}^{\pm}]\}=\oplus_{i\in\bz}B_{i}$. Define an
antisymmetric bilinear form $\omega(-,-)$ on $(B=\oplus_{i\in\bz}B_{i}, \cdot)$ by
\begin{align*}
&\omega(x_{1}^{i_{1}}x_{2}^{i_{2}}\partial_{2},\ \ x_{1}^{j_{1}}x_{2}^{j_{2}}\partial_{1})
=-\omega(x_{1}^{j_{1}}x_{2}^{j_{2}}\partial_{1},\ \ x_{1}^{i_{1}}x_{2}^{i_{2}}\partial_{2})
=\delta_{i_{1}+j_{1}, 0}\delta_{i_{2}+j_{2}, 0}, \\
&\qquad\quad\omega(x_{1}^{i_{1}}x_{2}^{i_{2}}\partial_{1},\ \
x_{1}^{j_{1}}x_{2}^{j_{2}}\partial_{1})=\omega(x_{1}^{i_{1}}x_{2}^{i_{2}}\partial_{2},\ \
x_{1}^{j_{1}}x_{2}^{j_{2}}\partial_{2})=0,
\end{align*}
for any $i_{1}, i_{2}, j_{1}, j_{2}\in\bz$.
Then $(B=\oplus_{i\in\bz}B_{i}, \diamond, \omega)$ is a quadratic $\bz$-graded perm algebra.
Moreover, $\{x_{1}^{-i_{1}}x_{2}^{-i_{2}}\partial_{2}$,\; $-x_{1}^{-i_{1}}
x_{2}^{-i_{2}}\partial_{1}\mid i_{1}, i_{2}\in\bz\}$ is the dual basis of
$\{x_{1}^{i_{1}}x_{2}^{i_{2}}\partial_{1},\; x_{1}^{i_{1}}x_{2}^{i_{2}}\partial_{2}\mid
i_{1}, i_{2}\in\bz\}$ with respect to $\omega(-,-)$, consisting of homogeneous elements.
\end{ex}

For a quadratic $\bz$-graded perm algebra $(B=\oplus_{i\in\bz}B_{i}, \diamond, \omega)$,
we have $\omega(b_{1}\diamond b_{2},\; b_{3})=\omega(b_{2},\; b_{1}\diamond b_{3})$
for any $b_{1}, b_{2}, b_{3}\in B$. Moreover, the nondegenerate skew-symmetric
bilinear form $\omega(-,-)$ induces bilinear forms
$$
(\underbrace{B\,\hat{\otimes}\,B\,\hat{\otimes}\,\cdots\,\hat{\otimes}\,
B}_{n\text{-fold}})\otimes(\underbrace{B\otimes B\otimes\cdots
\otimes B}_{n\text{-fold}})\longrightarrow\bk,
$$
for all $n\geq2$, which are denoted by $\hat{\omega}(-,-)$ and defined by
$$
\hat{\omega}\Big(\sum_{i_{1},\cdots,i_{n},\alpha} x_{1, i_{1}, \alpha}
\otimes\cdots\otimes x_{n, i_{n}, \alpha},\ \ y_{1}\otimes\cdots\otimes y_{n}\Big)
=\sum_{i_{1},\cdots,i_{n},\alpha}\prod_{j=1}^{n}
\omega(x_{j, i_{j}, \alpha},\; y_{j}).
$$
Then, one can check that $\hat{\omega}(-,-)$ is {\bf left nondegenerate}, i.e., if
$$
\hat{\omega}\Big(\sum_{i_{1}, \cdots, i_{n},\alpha}x_{1, i_{1}, \alpha}
\otimes\cdots\otimes x_{n, i_{n}, \alpha},\ \ y_{1}\otimes\cdots\otimes y_{n}\Big)
=\hat{\omega}\Big(\sum_{i_{1},\cdots,i_{n},\alpha} z_{1, i_{1}, \alpha}
\otimes\cdots\otimes z_{n, i_{n}, \alpha},\ \ y_{1}\otimes\cdots\otimes y_{n}\Big),
$$
for all homogeneous elements $y_{1}, y_{2},\cdots, y_{n}\in B$, then
$$
\sum_{i_{1},\cdots,i_{n},\alpha} x_{1, i_{1}, \alpha}
\otimes\cdots\otimes x_{n, i_{n}, \alpha}
=\sum_{i_{1},\cdots,i_{n},\alpha} z_{1, i_{1}, \alpha}
\otimes\cdots\otimes z_{n, i_{n}, \alpha}.
$$

\begin{lem}[\cite{LZB}]\label{lem:comp-dual}
Let $(B=\oplus_{i\in\bz}B_{i}, \diamond, \omega)$ be a quadratic $\bz$-graded perm algebra.
Define a linear map $\nu_{\omega}: B\rightarrow B\otimes B$ by $\hat{\omega}(\nu_{\omega}
(b_{1}),\; b_{2}\otimes b_{3})=-\omega(b_{1},\; b_{2}b_{3})$, for any $b_{1}, b_{2},
b_{3}\in B$. Then $(B, \nu_{\omega})$ is a completed perm coalgebra.
\end{lem}

\begin{ex}[\cite{LZB}]\label{ex:ind-coperm}
Consider the quadratic $\bz$-graded perm algebra $(B=\oplus_{i\in\bz}B_{i}, \diamond,
\omega)$ given in Example \ref{ex:qu-perm}. Then the induced completed perm coalgebra
$(B=\oplus_{i\in\bz}B_{i}, \nu_{\omega})$ is just the completed perm coalgebra
$(B=\oplus_{i\in\bz}B_{i}, \nu)$ given in Example \ref{ex:grpermco}.
\end{ex}

We now give the notion and results on completed infinitesimal bialgebras.

\begin{defi}\label{def:CASIbia}
A {\bf completed infinitesimal bialgebra} is a triple $(A, \cdot, \Delta)$ consisting of a
vector space $A=\oplus_{i\in\bz}A_{i}$, a bilinear map $\cdot: A\otimes A\rightarrow A$ and
a linear map $\Delta: A\rightarrow A\otimes A$ such that
\begin{enumerate}\itemsep=0pt
\item[$(i)$] $(A=\oplus_{i\in\bz}A_{i}, \cdot)$ is a $\bz$-graded commutative
     associative algebra;
\item[$(ii)$] $(A=\oplus_{i\in\bz}A_{i}, \Delta)$ is a completed cocommutative
     coassociative coalgebra;
\item[$(iii)$] for any $a_{1}, a_{2}\in A$,
\begin{align}
\Delta(a_{1}a_{2})=(\fu_{A}(a_{2})\,\hat{\otimes}\,\id)(\Delta(a_{1}))
+(\id\,\hat{\otimes}\,\fu_{A}(a_{1}))(\Delta(a_{2})).             \label{CASI}
\end{align}
\end{enumerate}
\end{defi}

\begin{thm}\label{thm:Z-perm-ass}
Let $(A, \ast, \vartheta)$ be a finite-dimensional Zinbiel bialgebra,
$(B=\oplus_{i\in\bz}B_{i}, \diamond, \omega)$ be a quadratic $\bz$-graded perm
algebra and $(A\otimes B, \cdot)$ be the induced $\bz$-graded commutative
associative algebra from $(A, \ast)$ by $(B=\oplus_{i\in\bz}B_{i}, \diamond)$.
Define a linear map $\Delta: A\otimes B\rightarrow(A\otimes B)\otimes(A\otimes B)$ by
Eq. \eqref{Zcoass} for $\nu=\nu_{\omega}$. Then $(A\otimes B, \cdot, \Delta)$ is a
completed infinitesimal bialgebra, which is called the {\bf completed infinitesimal
bialgebra induced from $(A, \ast, \vartheta)$ by $(B, \diamond, \omega)$}.

Moreover, if $(B=\oplus_{i\in\bz}B_{i}, \diamond, \omega)$ is the quadratic $\bz$-graded
perm algebra given in Example \ref{ex:qu-perm}, then $(A\otimes B, \cdot, \Delta)$ is
a completed infinitesimal bialgebra if and only if $(A, \ast, \vartheta)$ is a
Zinbiel bialgebra.
\end{thm}

\begin{proof}
By Proposition \ref{pro:Zco-coass} and Lemma \ref{lem:comp-dual}, we get
$(A\otimes B, \Delta)$ is a completed cocommutative coassociative coalgebra,
and need to show that Eq. \eqref{CASI} holds.
For any $a, a'\in A$, $b, b'\in B$, we have
\begin{align*}
&\;\Delta((a\otimes b)(a'\otimes b'))
-(\fu_{A\otimes B}(a'\otimes b')\,\hat{\otimes}\,\id)(\Delta(a\otimes b))
-(\id\,\hat{\otimes}\,\fu_{A\otimes B}(a\otimes b))(\Delta(a'\otimes b'))\\
=&\;\vartheta(a\ast a')\bullet\nu_{\omega}(b\diamond b')
+\tau(\vartheta(a\ast a'))\bullet\hat{\tau}(\nu_{\omega}(b\diamond b'))\\[-1mm]
&\;+\vartheta(a'\ast a)\bullet\nu_{\omega}(b'\diamond b)
+\tau(\vartheta(a'\ast a))\bullet\hat{\tau}(\nu_{\omega}(b'\diamond b))\\
&\;-\sum_{(a)}\sum_{i,j,\alpha}\Big(\big((a'\ast a_{(1)})\otimes
(b'\diamond b_{1,i,\alpha})\big)\otimes(a_{(2)}\otimes b_{2,j,\alpha})
+\big((a_{(1)}\ast a')\otimes(b_{1,i,\alpha}\diamond b')\big)
\otimes(a_{(2)}\otimes b_{2,j,\alpha})\Big)
\end{align*}
\begin{align*}
&\;-\sum_{(a)}\sum_{i,j,\alpha}\Big(\big((a'\ast a_{(2)})\otimes
(b'\diamond b_{2,j,\alpha})\big)\otimes(a_{(1)}\otimes b_{1,i,\alpha})
+\big((a_{(2)}\ast a')\otimes(b_{2,j,\alpha}\diamond b')\big)\otimes
(a_{(1)}\otimes b_{1,i,\alpha})\Big)\\[-2mm]
&\;-\sum_{(a')}\sum_{i,j,\alpha}\Big((a'_{(1)}\otimes b'_{1,i,\alpha})\otimes
\big((a\ast a'_{(2)})\otimes(b\diamond b'_{2,j,\alpha})\big)
+(a'_{(1)}\otimes b'_{1,i,\alpha})\otimes\big((a'_{(2)}\ast a)\otimes
(b'_{2,j,\alpha}\diamond b)\big)\Big)\\[-2mm]
&\;-\sum_{(a')}\sum_{i,j,\alpha}\Big((a'_{(2)}\otimes b'_{2,j,\alpha})\otimes
\big((a\ast a'_{(1)})\otimes(b\diamond b'_{1,i,\alpha})\big)
+(a'_{(2)}\otimes b'_{2,j,\alpha})\otimes\big((a'_{(1)}\ast a)
\otimes(b'_{1,i,\alpha}\diamond b)\big)\Big).
\end{align*}
Since $\omega(-,-)$ on $(B, \diamond)$ is invariant, we have
\begin{align*}
\hat{\omega}(\nu_{\omega}(b\diamond b'),\; e\otimes f)
&=-\omega(b\diamond b',\; e\diamond f)=-\omega(b',\; b\diamond(e\diamond f)),\\
\hat{\omega}\Big(\sum_{i,j,\alpha}(b'_{1,i,\alpha}\otimes
(b\diamond b'_{2,j,\alpha})),\; e\otimes f\Big)
&=-\omega(b',\; e\diamond(b\diamond f))=-\omega(b',\; b\diamond(e\diamond f)),
\end{align*}
for any $b, b', e, f\in B$. Since $\hat{\omega}(-,-)$ is left nondegenerate,
we obtain $\nu_{\omega}(b\diamond b')=\sum_{i,j,\alpha}(b'_{1,i,\alpha}\otimes
(b\diamond b'_{2,j,\alpha}))$. Similarly, we also have
$\sum_{i,j,\alpha}((b_{1,i,\alpha}\diamond b')\otimes b_{2,j,\alpha})
=\sum_{i,j,\alpha}(b'_{2,j,\alpha}\otimes(b'_{1,i,\alpha}\diamond b))=0$ and
\begin{align*}
\hat{\tau}(\nu_{\omega}(b'\diamond b))&=\sum_{i,j,\alpha}((b'\diamond b_{2,j,\alpha})
\otimes b_{1,i,\alpha})=\hat{\tau}(\nu_{\omega}(b\diamond b'))+\Phi,\\[-2mm]
\hat{\tau}(\nu_{\omega}(b\diamond b'))&=\sum_{i,j,\alpha}((b_{2,j,\alpha}\diamond b')
\otimes b_{1,i,\alpha})=\sum_{i,j,\alpha}(b'_{2,j,\alpha}\otimes
(b\diamond b'_{1,i,\alpha})),\\[-2mm]
\nu_{\omega}(b'\diamond b)&=\sum_{i,j,\alpha}((b'\diamond b_{1,i,\alpha})
\otimes b_{2,j,\alpha})=\sum_{i,j,\alpha}(b'_{1,i,\alpha}\otimes(b'_{2,j,\alpha}\diamond b)
=\nu_{\omega}(b\diamond b')+\Phi,
\end{align*}
where $\Phi\in B\otimes B$ such that $\hat{\omega}(\Phi,\; e\otimes f)
=\omega(b',\; e\diamond(f\diamond b))$. Thus, since $(A, \ast, \vartheta)$ is a
Zinbiel bialgebra, we obtain
\begin{align*}
&\;\Delta((a\otimes b)(a'\otimes b'))
-(\fu_{A\otimes B}(a'\otimes b')\,\hat{\otimes}\,\id)(\Delta(a\otimes b))
-(\id\,\hat{\otimes}\,\fu_{A\otimes B}(a\otimes b))(\Delta(a'\otimes b'))\\
=&\;\Big(\vartheta(a\ast a'+a'\ast a)-(\fl_{A}(a')\otimes\id)(\vartheta(a))
-(\id\otimes\fl_{A}(a))(\vartheta(a'))-(\id\otimes\fr_{A}(a))(\vartheta(a'))
\Big)\bullet\nu_{\omega}(b\diamond b')\\[-1mm]
&\ \ +\Big(\vartheta(a'\ast a)+\tau(\vartheta(a'\ast a))
-(\id\otimes\fr_{A}(a))(\vartheta(a')-(\fl_{A}(a')\otimes\id)(\vartheta(a))
-\tau((\id\otimes\fl_{A}(a'))(\vartheta(a))\Big)\bullet\Phi\\[-1mm]
&\ \ +\tau\Big(\vartheta(a\ast a'+a'\ast a)-(\id\otimes\fl_{A}(a'))(\vartheta(a))\\[-2mm]
&\qquad\qquad\qquad-(\id\otimes\fr_{A}(a'))(\vartheta(a))
-(\fl_{A}(a)\otimes\id)(\vartheta(a'))\Big)\bullet\hat{\tau}(\nu_{\omega}(b\diamond b'))\\
=&\; 0.
\end{align*}
Thus, $(A\otimes B, \cdot, \Delta)$ is a completed infinitesimal bialgebra.

Conversely, if $(B=\oplus_{i\in\bz}B_{i}, \diamond, \omega)$ is the quadratic $\bz$-graded
perm algebra given in Example \ref{ex:qu-perm} and $(A\otimes B, \cdot, \Delta)$ is
a completed infinitesimal bialgebra, then $(A, \ast)$ is a Zinbiel algebra and
$(A, \vartheta)$ is a Zinbiel coalgebra by Propositions \ref{pro:aff-Zalg} and
\ref{pro:Zco-coass} respectively. Now we only need to prove that Eqs. \eqref{Zbialg1}
and \eqref{Zbialg2} hold. Since $(A\otimes B, \cdot, \Delta)$ is a completed
infinitesimal bialgebra, for any $a, a'\in A$,
\begin{align*}
0=&\;\Delta((a\otimes\partial_{1})\ast(a'\otimes\partial_{1}))
-(\fu_{A\otimes B}(a'\otimes\partial_{1})\,\hat{\otimes}\,\id)(\Delta(a\otimes\partial_{1}))
-(\id\,\hat{\otimes}\,\fu_{A\otimes B}(a\otimes\partial_{1}))(\Delta(a'\otimes\partial_{1}))\\
=&\;\vartheta(a\ast a'+a'\ast a)\bullet\nu_{\omega}(x_{1}\partial_{1})
+\tau(\vartheta(a\ast a'+a'\ast a))\bullet\,\hat{\tau}(\nu_{\omega}(x_{1}\partial_{1}))\\
&\ \ -\sum_{(a)}\sum_{i_{1},i_{2}}\Big(
\Big((a_{(1)}\ast a')\otimes(x_{1}^{i_{1}+1}x_{2}^{i_{2}}\partial_{1})
+(a\ast a_{(1)})\otimes(x_{1}^{i_{1}+1}x_{2}^{i_{2}}\partial_{1})\Big)
\otimes(a_{(2)}\otimes(x_{1}^{-i_{1}}x_{2}^{1-i_{2}}\partial_{1}))\\[-5mm]
&\qquad\qquad\quad-\Big((a_{(1)}\ast a')\otimes(x_{1}^{i_{1}}x_{2}^{i_{2}+1}\partial_{1})
+(a\ast a_{(1)})\otimes(x_{1}^{i_{1}+1}x_{2}^{i_{2}}\partial_{2})\Big)
\otimes(a_{(2)}\otimes(x_{1}^{1-i_{1}}x_{2}^{-i_{2}}\partial_{1}))\Big)\\
&\ \ -\sum_{(a)}\sum_{i_{1},i_{2}}\Big(
\Big((a_{(1)}\ast a')\otimes(x_{1}^{i_{1}+1}x_{2}^{i_{2}}\partial_{1})
+(a'\ast a_{(1)})\otimes(x_{1}^{i_{1}+1}x_{2}^{i_{2}}\partial_{1})\Big)
\otimes(a_{(2)}\otimes(x_{1}^{-i_{1}}x_{2}^{1-i_{2}}\partial_{1}))\\[-5mm]
&\qquad\qquad\quad-\Big((a_{(1)}\ast a')\otimes(x_{1}^{i_{1}+1}x_{2}^{i_{2}}\partial_{1})
+(a'\ast a_{(1)})\otimes(x_{1}^{i_{1}+1}x_{2}^{i_{2}}\partial_{1})\Big)
\otimes(a_{(2)}\otimes(x_{1}^{1-i_{1}}x_{2}^{-i_{2}}\partial_{2}))\Big)\\
&\ \ -\sum_{(a')}\sum_{i_{1},i_{2}}\Big(
(a'_{(1)}\otimes(x_{1}^{i_{1}}x_{2}^{i_{2}}\partial_{1}))\otimes
\Big((a\ast a'_{(2)})\otimes(x_{1}^{1-i_{1}}x_{2}^{1-i_{2}}\partial_{1})
+(a'_{(2)}\ast a)\otimes(x_{1}^{1-i_{1}}x_{2}^{1-i_{2}}\partial_{1})\Big)\\[-5mm]
&\qquad\qquad\quad-(a'_{(1)}\otimes(x_{1}^{i_{1}}x_{2}^{i_{2}}\partial_{2}))\otimes
\Big((a\ast a'_{(2)})\otimes(x_{1}^{2-i_{1}}x_{2}^{-i_{2}}\partial_{1})
+(a'_{(2)}\ast a)\otimes(x_{1}^{2-i_{1}}x_{2}^{-i_{2}}\partial_{1})\Big)\Big)
\end{align*}
\begin{align*}
&\ \ -\sum_{(a')}\sum_{i_{1},i_{2}}\Big(
(a'_{(1)}\otimes(x_{1}^{i_{1}}x_{2}^{i_{2}}\partial_{1}))\otimes
\Big((a\ast a'_{(2)})\otimes(x_{1}^{1-i_{1}}x_{2}^{1-i_{2}}\partial_{1})
+(a'_{(2)}\ast a)\otimes(x_{1}^{1-i_{1}}x_{2}^{1-i_{2}}\partial_{1})\Big)\\[-5mm]
&\qquad\qquad\quad-(a'_{(1)}\otimes(x_{1}^{i_{1}}x_{2}^{i_{2}}\partial_{1}))\otimes
\Big((a\succ a'_{(2)})\otimes(x_{1}^{2-i_{1}}x_{2}^{-i_{2}}\partial_{2})
+(a'_{(2)}\ast a)\otimes(x_{1}^{1-i_{1}}x_{2}^{1-i_{2}}\partial_{1})\Big)\Big).
\end{align*}
Comparing the coefficients of $\partial_{2}\otimes\partial_{1}$ in the equation above,
we get Eq. \eqref{Zbialg2} holds. Similarly, comparing the coefficients of
$\partial_{1}\otimes\partial_{1}$ in the equation $\Delta((a\otimes\partial_{1})
\ast(a'\otimes\partial_{2}))-(\fr_{A\otimes B}(a'\otimes\partial_{2})\,\hat{\otimes}\,\id)
(\Delta(a\otimes\partial_{1}))-(\id\,\hat{\otimes}\,\fl_{A\otimes B}(a\otimes\partial_{1}))
(\Delta(a'\otimes\partial_{2}))=0$, we get Eq. \eqref{Zbialg1} holds.
Thus, $(A, \ast, \vartheta)$ is a Zinbiel bialgebra in this case.
\end{proof}

Now let us return to finite-dimensional infinitesimal bialgebras.
For a special case of Theorem \ref{thm:Z-perm-ass}, where $(B=B_{0}, \diamond,
\omega)$ is a finite-dimensional quadratic perm algebra, we have

\begin{cor}\label{cor:indassbia}
Let $(A, \ast, \vartheta)$ be a finite-dimensional Zinbiel bialgebra
and $(B, \diamond, \omega)$ be a quadratic perm algebra, and $(A\otimes B, \cdot)$
be the induced commutative associative algebra from $(A, \ast)$ by $(B, \diamond)$.
Define a linear map $\Delta: A\otimes B\rightarrow(A\otimes B) \otimes(A\otimes B)$ by
\begin{align}
\Delta(a\otimes b)
&=(\id\otimes\id+\tau)\big(\vartheta(a)\bullet\nu_{\omega}(b)\big)     \label{fZcoass} \\
&:=\sum_{(a)}\sum_{(b)}\Big((a_{(1)}\otimes b_{(1)})\otimes(a_{(2)}\otimes b_{(2)})
+(a_{(2)}\otimes b_{(2)})\otimes(a_{(1)}\otimes b_{(1)})\Big),   \nonumber
\end{align}
for any $a\in A$ and $b\in B$, where $\vartheta(a)=\sum_{(a)}a_{(1)}\otimes a_{(2)}$ and
$\nu_{\omega}(b)=\sum_{(b)}b_{(1)}\otimes b_{(2)}$ in the Sweedler notation.
Then $(A\otimes B, \cdot, \Delta)$ is an infinitesimal bialgebra, which is called the
{\bf infinitesimal bialgebra induced from $(A, \ast, \vartheta)$ by $(B, \diamond, \omega)$}.
\end{cor}

\subsection{Infinite-dimensional infinitesimal bialgebra from the Zinbiel
Yang-Baxter equation}\label{subsec:affYBE}
In this subsection, we use symmetric solutions of the $\ZYBE$ to construct skew-symmetric
solutions of the $\AYBE$ in the induced commutative associative algebra. Thus, we
can construct a triangular infinite-dimensional infinitesimal bialgebra by the
tensor product of a triangular Zinbiel bialgebra and a quadratic $\bz$-graded perm algebra.
Moreover, we also present a construction of quasi-Frobenius $\bz$-graded
commutative associative algebras from quasi-Frobenius Zinbiel algebras corresponding to
a class of symmetric solutions of the $\ZYBE$.

Recall that a Zinbiel bialgebra $(A, \ast, \vartheta)$ is called {\bf coboundary}
if there exists $r\in A\otimes A$ such that
\begin{align}
\vartheta(a)=\vartheta_{r}(a):=\big(\id\otimes(\fl_{A}+\fr_{A})(a)
-\fl_{A}(a)\otimes\id\big)(r),                       \label{Zcobo}
\end{align}
for all $a\in A$. Let $(A, \ast)$ be a Zinbiel algebra and $r=\sum_{i}x_{i}\otimes y_{i}
\in A\otimes A$. We say $r$ is a solution of the {\bf Zinbiel Yang-Baxter equation}
(or $\ZYBE$) in $(A, \ast)$ if
\begin{align*}
\mathbf{Z}_{r}:&=r_{13}\ast r_{12}+r_{23}\ast r_{21}+r_{13}\ast r_{23}
+r_{23}\ast r_{13} \\[-1mm]
&\quad-r_{12}\ast r_{23}-r_{23}\ast r_{12}-r_{13}\ast r_{21}-r_{21}\ast r_{13}=0,
\end{align*}
where $r_{13}\ast r_{12}:=\sum_{i,j}(x_{i}\ast x_{j})\otimes y_{j}\otimes y_{i}$,
$r_{23}\ast r_{21}:=\sum_{i,j}y_{j}\otimes(x_{i}\ast x_{j})\otimes y_{i}$,
$r_{13}\ast r_{23}:=\sum_{i,j}x_{i}\otimes x_{j}\otimes(y_{i}\ast y_{j})$,
$r_{21}\ast r_{13}:=\sum_{i,j}(y_{i}\ast x_{j})\otimes x_{i}\otimes y_{j}$,
$r_{23}\ast r_{12}:=\sum_{i,j}x_{j}\otimes(x_{i}\ast y_{j})\otimes y_{i}$,
$r_{23}\ast r_{13}:=\sum_{i,j}x_{j}\otimes x_{i}\otimes(y_{i}\ast y_{j})$,
$r_{13}\ast r_{21}:=\sum_{i,j}(x_{i}\ast y_{j})\otimes x_{j}\otimes y_{i}$,
$r_{12}\ast r_{23}:=\sum_{i,j}x_{i}\otimes(y_{i}\ast x_{j})\otimes y_{j}$.

\begin{pro}[\cite{Wan}]\label{pro:spec-Zbia}
Let $(A, \ast)$ be a Zinbiel algebra, $r\in A\otimes A$ and $\vartheta_{r}: A\rightarrow
A\otimes A$ be a linear map defined by Eq. \eqref{Zcobo}. If $r$ is a symmetric
solution of the $\ZYBE$ in $(A, \ast)$, then $(A, \ast, \vartheta_{r})$ is a
Zinbiel bialgebra, which is called a {\bf triangular Zinbiel bialgebra} associated with $r$.
\end{pro}

Recall that an infinitesimal bialgebra $(A, \cdot, \Delta)$ is called {\bf coboundary} if
there exists an element $r\in A\otimes A$ such that
\begin{align}
\Delta(a)=\Delta_{r}(a):=(\id\otimes\fu_{A}(a)-\fu_{A}(a)\otimes\id)(r), \label{ascobo}
\end{align}
for all $a\in A$. Let $(A, \cdot)$ be a cocommutative coassociative algebra
and $r=\sum_{i}x_{i}\otimes y_{i}\in A\otimes A$. The equation
$$
\mathbf{A}_{r}:=r_{12}r_{13}+r_{12}r_{23}-r_{23}r_{13}=0
$$
is called the {\bf associative Yang-Baxter equation} (or $\AYBE$) in $(A, \cdot)$,
where $r_{12}r_{13}:=\sum_{i,j}(x_{i}x_{j})\otimes y_{i}\otimes y_{j}$,
$r_{12}r_{23}:=\sum_{i,j}x_{i}\otimes(y_{i}x_{j})\otimes y_{j}$ and
$r_{23}r_{13}:=\sum_{i,j}x_{j}\otimes x_{i}\otimes(y_{i}y_{j})$.
For a solution of the $\AYBE$ in commutative associative algebra $(A, \cdot)$,
if $r$ is a skew-symmetric solution of the $\AYBE$ in $(A, \cdot)$,
then $(A, \cdot, \Delta_{r})$ is an infinitesimal bialgebra, which is
called a {\bf triangular infinitesimal bialgebra} associated with $r$ \cite{Bai}.
Let $(A=\oplus_{i\in\bz}A_{i}, \cdot)$ be a $\bz$-graded commutative associative algebra.
Suppose that $r=\sum_{i,j,\alpha}x_{i\alpha}\otimes y_{j\alpha}\in A\,\hat{\otimes}\,A$.
We denote $r_{12}r_{13}:=\sum_{i,j,k,l,\alpha,\beta}(x_{i,\alpha}x_{k,\beta})\otimes
y_{j,\alpha}\otimes y_{l,\beta}$, $r_{12}r_{23}:=\sum_{i,j,k,l,\alpha,\beta}
x_{i,\alpha}\otimes(y_{j,\alpha}x_{k,\beta})\otimes y_{l,\beta}$,
$r_{23}r_{13}:=\sum_{i,j,k,l,\alpha,\beta}x_{k,\alpha}\otimes x_{i,\beta}\otimes
(y_{j,\alpha}y_{l,\beta})$. If $r\in A\,\hat{\otimes}\,A$ satisfies
$\mathbf{A}_{r}=r_{12}r_{13}+r_{12}r_{23}-r_{23}r_{13}=0$
as an element in $A\,\hat{\otimes}\,A\,\hat{\otimes}\,A$, then $r$ is called a
{\bf completed solution} of the $\AYBE$ in $(A=\oplus_{i\in\bz}A_{i}, \cdot)$.
If $A=A_{0}$ is finite-dimensional, a completed solution of the $\AYBE$ in $(A, \cdot)$
is just a solution of the $\AYBE$ in commutative associative algebra
$(A=A_{0}, \cdot)$. The same argument of the proof for \cite[Corollary 2.4.1]{Bai}
extends to the completed case. We obtain the following proposition.

\begin{pro}\label{pro:ass-tri}
Let $(A=\oplus_{i\in\bz}A_{i}, \cdot)$ be a $\bz$-graded commutative associative algebra
and $r\in A\,\hat{\otimes}\,A$ is a completed solution of the $\AYBE$ in
$(A=\oplus_{i\in\bz}A_{i}, \cdot)$. Define a bilinear map $\Delta_{r}: A\rightarrow
A\,\hat{\otimes}\,A$ by
\begin{align}
\Delta_{r}(a):=(\id\,\hat{\otimes}\,\fu_{A}(a)
-\fu_{A}(a)\,\hat{\otimes}\,\id)(r),          \label{cass-cobo}
\end{align}
for any $a\in A$. If $r$ is skew-symmetric, i.e., $r=-\hat{\tau}(r)$, then
$(A, \cdot, \Delta_{r})$ is a completed infinitesimal bialgebra, which is called a
{\bf triangular completed infinitesimal bialgebra} associated with $r$.
\end{pro}

We give the following relation between the solutions of the $\ZYBE$ in a Zinbiel algebra
and the completed solutions of the $\AYBE$ in the induced $\bz$-graded commutative
associative algebra. Let $(B=\oplus_{i\in\bz}B_{i}, \diamond, \omega)$ be a quadratic
$\bz$-graded perm algebra and $\{e_{j}\}_{j\in\Omega}$ be a basis of
$B=\oplus_{i\in\bz}B_{i}$ consisting of homogeneous elements. Since $\omega(-,-)$ is
graded, antisymmetric and nondegenerate, we get a homogeneous dual basis
$\{f_{j}\}_{j\in\Omega}$ of $B=\oplus_{i\in\bz}B_{i}$, which is called the dual basis
of $\{e_{i}\}_{i\in\Omega}$ with respect to $\omega(-,-)$, by $\omega(f_{i}, e_{j})
=\delta_{ij}$, where $\delta_{ij}$ is the Kronecker delta.

\begin{pro}\label{pro:ZYBE-AYBE}
Let $(A, \ast)$ be a Zinbiel algebra and $(B=\oplus_{i\in\bz}B_{i}, \diamond,
\omega)$ be a quadratic $\bz$-graded perm algebra, and $(A\otimes B, \cdot)$ be the
induced $\bz$-graded commutative associative algebra. Suppose that $r=\sum_{i}x_{i}
\otimes y_{i}\in A\otimes A$ is a symmetric solution of the $\ZYBE$ in $(A, \ast)$. Then
\begin{align}
\widehat{r}=\sum_{i}\sum_{j\in\Omega}(x_{i}\otimes e_{j})\otimes(y_{i}\otimes f_{j})
\in(A\otimes B)\,\hat{\otimes}\,(A\otimes B)  \label{r-indass}
\end{align}
is a skew-symmetric completed solution of the $\AYBE$ in $(A\otimes B, \cdot)$,
where $\{e_{j}\}_{j\in\Omega}$ is a homogeneous basis of $B=\oplus_{i\in\bz}B_{i}$
and $\{f_{j}\}_{j\in\Omega}$ is the homogeneous dual basis of $\{e_{j}\}_{j\in\Omega}$
with respect to $\omega(-,-)$.
\end{pro}

\begin{proof}
First, for any $p, q\in\Omega$, we have
\begin{align*}
&\;\widehat{r}_{12}\widehat{r}_{13}+\widehat{r}_{13}\widehat{r}_{23}
-\widehat{r}_{23}\widehat{r}_{12}\\
=&\;\sum_{i,j}\sum_{p,q}\big((x_{i}\ast x_{j})\otimes y_{i}\otimes y_{j}\big)
\bullet\big((e_{p}\diamond e_{q})\otimes f_{p}\otimes f_{q}\big)
+\big((x_{j}\ast x_{i})\otimes y_{i}\otimes y_{j}\big)
\bullet\big((e_{q}\diamond e_{p})\otimes f_{p}\otimes f_{q}\big)\\[-4mm]
&\qquad\quad+\big(x_{i}\otimes x_{j}\otimes(y_{i}\ast y_{j})\big)\bullet
\big(e_{p}\otimes e_{q}\otimes(f_{p}\diamond f_{q})\big)
+\big(x_{i}\otimes x_{j}\otimes(y_{j}\ast y_{i})\big)\bullet
\big(e_{p}\otimes e_{q}\otimes(f_{q}\diamond f_{p})\big)\\[-1mm]
&\qquad\quad-\big(x_{j}\otimes(x_{i}\ast y_{j})\otimes y_{i}\big)
\bullet\big(e_{q}\otimes(e_{p}\diamond f_{q})\otimes f_{p}\big)
-\big(x_{j}\otimes(y_{j}\ast x_{i})\otimes y_{i}\big)
\bullet\big(e_{q}\otimes(f_{q}\diamond e_{p})\otimes f_{p}\big).
\end{align*}
Moreover, for given $e_{s}, e_{u}, e_{v}\in B$, $s, u, v\in\Omega$, we have
\begin{align*}
\hat{\omega}\Big(\sum_{p,q}(e_{p}\diamond e_{q})\otimes f_{p}\otimes f_{q},\ \
e_{s}\otimes e_{u}\otimes e_{v}\Big)&=\omega(e_{u}\diamond e_{v},\; e_{s}),\\[-2mm]
\hat{\omega}\Big(\sum_{p,q}(e_{q}\diamond e_{p})\otimes f_{p}\otimes f_{q},\ \
e_{s}\otimes e_{u}\otimes e_{v}\Big)&=\omega(e_{v}\diamond e_{u},\; e_{s}),\\[-2mm]
\hat{\omega}\Big(\sum_{p,q} e_{q}\otimes(f_{q}\diamond e_{p})\otimes f_{p},\ \
e_{s}\otimes e_{u}\otimes e_{v}\Big)&=\omega(e_{v}\diamond e_{u}-e_{u}\diamond e_{v},\; e_{s}).
\end{align*}
By the left nondegeneracy of $\hat{\omega}(-,-)$, we get
$$
\sum_{p,q}e_{q}\otimes(f_{q}\diamond e_{p})\otimes f_{p}
=\sum_{p,q}\big((e_{q}\diamond e_{p})\otimes f_{p}\otimes f_{q}
-(e_{p}\diamond e_{q})\otimes f_{p}\otimes f_{q}\big).
$$
Similarly, we also have $\sum_{p,q}e_{p}\otimes e_{q}\otimes(f_{p}\diamond f_{q})
=\sum_{p,q}\big((e_{q}\diamond e_{p})\otimes f_{p}\otimes f_{q}
-(e_{p}\diamond e_{q})\otimes f_{p}\otimes f_{q}\big)$,
$\sum_{p,q}e_{q}\otimes(e_{p}\diamond f_{q})\otimes f_{p}
=\sum_{p,q}(e_{q}\diamond e_{p})\otimes f_{p}\otimes f_{q}$ and
$\sum_{p,q}e_{p}\otimes e_{q}\otimes(f_{q}\diamond f_{p})
=-\sum_{p,q}(e_{p}\diamond e_{q})\otimes f_{p}\otimes f_{q}$.
Moreover, since $r$ is symmetric, the $\ZYBE$ can be simplified as
$$
\mathbf{Z}_{r}=\sum_{i,j}\Big(-(x_{i}\ast x_{j})\otimes y_{i}\otimes y_{j}
-x_{i}\otimes(y_{i}\ast x_{j})\otimes y_{j}+x_{i}\otimes x_{j}\otimes(y_{i}\ast y_{j})
+x_{i}\otimes x_{j}\otimes(y_{j}\ast y_{i})\Big).
$$
Thus, we obtain
\begin{align*}
&\;\widehat{r}_{12}\widehat{r}_{13}+\widehat{r}_{13}\widehat{r}_{23}
-\widehat{r}_{23}\widehat{r}_{12}\\
=&\;\sum_{i,j}\sum_{p,q}\Big((x_{i}\ast x_{j})\otimes y_{i}\otimes y_{j}
+x_{i}\otimes(y_{i}\ast x_{j})\otimes y_{j}
-x_{i}\otimes x_{j}\otimes(y_{i}\ast y_{j})\\[-5mm]
&\qquad\qquad\qquad\qquad\qquad\qquad-x_{i}\otimes x_{j}\otimes(y_{j}\ast y_{i})
\Big)\bullet\big((e_{p}\diamond e_{q})\otimes f_{p}\otimes f_{q}\big)\\[-2mm]
&\quad+\Big((x_{j}\ast x_{i})\otimes y_{i}\otimes y_{j}
-x_{i}\otimes(y_{i}\ast x_{j})\otimes y_{j}
-x_{i}\otimes(x_{j}\ast y_{i})\otimes y_{j}\\[-2mm]
&\qquad\qquad\qquad\qquad\qquad\qquad+x_{i}\otimes x_{j}\otimes(y_{i}\ast y_{j})
\Big)\bullet\big((e_{q}\diamond e_{p})\otimes f_{p}\otimes f_{q}\big)\\
=&\; -\mathbf{Z}_{r}\bullet\Big(\sum_{p,q}(e_{p}\diamond e_{q})\otimes f_{p}\otimes f_{q}\Big)
-(\id\otimes\tau)(\mathbf{Z}_{r})\bullet\Big(\sum_{p,q}(e_{q}\diamond e_{p})
\otimes f_{p}\otimes f_{q}\Big).
\end{align*}
Thus, $\widehat{r}$ is a completed solution of the $\AYBE$ in $(A\otimes B, \cdot)$
if $r$ is a solution of the $\ZYBE$ in $(A, \ast)$.
Finally, for any $e_{s}, e_{t}\in B$, $s, t\in\Omega$, we have
$$
\hat{\omega}\Big(\sum_{j\in\Omega}e_{j}\otimes f_{j},\; e_{s}\otimes e_{t}\Big)
=\omega(e_{s}, e_{t})=-\omega(e_{t}, e_{s})
=-\hat{\omega}\Big(\sum_{j}f_{j}\otimes e_{j},\; e_{s}\otimes e_{t}\Big).
$$
The left nondegeneracy of $\hat{\omega}(-,-)$ yields that $\sum_{j}e_{j}\otimes f_{j}
=-\sum_{j}f_{j}\otimes e_{j}$. Since $r$ is symmetric, we get $\widehat{r}$
is skew-symmetric. The proof is finished.
\end{proof}

We give the main conclusion of this subsection.

\begin{thm}\label{thm:indu-triass}
Let $(A, \ast, \vartheta)$ be a Zinbiel bialgebra, $(B=\oplus_{i\in\bz}B_{i}, \diamond,
\omega)$ be a quadratic $\bz$-graded perm algebra and $(A\otimes B, \cdot, \Delta)$ be
the induced completed infinitesimal bialgebra from $(A, \ast, \vartheta)$ by
$(B, \diamond, \omega)$. If $(A, \ast, \vartheta)$ is triangular,
i.e., there exists a symmetric element $r\in A\otimes A$ such that $\vartheta=
\vartheta_{r}$, then $(A\otimes B, \cdot, \Delta)$ is also triangular,
where $\vartheta_{r}$ is given by Eq. \eqref{Zcobo}.
\end{thm}

\begin{proof}
Let $r=\sum_{i}x_{i}\otimes y_{i}$ be a symmetric solution of the $\ZYBE$ in $(A, \ast)$.
By Proposition \ref{pro:ZYBE-AYBE}, we have a skew-symmetric completed solution
$\widehat{r}$ of the $\AYBE$ in $(A\otimes B, \cdot)$. Thus, we get a triangular
completed infinitesimal bialgebra $(A\otimes B, \cdot, \Delta_{\widehat{r}})$, where
$\Delta_{\widehat{r}}$ is given by Eq. \eqref{ascobo} for $\widehat{r}$.
Now we only need to prove that $(A\otimes B, \cdot, \Delta)=(A\otimes B, \cdot,
\Delta_{\widehat{r}})$ as completed infinitesimal bialgebras, that is,
$\Delta=\Delta_{\widehat{r}}$. In fact, for any $a\in A$ and $b\in B$, we have
\begin{align*}
\Delta(a\otimes b)&=\sum_{i}\sum_{l,k,\alpha}\Big(
(x_{i}\otimes(a\ast y_{i}))\bullet(b_{1,i,\alpha}\otimes b_{2,k,\alpha})
+((a\ast y_{i})\otimes x_{i})\bullet(b_{2,k,\alpha}\otimes b_{1,i,\alpha})\\[-5mm]
&\qquad\qquad\quad +(x_{i}\otimes(y_{i}\ast a))\bullet(b_{1,i,\alpha}\otimes b_{2,k,\alpha})
+((y_{i}\ast a)\otimes x_{i})\bullet(b_{2,k,\alpha}\otimes b_{1,i,\alpha})\\[-2mm]
&\qquad\qquad\quad-((a\ast x_{i})\otimes y_{i})\bullet(b_{1,i,\alpha}\otimes b_{2,k,\alpha})
-(y_{i}\otimes(a\ast x_{i}))\bullet(b_{2,k,\alpha}\otimes b_{1,i,\alpha})\Big),
\end{align*}
where $\vartheta_{r}(a)=\big(\id\otimes(\fl_{A}+\fr_{A})(a)-\fl_{A}(a)\otimes\id\big)(r)
=\sum_{i}\big(x_{i}\otimes(a\ast y_{i})+x_{i}\otimes(y_{i}\ast a)
-(a\ast x_{i})\otimes y_{i}\big)$ and $\nu_{\omega}(b)=\sum_{l,k,\alpha}b_{1,l,\alpha}
\otimes b_{2,k,\alpha}$. On the other hand,
\begin{align*}
\Delta_{\widehat{r}}(a\otimes b)&=(\id\otimes\fu_{A\otimes B}(a\otimes b)
-\fu_{A\otimes B}(a\otimes b)\otimes\id)(\widehat{r})\\
&=\sum_{i}\sum_{j\in\Omega}\Big(
(x_{i}\otimes(a\ast y_{i}))\bullet(e_{j}\otimes(b\diamond f_{j}))
+(x_{i}\otimes(y_{i}\ast a))\bullet(e_{j}\otimes(f_{j}\diamond b))\\[-5mm]
&\qquad\qquad-((a\ast x_{i})\otimes y_{i})\bullet((b\diamond e_{j})\otimes f_{j})
-((x_{i}\ast a)\otimes y_{i})\bullet((e_{j}\diamond b)\otimes f_{j})\Big),
\end{align*}
where $\widehat{r}=\sum_{i}\sum_{j\in\Omega}(x_{i}\otimes e_{j})\otimes(y_{i}\otimes f_{j})$,
$\{e_{j}\}_{j\in\Omega}$ is a homogeneous basis of $B=\oplus_{i\in\bz}B_{i}$
and $\{f_{j}\}_{j\in\Omega}$ is the homogeneous dual basis of $\{e_{j}\}_{j\in\Omega}$
with respect to $\omega(-,-)$.

For two homogeneous basis elements $e_{s}, e_{t}\in B$, $s, t\in\Omega$, since
$$
\hat{\omega}\Big(\sum_{l,k,\alpha}b_{1,l,\alpha}\otimes b_{2,k,\alpha},\;
e_{s}\otimes e_{t}\Big)=\omega(b\diamond e_{t}-e_{t}\diamond b,\; e_{s})
=\hat{\omega}\Big(\sum_{j}e_{j}\otimes(f_{j}\diamond b),\; e_{s}\otimes e_{t}\Big),
$$
and $\hat{\omega}(-,-)$ is left nondegenerate, we get $\sum_{l,k,\alpha}b_{1,l,\alpha}
\otimes b_{2,k,\alpha}=\sum_{j}e_{j}\otimes(f_{j}\diamond b)$. Similarly, we also have
$\sum_{l,k,\alpha}b_{2,k,\alpha}\otimes b_{1,l,\alpha}=-\sum_{j}(e_{j}\diamond b)
\otimes f_{j}$ and $\sum_{j}(b\diamond e_{j})\otimes f_{j}=\sum_{j}e_{j}\otimes
(b\diamond f_{j})=\sum_{l,k,\alpha}\big(b_{1,l,\alpha}\otimes b_{2,k,\alpha}
-b_{2,k,\alpha}\otimes b_{1,l,\alpha}\big)$. Thus,
\begin{align*}
\Delta(a\otimes b)-\Delta_{\widehat{r}}(a\otimes b)
=&\;\sum_{i}\sum_{l,k,\alpha}\Big(
(a\ast y_{i})\otimes x_{i}+(y_{i}\ast a)\otimes x_{i}-y_{i}\otimes(a\ast x_{i})\\[-5mm]
&\qquad\qquad +x_{i}\otimes(a\ast y_{i})-(a\ast x_{i})\otimes y_{i}
-(x_{i}\ast a)\otimes y_{i}\Big)\bullet(b_{2,k,\alpha}\otimes b_{1,l,\alpha})\\
=&\; 0,
\end{align*}
if $r$ is symmetric. Therefore, $(A\otimes B, \cdot, \Delta)=(A\otimes B, \cdot,
\Delta_{\widehat{r}})$ as completed infinitesimal bialgebras, and so that
$(A\otimes B, \cdot, \Delta)$ is triangular.
\end{proof}

It is worth noting that this theorem also gives the following commutative diagram:
$$
\xymatrix@C=3cm@R=0.5cm{
\txt{$r$ \\ {\tiny a symmetric solution}\\ {\tiny of the $\ZYBE$ in $(A, \ast)$}}
\ar[d]_{{\rm Pro.}~\ref{pro:ZYBE-AYBE}}\ar[r]^{{\rm Pro.}~\ref{pro:spec-Zbia}} &
\txt{$(A, \ast, \vartheta_{r})$ \\ {\tiny a triangular Zinbiel bialgebra}}
\ar[d]^{{\rm Thm.}~\ref{thm:indu-triass}}_{{\rm Thm.}~\ref{thm:Z-perm-ass}} \\
\txt{$\widehat{r}$ \\ {\tiny a skew-symmetric completed solution}\\ {\tiny of the $\AYBE$
in $(A\otimes B, \cdot)$}}
\ar[r]^{{\rm Pro.}~\ref{pro:ass-tri}\quad} &
\txt{$(A\otimes B, \cdot, \Delta)=(A\otimes B, \cdot, \Delta_{\widehat{r}})$ \\
{\tiny a completed infinitesimal bialgebra}}}
$$

For finite-dimensional algebras, the $\mathcal{O}$-operators are considered to be the
operator form of solutions of the Yang-Baxter equation. We now consider the relationship
between the $\mathcal{O}$-operators of Zinbiel algebra and the $\mathcal{O}$-operators
of the induced commutative associative algebra. Let $(A, \cdot)$ be a commutative
associative algebra and $(V, \mu)$ be a representation of $(A, \cdot)$. Recall that a
linear map $\mathcal{T}: V\rightarrow A$ is called an {\bf $\mathcal{O}$-operator of
$(A, \cdot)$ associated to $(V, \mu)$} if for any $v_{1}, v_{2}\in V$,
$$
\mathcal{T}(v_{1})\mathcal{T}(v_{2})=\mathcal{T}\big(\mu(\mathcal{T}(v_{1}))(v_{2})
+\mu(\mathcal{T}(v_{2}))(v_{1})\big).
$$

\begin{pro}[\cite{Bai}]\label{pro:o-ass}
Let $(A, \cdot)$ be a commutative associative algebra, $r\in A\otimes A$ be
skew-symmetric. Then $r$ is a solution of the $\AYBE$ in $(A, \cdot)$ if and only
if $r^{\sharp}: A^{\ast}\rightarrow A$ is an $\mathcal{O}$-operator of $(A, \cdot)$
associated to the coregular representation $(A^{\ast}, -\fu_{A}^{\ast})$.
\end{pro}

The $\mathcal{O}$-operator of Zinbiel algebras was considered in \cite{Wan}.
Recall that an {\bf $\mathcal{O}$-operator of a Zinbiel algebra $(A, \ast)$
associated to a representation $(V, \kl, \kr)$} is a linear map
$\mathcal{T}: V\rightarrow A$ such that
$$
\mathcal{T}(v_{1})\diamond\mathcal{T}(v_{2})=
\mathcal{T}\big(\kl(\mathcal{T}(v_{1}))(v_{2})+\kr(\mathcal{T}(v_{2}))(v_{1})\big),
$$
for any $v_{1}, v_{2}\in V$.

\begin{pro}[\cite{Wan}]\label{pro:o-Zib}
Let $(A, \ast)$ be a Zinbiel algebra, $r\in A\otimes A$ be symmetric.
Then $r$ is a solution of the $\ZYBE$ in $(A, \ast)$ if and only if $r^{\sharp}:
A^{\ast}\rightarrow A$ is an $\mathcal{O}$-operator of $(A, \ast)$ associated to
the coregular representation $(A^{\ast}, -\fl_{A}^{\ast}-\fr_{A}^{\ast}, \fr_{A}^{\ast})$.
\end{pro}

Let $(B=B_{0}, \diamond, \omega)$ be a (finite-dimensional) quadratic perm algebra,
$\{e_{1}, e_{2},\cdots,e_{n}\}$ be a basis of $B$ and $\{f_{1}, f_{2},\cdots,f_{n}\}$
be the dual basis of $\{e_{1}, e_{2},\cdots,e_{n}\}$ with respect to $\omega(-,-)$.
Suppose $(A, \ast)$ is a Zinbiel algebra. By Proposition \ref{pro:ZYBE-AYBE}, we get
$$
\widehat{r}=\sum_{i}\sum_{j}(x_{i}\otimes e_{j})\otimes(y_{i}\otimes f_{j})
\in(A\otimes B)\otimes(A\otimes B)
$$
is a skew-symmetric solution of the $\AYBE$ in the induced commutative associative
algebra $(A\otimes B, \cdot)$ if $r=\sum_{i}x_{i}\otimes y_{i}$ is a symmetric solution
of the $\ZYBE$ in $(A, \ast)$.

\begin{pro}\label{pro:ind-oper}
With the notions above, the $\mathcal{O}$-operator $\widehat{r}^{\sharp}$
of $(A\otimes B, \cdot)$ associated to representation $((A\otimes B)^{\ast},
-\fu^{\ast}_{A\otimes B})$ is given by
$$
\widehat{r}^{\sharp}=r^{\sharp}\otimes\kappa^{\sharp},
$$
where $\kappa:=\sum_{j}e_{j}\otimes f_{j}$, $\kappa^{\sharp}: B^{\ast}\rightarrow B$
is given by $\langle\eta_{2},\; \kappa^{\sharp}(\eta_{1})\rangle=\langle\eta_{1}
\otimes\eta_{2},\; \kappa\rangle$, for any $\eta_{1}, \eta_{2}\in B^{\ast}$.
\end{pro}

\begin{proof}
By direct calculation, one can check that $\kappa^{\sharp}$ is
a linear isomorphism and $\langle\eta_{2},\; \kappa^{\sharp}(\eta_{1})\rangle
=\langle\eta_{1}\otimes\eta_{2},\; \kappa\rangle
=-\langle\eta_{2}\otimes\eta_{1},\; \kappa\rangle
=-\langle\eta_{1},\; \kappa^{\sharp}(\eta_{2})\rangle$.
Therefore, for any $\xi_{1}, \xi_{2}\in A^{\ast}$ and $\eta_{1}, \eta_{2}\in B^{\ast}$,
\begin{align*}
\langle\xi_{2}\otimes\eta_{2},\; \widehat{r}^{\sharp}(\xi_{1}\otimes\eta_{1})\rangle
&=\sum_{i,j}\langle(\xi_{1}\otimes\eta_{1})\otimes(\xi_{2}\otimes\eta_{2}),\ \
(x_{i}\otimes e_{j})\otimes(y_{i}\otimes f_{j})\rangle\\[-2mm]
&=\Big(\sum_{i}\langle\xi_{1}, x_{i}\rangle\langle\xi_{2}, y_{i}\rangle\Big)
\Big(\sum_{j}\langle\eta_{1}, e_{j}\rangle\langle\eta_{2}, f_{j}\rangle\Big)\\[-2mm]
&=\langle\xi_{2},\; r^{\sharp}(\xi_{1})\rangle
\langle\eta_{2},\; \kappa^{\sharp}(\eta_{1})\rangle\\
&=\langle\xi_{2}\otimes\eta_{2},\; r^{\sharp}(\xi_{1})\otimes
\kappa^{\sharp}(\eta_{1})\rangle.
\end{align*}
That is, $\widehat{r}^{\sharp}=r^{\sharp}\otimes\kappa^{\sharp}$ as linear maps.
\end{proof}

Therefore, by the Proposition above, we also have the following commutative diagram:
$$
\xymatrix@C=3cm@R=0.5cm{
\txt{$r$ \\ {\tiny a symmetric solution} \\ {\tiny of the $\ZYBE$ in $(A, \ast)$}}
\ar[d]_-{{\rm Pro.}~\ref{pro:ZYBE-AYBE}}\ar[r]^-{{\rm Pro.}~\ref{pro:o-Zib}} &
\txt{$r^{\sharp}$\\ {\tiny an $\mathcal{O}$-operator of $(A, \ast)$} \\
{\tiny associated to $(A^{\ast}, -\fl_{A}^{\ast}-\fr_{A}^{\ast}, -\fr_{A}^{\ast})$}}
\ar[d]^-{\mbox{$-\otimes\kappa^{\sharp}$}} \\
\txt{$\widehat{r}$ \\ {\tiny a skew-symmetric solution} \\ {\tiny of the $\CYBE$ in
$(A\otimes B, [-,-])$}} \ar[r]^-{{\rm Pro.}~\ref{pro:o-ass}}
& \txt{$\widehat{r}^{\sharp}=r^{\sharp}\otimes\kappa^{\sharp}$ \\
{\tiny an $\mathcal{O}$-operator of $(A\otimes B, \cdot)$ } \\
{\tiny associated to $((A\otimes B)^{\ast}, -\fu^{\ast}_{A\otimes B})$}}}
$$

\begin{ex}\label{ex:Zbia-Assbia}
Let $(A=\bk\{e_{1}, e_{2}\}, \ast)$ be a two dimensional Zinbiel algebra,
where $e_{1}\ast e_{1}=e_{2}$. Then $r=e_{1}\otimes e_{2}+e_{2}\otimes e_{1}$ is a
symmetric solution of the $\ZYBE$ in $(A, \ast)$. Thus, $r$ induces a triangular
Zinbiel bialgebra structure on $A=\bk\{e_{1}, e_{2}\}$, where the coproduct
$\vartheta$ on $A$ is given by $\vartheta(e_{1})=e_{2}\otimes e_{2}$ and
$\vartheta(e_{2})=0$. Consider the two dimensional quadratic perm algebra
$(B=\bk\{x_{1}, x_{2}\}, \diamond, \omega)$, where the nonzero products are
given by $x_{2}\diamond x_{1}=x_{1}$, $x_{2}\diamond x_{2}=x_{2}$, and bilinear
form $\omega(-,-)$ is given by $\omega(x_{2}, x_{1})=1$. Then we get a perm coalgebra
$(B, \nu)$, where $\nu(x_{1})=x_{1}\otimes x_{1}$ and $\nu(x_{2})=x_{1}\otimes x_{2}$.
By Corollary \ref{cor:indassbia}, we have an infinitesimal bialgebra $(A\otimes B, \cdot,
\Delta)$, where the commutative associative algebra $(A\otimes B, \cdot)$ has given
in Example \ref{ex:preP-P}, the coproduct $\Delta$ is given by
$$
\Delta(y_{1})=2y_{3}\otimes y_{3},\qquad \Delta(y_{2})=y_{3}\otimes y_{4}
+y_{4}\otimes y_{3}, \qquad \Delta(y_{3})=\Delta(y_{4})=0,
$$
$y_{1}:=e_{1}\otimes x_{1}$, $y_{2}:=e_{1}\otimes x_{2}$, $y_{3}:=e_{2}\otimes x_{1}$,
$y_{4}:=e_{2}\otimes x_{2}$.

On the other hand, note that $x_{2}, -x_{1}$ is the dual basis of $x_{1}, x_{2}$ in $B$
with respect to $\omega(-,-)$. We get
$$
\widehat{r}:=(e_{1}\otimes x_{1})\otimes(e_{2}\otimes x_{2})
-(e_{1}\otimes x_{2})\otimes(e_{2}\otimes x_{1})
+(e_{2}\otimes x_{1})\otimes(e_{1}\otimes x_{2})
-(e_{2}\otimes x_{2})\otimes(e_{1}\otimes x_{1})
$$
is a skew-symmetric solution of the $\AYBE$ in $(A\otimes B, \cdot)$.
One can check that the infinitesimal bialgebra structure induced by $\widehat{r}$ on
$A\otimes B$ is just the infinitesimal bialgebra structure given as above.
\end{ex}

\subsection{Affinization of quasi-Frobenius Zinbiel algebras}\label{subsec:affQF}
In this subsection, we consider the affinization of quasi-Frobenius Zinbiel algebras.
We give a construction of quasi-Frobenius $\bz$-graded commutative associative algebras
from quasi-Frobenius Zinbiel algebras corresponding to a class of symmetric
solutions of the $\ZYBE$.

\begin{defi}\label{def:frob-Z}
Let $(A, \ast)$ be a Zinbiel algebra. If there is a symmetric nondegenerate
bilinear form $\varpi(-,-)$ on $A$ satisfying
$$
\varpi(a_{1}\ast a_{2}+a_{2}\ast a_{1},\; a_{3})=\varpi(a_{1},\; a_{2}\ast a_{3})
+\varpi(a_{2},\; a_{1}\ast a_{3}),
$$
for any $a_{1}, a_{2}, a_{3}\in A$, then $(A, \ast, \varpi)$ is called a
{\bf quasi-Frobenius Zinbiel algebra}, and $\varpi(-,-)$ is called a 2-cocycle
of $(A, \ast)$.
\end{defi}

The quasi-Frobenius Zinbiel algebras are not only closely related to the pre-Zinbiel
algebras, but also closely related to symmetric solutions of the $\ZYBE$.

\begin{pro}\label{pro:quasi-Z}
Let $(A, \ast)$ be a Zinbiel algebra with a nondegenerate bilinear form $\varpi(-,-)$,
$\{e_{1}, e_{2}$, $\cdots, e_{n}\}$ be a basis of $A$ and $\{f_{1}, f_{2},\cdots,f_{n}\}$
be the dual basis of $\{e_{1}, e_{2},\cdots,e_{n}\}$ with respect to $\varpi(-,-)$.
Denote $r=\sum_{i}e_{i}\otimes f_{i}\in A\otimes A$.
Then $r$ is a symmetric solution of the $\ZYBE$ in $(A, \ast)$ if and only
if $(A, \ast, \varpi)$ is a quasi-Frobenius Zinbiel algebra.
\end{pro}

\begin{proof}
First, it is easy to see that $r=\sum_{i}e_{i}\otimes f_{i}$ is symmetric if and only if
$\varpi(-,-)$ is symmetric. Define a bilinear map $\varphi: A\otimes A\rightarrow
\End_{\bk}(A)$ by $\varphi(a_{1}\otimes a_{2})(a_{3})=\varpi(a_{3}, a_{2})a_{1}$,
for any $a_{1}, a_{2}, a_{3}\in A$. Then $\varphi$ is a linear isomorphism
if $\varpi(-,-)$ is nondegenerate. Note that the $\ZYBE$ can be simplified as
$$
\mathbf{Z}_{r}:=r_{13}\ast r_{23}+r_{23}\ast r_{13}-r_{12}\ast r_{13}-r_{12}\ast r_{23}
$$
if $r$ is symmetric, and in this case,
\begin{align*}
r_{13}\ast r_{23}&
=\sum_{i,j,k}\varpi(f_{i}\ast f_{j},\; e_{k})e_{i}\otimes e_{j}\otimes f_{k},\qquad
r_{23}\ast r_{13}
=\sum_{i,j,k}\varpi(f_{j}\ast f_{i},\; e_{k})e_{i}\otimes e_{j}\otimes f_{k},\\[-2mm]
r_{12}\ast r_{13}&
=\sum_{i,j,k}\varpi(f_{j}\ast e_{k},\; f_{i})e_{i}\otimes e_{j}\otimes f_{k},\qquad
r_{12}\ast r_{23}
=\sum_{i,j,k}\varpi(f_{i}\ast e_{k},\; f_{j})e_{i}\otimes e_{j}\otimes f_{k}.
\end{align*}
We obtain
\begin{align*}
\mathbf{Z}_{r}&=\sum_{i,j,k}\Big(\varpi(f_{i}\ast f_{j}+f_{i}\ast f_{j},\; e_{k})
-\varpi(f_{j}\ast e_{k},\; f_{i})-\varpi(f_{i}\ast e_{k},\; f_{j})\Big)
e_{i}\otimes e_{j}\otimes f_{k}.
\end{align*}
This means that $r$ is a symmetric solution of the $\ZYBE$ in $(A, \ast)$ if and
only if $(A, \ast, \varpi)$ is a quasi-Frobenius Zinbiel algebra.
\end{proof}

\begin{defi}
Let $(A=\oplus_{i\in\bz}A_{i}, \cdot)$ be a $\bz$-graded commutative associative algebra.
If there is an antisymmetric nondegenerate graded bilinear form $\mathcal{B}(-,-)$ on
$A=\oplus_{i\in\bz}A_{i}$ satisfying
\begin{align*}
\mathcal{B}(a_{1}a_{2},\; a_{3})+\mathcal{B}(a_{2}a_{3},\; a_{1})
+\mathcal{B}(a_{3}a_{1},\; a_{2})=0,
\end{align*}
for any $a_{1}, a_{2}, a_{3}\in A$, then $(A=\oplus_{i\in\bz}A_{i}, \cdot, \mathcal{B})$
is called a {\bf $\bz$-graded commutative associative algebra with Connes cocycle}.
\end{defi}

\begin{pro}\label{pro:Con-ass}
Let $(A, \ast, \varpi)$ be a quasi-Frobenius Zinbiel algebra, $(B=\oplus_{i\in\bz}B_{i},
\diamond, \omega)$ be a quadratic $\bz$-graded perm algebra and $(A\otimes B, \cdot)$
be the induced commutative associative algebra. Define a bilinear form
$\mathcal{B}(-,-)$ on $A\otimes B$ by
\begin{align}
\mathcal{B}(a_{1}\otimes b_{1},\; a_{2}\otimes b_{2})
=\varpi(a_{1}, a_{2})\omega(b_{1}, b_{2}),            \label{bilinear}
\end{align}
for any $a_{1}, a_{2}\in A$ and $b_{1}, b_{2}\in B$. Then $(A\otimes B, \cdot, \mathcal{B})$
is a $\bz$-graded commutative associative algebra with Connes cocycle.

Furthermore, if the quadratic $\bz$-graded perm algebra is given in Example \ref{ex:qu-perm},
then $(A\otimes B, \cdot, \mathcal{B})$ is a $\bz$-graded commutative associative
algebra with Connes cocycle if and only if $(A, \ast, \varpi)$ is a quasi-Frobenius
Zinbiel algebra.
\end{pro}

\begin{proof}
First, since $\varpi(-,-)$ is symmetric nondegenerate and $\omega(-,-)$ is
antisymmetric nondegenerate graded, we obtain that $\mathcal{B}(-,-)$ is antisymmetric
nondegenerate graded. For any $a_{1}, a_{2}, a_{3}\in A$ and $b_{1}, b_{2}, b_{3}\in B$,
we have
\begin{align*}
&\;\mathcal{B}\big((a_{1}\otimes b_{1})(a_{2}\otimes b_{2}),\ \ (a_{3}\otimes b_{3})\big)
+\mathcal{B}\big((a_{2}\otimes b_{2})(a_{3}\otimes b_{3}),\ \ (a_{1}\otimes b_{1})\big)\\[-1mm]
&\qquad+\mathcal{B}\big((a_{3}\otimes b_{3})(a_{1}\otimes b_{1}),\ \
(a_{2}\otimes b_{2})\big)\\
=&\;\varpi(a_{1}\ast a_{2},\; a_{3})\omega(b_{1}\diamond b_{2},\; b_{3})
+\varpi(a_{2}\ast a_{1},\; a_{3})\omega(b_{2}\diamond b_{1},\; b_{3})
+\varpi(a_{2}\ast a_{3},\; a_{1})\omega(b_{2}\diamond b_{3},\; b_{1})\\[-1mm]
&\quad+\varpi(a_{3}\ast a_{2},\; a_{1})\omega(b_{3}\diamond b_{2},\; b_{1})
+\varpi(a_{3}\ast a_{1},\; a_{2})\omega(b_{3}\diamond b_{1},\; b_{2})
+\varpi(a_{1}\ast a_{3},\; a_{2})\omega(b_{1}\diamond b_{3},\; b_{2})\\
=&\;\Big(\varpi(a_{2}\ast a_{3},\; a_{1})-\varpi(a_{1}\ast a_{2},\; a_{3})
+\varpi(a_{1}\ast a_{3},\; a_{2})-\varpi(a_{2}\ast a_{1},\; a_{3})\Big)
\omega(b_{2}\diamond b_{3},\; b_{1})\\[-2mm]
&\quad+\Big(\varpi(a_{3}\ast a_{2},\; a_{1})+\varpi(a_{1}\ast a_{2},\; a_{3})
-\varpi(a_{1}\ast a_{3},\; a_{2})-\varpi(a_{3}\ast a_{1},\; a_{2})\Big)
\omega(b_{3}\diamond b_{2},\; b_{1})\\
=&\; 0.
\end{align*}
Therefore, $(A\otimes B, \cdot, \mathcal{B})$ is a $\bz$-graded commutative associative
algebra with Connes cocycle.

Suppose that the quadratic $\bz$-graded perm algebra $(B=\oplus_{i\in\bz}B_{i},
\diamond, \omega)$ is given in Example \ref{ex:qu-perm} and $(A\otimes B, \cdot,
\mathcal{B})$ is a $\bz$-graded commutative associative algebra with Connes cocycle.
We only need to show that $(A, \ast, \varpi)$ is a quasi-Frobenius Zinbiel algebra
if $(A\otimes B, \cdot, \mathcal{B})$ is a $\bz$-graded commutative associative
algebra with Connes cocycle. First, it is easy to see that $\varpi(-,-)$ is symmetric.
Second, for any $a_{1}, a_{2}, a_{3}\in A$, we have
\begin{align*}
0&=\mathcal{B}\big((a_{1}\otimes \partial_{1})(a_{2}\otimes \partial_{1}),\ \
(a_{3}\otimes x_{1}^{-1}\partial_{2})\big)+\mathcal{B}\big((a_{2}\otimes \partial_{1})
(a_{3}\otimes x_{1}^{-1}\partial_{2}),\ \ (a_{1}\otimes \partial_{1})\big)\\[-1mm]
&\qquad+\mathcal{B}\big((a_{3}\otimes x_{1}^{-1}\partial_{2})
(a_{1}\otimes \partial_{1}),\ \ (a_{2}\otimes \partial_{1})\big)\\
&= -\varpi(a_{1}\ast a_{2}+a_{2}\ast a_{1},\; a_{3})+\varpi(a_{1},\; a_{2}\ast a_{3})
+\varpi(a_{2},\; a_{1}\ast a_{3}).
\end{align*}
Thus, $(A, \ast, \varpi)$ is a quasi-Frobenius Zinbiel algebra in this case.
\end{proof}

Combining Propositions \ref{pro:ZYBE-AYBE}, \ref{pro:quasi-Z} and
\ref{pro:Con-ass}, we obtain the equivalences.

\begin{thm}\label{thm:quasi-ass-eq}
Let $(A, \ast)$ be a Zinbiel algebra with a nondegenerate bilinear form $\varpi(-,-)$,
$\{e_{1}, e_{2}$, $\cdots,e_{n}\}$ be a basis of $A$ and $\{f_{1}, f_{2},\cdots,f_{n}\}$
be the dual basis of $\{e_{1}, e_{2},\cdots,e_{n}\}$ with respect to $\varpi(-,-)$.
Denote $r=\sum_{i}e_{i}\otimes f_{i}\in A\otimes A$ and suppose $(A\otimes B, \cdot)$
is the commutative associative algebra induced by $(A, \ast)$ and the $\bz$-graded
perm algebra $(B, \diamond)$ given in Example \ref{ex:grperm}. Define a bilinear form
$\mathcal{B}(-,-)$ on $A\otimes B$ by Eq. \eqref{bilinear} for $\omega(-,-)$ given in
Example \ref{ex:qu-perm}. Then the following conditions are equivalent.
\begin{enumerate}\itemsep=0pt
\item[$(i)$] $(A, \ast, \varpi)$ is a quasi-Frobenius Zinbiel algebra;
\item[$(ii)$] $r$ is a symmetric solution of the $\ZYBE$ in $(A, \ast)$;
\item[$(iii)$] $\widehat{r}:=\sum_{i}\sum_{i_{1}, i_{2}\in\bz}\big((e_{i}\otimes
  x_{1}^{i_{1}}x_{2}^{i_{2}}\partial_{1})\otimes(f_{i}\otimes x_{1}^{-i_{1}}x_{2}^{-i_{2}}
  \partial_{2})+(e_{i}\otimes x_{1}^{i_{1}}x_{2}^{i_{2}}\partial_{2})\otimes(f_{i}\otimes
  x_{1}^{-i_{1}}x_{2}^{-i_{2}}\partial_{1})\big)\in(A\otimes B)\,\hat{\otimes}\,(A\otimes B)$
  is a skew-symmetric completed solution of the $\AYBE$ in $(A\otimes B, \cdot)$;
\item[$(iv)$] $((A\otimes B, \cdot, \mathcal{B})$ is a $\bz$-graded commutative
  associative algebra with Connes cocycle.
\end{enumerate}
\end{thm}

It is worth remarking that the equivalence of $(iii)$ and $(iv)$ in Theorem
\ref{thm:quasi-ass-eq} provides an equivalence between skew-symmetric solutions of
the $\AYBE$ and commutative associative algebras with a Connes cocycle.

\begin{ex}\label{ex:FZ-Fass}
Consider the two dimensional quasi-Frobenius Zinbiel algebra $(A=\bk\{e_{1}, e_{2}\},
\ast, \varpi)$, where $e_{1}\ast e_{1}=e_{2}$, $\varpi(e_{1}, e_{2})=1$. Then
$e_{2}, e_{1}$ is the dual basis of $e_{1}, e_{2}$ with respect to $\varpi(-,-)$,
and the element $r$ defined in Theorem \ref{thm:quasi-ass-eq} is given by
$r=e_{1}\otimes e_{2}+e_{2}\otimes e_{1}$, which is just the symmetric solution of
the $\ZYBE$ in $(A, \ast)$ considered in Example \ref{ex:Zbia-Assbia}.
Let $(B=\bk\{x_{1}, x_{2}\}, \diamond, \omega)$ be the two dimensional quadratic
perm algebra given in Example \ref{ex:Zbia-Assbia}. Then $\widehat{r}$ defined by
Eq. \eqref{r-indass} is a skew-symmetric solution of the $\AYBE$ in $(A\otimes B, \cdot)$,
and the bilinear form $\mathcal{B}(-,-)$ on $A\otimes B$ is given by
$\mathcal{B}(e_{2}\otimes x_{2},\; e_{1}\otimes x_{1})=1=\mathcal{B}
(e_{1}\otimes x_{2},\; e_{2}\otimes x_{1})$.
\end{ex}

\smallskip
\section{Infinite-dimensional Poisson bialgebra from pre-Poisson bialgebras} \label{sec:pbialg}
In this section, we consider the affinization of pre-Poisson bialgebras.
More exactly, we show that the tensor product of a finite-dimensional pre-Poisson
bialgebra and a quadratic $\bz$-graded perm algebra can be naturally endowed with
a completed Poisson bialgebra. A pre-Poisson algebra is the structure of a pre-Lie
algebra and a Zinbiel algebra on the same vector space. The affinization of pre-Lie
bialgebras has been considered in \cite{LZB}. First, let us review the affinization
of pre-Lie bialgebras. A {\bf $\bz$-graded Lie algebra} is a Lie algebra $(\g, [-,-])$
with a linear decomposition $\g=\oplus_{i\in\bz}\g_{i}$ such that each $\g_{i}$ is
finite-dimensional and $[\g_{i}, \g_{j}]\subseteq\g_{i+j}$ for all $i, j\in\bz$.
A {\bf completed Lie coalgebra} is a pair $(\g, \delta)$ where $\g=\oplus_{i\in\bz}\g_{i}$
is a $\bz$-graded vector space and $\delta: \g\rightarrow\g\,\hat{\otimes}\,\g$ is a
linear map satisfying
$$
\hat{\tau}\delta=-\delta,\qquad\quad (\id\,\hat{\otimes}\,\delta)\delta
-(\hat{\tau}\,\hat{\otimes}\,\id)(\id\,\hat{\otimes}\,\delta)\delta
=(\delta\,\hat{\otimes}\,\id)\delta.
$$

For the affinization of pre-Lie algebras and pre-Lie coalgebras, we have

\begin{pro}[\cite{LZB}]\label{pro:aff-preli}
Let $(A, \circ)$ be a finite-dimensional pre-Lie algebra and $(B=\oplus_{i\in\bz}B_{i},
\diamond)$ be a $\bz$-graded perm algebra. If we define a bracket $[-,-]$ on $A\otimes B$
by Eq. \eqref{ind-lie}, then $(A\otimes B, [-,-])$ is a $\bz$-graded Poisson algebra.
Moreover, if $(B=\oplus_{i\in\bz}B_{i}, \diamond)$ is the $\bz$-graded perm algebra
given in Example \ref{ex:grperm}, then $(A\otimes B, [-,-])$ is a $\bz$-graded Lie
algebra if and only if $(A, \circ)$ is a pre-Lie algebra.

Let $(A, \theta)$ be a finite-dimensional pre-Lie coalgebra and $(B=\oplus_{i\in\bz}B_{i},
\nu)$ be a completed perm coalgebra. If we define a linear map $\delta: A\otimes B
\rightarrow(A\otimes B)\,\hat{\otimes}\,(A\otimes B)$ by
\begin{align}
\delta(a\otimes b)
&=(\id\otimes\id-\tau)(\theta(a)\bullet\nu(b))            \label{PLcoLie} \\
&:=\sum_{(a)}\sum_{i,j,\alpha}\Big((a_{(1)}\otimes b_{1,i,\alpha})
\otimes(a_{(2)}\otimes b_{2,j,\alpha})-(a_{(2)}\otimes b_{2,j,\alpha})
\otimes(a_{(1)}\otimes b_{1,i,\alpha})\Big),   \nonumber
\end{align}
for any $a\in A$ and $b\in B$, where $\theta(a)=\sum_{(a)}a_{(1)}\otimes a_{(2)}$
in the Sweedler notation and $\nu(b)=\sum_{i,j,\alpha}b_{1,i,\alpha}
\otimes b_{2,j,\alpha}$, then $(A\otimes B, \delta)$ is a completed Lie coalgebra.
Moreover, if $(B=\oplus_{i\in\bz}B_{i}, \nu)$ is the completed perm coalgebra given in
Example \ref{ex:grpermco}, then $(A\otimes B, \delta)$ is a completed Lie
coalgebra if and only if $(A, \theta)$ is a pre-Lie coalgebra.
\end{pro}

To provide the affinization of pre-Lie bialgebras, we first introduce the
concept of completed Lie bialgebras.

\begin{defi}\label{def:comp-Libialg}
A {\bf completed Lie bialgebra} is a triple $(\g, [-,-], \delta)$ consisting of a
vector space $\g=\oplus_{i\in\bz}\g_{i}$, a bilinear bracket $[-,-]: \g\otimes\g
\rightarrow\g$ and a linear map $\delta: \g\rightarrow\g\otimes\g$ such that
\begin{enumerate}\itemsep=0pt
\item[$(i)$] $(\g=\oplus_{i\in\bz}\g_{i}, [-,-])$ is a $\bz$-graded Lie algebra;
\item[$(ii)$] $(\g=\oplus_{i\in\bz}\g_{i}, \delta)$ is a completed Lie coalgebra;
\item[$(iii)$] for any $g_{1}, g_{2}\in\g$,
$$
\delta([g_{1}, g_{2}])=(\ad_{\g}(g_{1})\,\hat{\otimes}\,\id
+\id\,\hat{\otimes}\,\ad_{\g}(g_{1}))(\delta(g_{2}))-(\ad_{\g}(g_{2})\,\hat{\otimes}\,\id
+\id\,\hat{\otimes}\,\ad_{\g}(g_{2}))(\delta(g_{1})).
$$
\end{enumerate}
\end{defi}

Then, for the affinization of pre-Lie bialgebras, we have

\begin{thm}[\cite{LZB}]\label{thm:pre-lie}
Let $(A, \circ, \theta)$ be a pre-Lie bialgebra and $(B=\oplus_{i\in\bz}B_{i},
\diamond, \omega)$ be a quadratic $\bz$-graded perm algebra, and $(A\otimes B, [-,-])$
be the induced Lie algebra from $(A, \circ)$ by $(B, \diamond)$. Define a linear
map $\delta: A\otimes B\rightarrow(A\otimes B)\,\hat{\otimes}\,(A\otimes B)$ by
Eq. \eqref{PLcoLie} for $\nu=\nu_{\omega}$. Then $(A\otimes B, [-,-], \delta)$ is
a completed Lie bialgebra, which is called the {\bf completed Lie bialgebra induced
from $(A, \circ, \theta)$ by $(B=\oplus_{i\in\bz}B_{i}, \diamond, \omega)$}.

Moreover, if $(B=\oplus_{i\in\bz}B_{i}, \diamond, \omega)$ is the quadratic $\bz$-graded
perm algebra given in Example \ref{ex:qu-perm}, then $(A\otimes B, [-,-], \delta)$ is
a completed Lie bialgebra if and only if $(A, \circ, \theta)$ is a pre-Lie bialgebra.
\end{thm}

Let $(\g=\oplus_{i\in\bz}\g_{i}, [-,-])$ be a $\bz$-graded Lie algebra.
Suppose that $r=\sum_{i,j,\alpha}x_{i\alpha}\otimes y_{j\alpha}\in\g\,\hat{\otimes}\,\g$.
We denote $[r_{12}, r_{13}]:=\sum_{i,j,k,l,\alpha,\beta}[x_{i,\alpha}, x_{k,\beta}]
\otimes y_{j,\alpha}\otimes y_{l,\beta}$, $[r_{12}, r_{23}]:=\sum_{i,j,k,l,\alpha,\beta}
x_{i,\alpha}\otimes[y_{j,\alpha}, x_{k,\beta}]\otimes y_{l,\beta}$,
$[r_{13}, r_{23}]:=\sum_{i,j,k,l,\alpha,\beta}x_{i,\beta}\otimes x_{k,\alpha}\otimes
[y_{j,\alpha}, y_{l,\beta}]$. The equation
$$
\mathbf{C}_{r}:=[r_{12}, r_{13}]+[r_{13}, r_{23}]+[r_{12}, r_{23}]=0
$$
is called the {\bf classical Yang-Baxter equation} (or $\CYBE$) in
$(\g=\oplus_{i\in\bz}\g_{i}, [-,-])$. If $r\in\g\,\hat{\otimes}\,\g$ satisfies
$\mathbf{C}_{r}=0$ as an element in $\g\,\hat{\otimes}\,\g\,\hat{\otimes}\,\g$,
then $r$ is called a {\bf completed solution} of the $\AYBE$ in
$(A=\oplus_{i\in\bz}A_{i}, \cdot)$.

\begin{pro}[\cite{HBG}]\label{pro:splie-bia}
Let $(\g=\oplus_{i\in\bz}\g_{i}, [-,-])$ be a $\bz$-graded Lie algebra,
$r\in\g\,\hat{\otimes}\,\g$ and $\delta_{r}: \g\rightarrow\g\,\hat{\otimes}\,\g$ be
the linear map defined by
\begin{align}
\delta_{r}(g)=(\id\,\hat{\otimes}\,\ad_{\g}(g)
+\ad_{\g}(g)\,\hat{\otimes}\,\id)(r),          \label{cli-cobo}
\end{align}
for any $g\in\g$. If $r$ is a skew-symmetric completed solution of the $\CYBE$ in
$(\g=\oplus_{i\in\bz}\g_{i}, [-,-])$ then $(\g=\oplus_{i\in\bz}\g_{i}, [-,-]
\delta_{r})$ is a completed Lie bialgebra, which is called a {\bf triangular
completed Lie bialgebra} associated with $r$.
\end{pro}

Recall that a pre-Lie bialgebra $(A, \circ, \theta)$ is called {\bf coboundary}
if there exists an element $r\in A\otimes A$ such that
\begin{align}
\theta(a)=\theta_{r}(a):=\big(\fl_{A}(a)\otimes\id
+\id\otimes(\fl_{A}-\fr_{A})(a)\big)(r),                \label{preli-cobo}
\end{align}
for all $a\in A$. Let $(A, \circ)$ be a pre-Lie algebra and $r\in A\otimes A$.
We say $r$ is a solution of the {\bf pre-Lie Yang-Baxter equation} (or $\PLYBE$)
in $(A, \circ)$ if
\begin{align*}
\mathbf{PL}_{r}:&=r_{13}\circ r_{12}+r_{23}\circ r_{12}+r_{21}\circ r_{13}
+r_{23}\circ r_{13} \\[-1mm]
&\quad-r_{23}\circ r_{21}-r_{12}\circ r_{23}-r_{13}\circ r_{21}-r_{13}\circ r_{23}=0,
\end{align*}
where $r_{13}\circ r_{12}:=\sum_{i,j}(x_{i}\circ x_{j})\otimes y_{j}\otimes y_{i}$,
$r_{23}\circ r_{12}:=\sum_{i,j}x_{j}\otimes(x_{i}\circ y_{j})\otimes y_{i}$,
$r_{21}\circ r_{13}:=\sum_{i,j}(y_{i}\circ x_{j})\otimes x_{i}\otimes y_{j}$,
$r_{23}\circ r_{13}:=\sum_{i,j}x_{j}\otimes x_{i}\otimes(y_{i}\circ y_{j})$,
$r_{23}\circ r_{21}:=\sum_{i,j}y_{j}\otimes(x_{i}\circ x_{j})\otimes y_{i}$,
$r_{12}\circ r_{23}:=\sum_{i,j}x_{i}\otimes(y_{i}\circ x_{j})\otimes y_{j}$,
$r_{13}\circ r_{21}:=\sum_{i,j}(x_{i}\circ y_{j})\otimes x_{j}\otimes y_{i}$,
$r_{13}\circ r_{23}:=\sum_{i,j}x_{i}\otimes x_{j}\otimes(y_{i}\circ y_{j})$.
If $r$ is symmetric, then the $\PLYBE$ can be simplified as
$$
\mathbf{PL}_{r}=r_{12}\circ r_{13}-r_{12}\circ r_{23}-r_{13}\circ r_{23}+r_{23}\circ r_{13},
$$
which is called $\mathcal{S}$-equation in \cite{Bai1}.

\begin{pro}[\cite{Bai1}]\label{pro:spec-plbia}
Let $(A, \circ)$ be a pre-Lie algebra, $r\in A\otimes A$ and
$\theta_{r}: A\rightarrow A\otimes A$ be a linear map defined by Eq. \eqref{preli-cobo}.
If $r$ is a symmetric solution of the $\PLYBE$ in $(A, \circ)$, then $(A, \circ,
\theta_{r})$ is a pre-Lie bialgebra, which is called a {\bf triangular pre-Lie bialgebra}
associated with $r$.
\end{pro}

\begin{pro}[\cite{LZB}]\label{pro:PLYBE-CYBE}
Let $(A, \circ)$ be a pre-Lie algebra and $(B=\oplus_{i\in\bz}B_{i}, \diamond,
\omega)$ be a quadratic $\bz$-graded perm algebra, and $(A\otimes B, [-,-])$ be the
induced $\bz$-graded Lie algebra. Suppose that $r=\sum_{i}x_{i}\otimes y_{i}\in A\otimes
A$ is a symmetric solution of the $\PLYBE$ in $(A, \circ)$. Then $\widehat{r}$ given
by Eq. \eqref{r-indass} is a skew-symmetric completed solution of the $\CYBE$ in
$(A\otimes B, [-,-])$.
\end{pro}

Let $(A, \circ)$ be a pre-Lie algebra and $(B=\oplus_{i\in\bz}B_{i}, \diamond,
\omega)$ be a quadratic $\bz$-graded perm algebra, and $(A\otimes B, [-,-])$ be the
induced $\bz$-graded Lie algebra. If $r=\sum_{i}x_{i}\otimes y_{i}$ is a symmetric
solution of the $\ZYBE$ in $(A, \ast)$. By Proposition \ref{pro:PLYBE-CYBE}, we get
that $\widehat{r}$ is a skew-symmetric completed solution of the $\CYBE$ in
$(A\otimes B, [-,-])$. Thus, $(A\otimes B, [-,-], \delta_{\widehat{r}})$ is a
completed infinitesimal bialgebra, where $\delta_{\widehat{r}}$ is given by
Eq. \eqref{cli-cobo} for $\widehat{r}$. One can check $(A\otimes B, [-,-],
\delta_{\widehat{r}})=(A\otimes B, [-,-], \delta)$, where $(A\otimes B, [-,-], \delta)$
is the completed Lie bialgebra induced from $(A, \circ, \theta)$ by
$(B=\oplus_{i\in\bz}B_{i}, \diamond, \omega)$. Thus, we obtain

\begin{pro}[\cite{LZB}]\label{pro:indu-trilie}
Let $(A, \circ, \theta)$ be a pre-Lie bialgebra, $(B=\oplus_{i\in\bz}B_{i}, \diamond,
\omega)$ be a quadratic $\bz$-graded perm algebra, and $(A\otimes B, [-,-], \delta)$ be
the induced completed Lie bialgebra from $(A, \circ, \theta)$ by $(B, \diamond, \omega)$.
Then $(A\otimes B, [-,-], \delta)$ is triangular if $(A, \circ, \theta)$ is triangular.
\end{pro}

In addition, we have the following commutative diagram:
$$
\xymatrix@C=3cm@R=0.5cm{
\txt{$r$ \\ {\tiny a symmetric solution}\\ {\tiny of the $\PLYBE$ in $(A, \circ)$}}
\ar[d]_{{\rm Pro.}~\ref{pro:PLYBE-CYBE}}\ar[r]^{{\rm Pro.}~\ref{pro:spec-plbia}} &
\txt{$(A, \circ, \theta_{r})$ \\ {\tiny a triangular pre-Lie bialgebra}}
\ar[d]^{{\rm Thm.}~\ref{thm:pre-lie}}_{{\rm Pro.}~\ref{pro:indu-trilie}} \\
\txt{$\widehat{r}$ \\ {\tiny a skew-symmetric completed solution}\\ {\tiny of the $\CYBE$
in $(A\otimes B, [-,-])$}}
\ar[r]^{{\rm Pro.}~\ref{pro:splie-bia}\quad} &
\txt{$(A\otimes B, [-,-], \delta)=(A\otimes B, [-,-], \delta_{\widehat{r}})$ \\
{\tiny a completed Lie bialgebra}}}
$$

\subsection{Affinization of pre-Poisson bialgebras}\label{subsec:affPPbia}
In this subsection, we give the affinization of pre-Poisson bialgebras.
A Poisson algebra $(P, \cdot, [-,-])$ is called a {\bf $\bz$-graded Poisson algebra}
if there is a linear decomposition $P=\oplus_{i\in\bz}P_{i}$ such that $(P, \cdot)$ is a
$\bz$-graded commutative associative algebra and $(P, [-,-])$ is a $\bz$-graded Lie algebra.

\begin{pro}\label{pro:aff-PPalg}
Let $(A, \ast, \circ)$ be a finite-dimensional pre-Poisson algebra and
$(B=\oplus_{i\in\bz}B_{i}, \diamond)$ be a $\bz$-graded perm algebra.
If we define two binary operations $\cdot$ and $[-,-]$ on $A\otimes B$
by Eqs. \eqref{ind-ass} and \eqref{ind-lie} respectively, then $(A\otimes B, \cdot,
[-,-])$ is a $\bz$-graded Poisson algebra. Moreover, if $(B=\oplus_{i\in\bz}B_{i},
\diamond)$ is the $\bz$-graded perm algebra given in Example \ref{ex:grperm},
then $(A\otimes B, \cdot, [-,-])$ is a $\bz$-graded Poisson algebra if and only
if $(A, \ast, \circ)$ is a pre-Poisson algebra.
\end{pro}

\begin{proof}
If $(A, \ast, \circ)$ is a pre-Poisson algebra and $(B=\oplus_{i\in\bz}B_{i}, \diamond)$
is a $\bz$-graded perm algebra, by Propositions \ref{pro:tensor-ass}, \ref{pro:aff-Zalg}
and \ref{pro:aff-preli}, we get that $(A\otimes B, \cdot, [-,-])$ is a $\bz$-graded
Poisson algebra. Conversely, if $(A\otimes B, \cdot, [-,-])$ is a $\bz$-graded
Poisson algebra and $(B=\oplus_{i\in\bz}B_{i}, \diamond)$ is the $\bz$-graded perm
algebra given in Example \ref{ex:grperm}, we only need to show the Eqs. \eqref{PPalg1}
and \eqref{PPalg2} hold. In fact, by direct calculation, for any $a_{1}, a_{2}, a_{3}\in A$,
we have
\begin{align*}
[(a_{1}\otimes\partial_{1}),\ \ (a_{2}\otimes\partial_{2})(a_{3}\otimes\partial_{1})]
&=(a_{1}\circ(a_{2}\ast a_{3}))\otimes x_{1}x_{2}\partial_{1}
-((a_{2}\ast a_{3})\circ a_{1})\otimes x_{1}x_{2}\partial_{1}\\[-1mm]
&\quad+(a_{1}\circ(a_{3}\ast a_{2}))\otimes x_{1}^{2}\partial_{2}
-((a_{3}\ast a_{2})\circ a_{1})\otimes x_{1}x_{2}\partial_{1},\\[-10mm]
\end{align*}
\begin{align*}
[(a_{1}\otimes\partial_{1}),\ \ (a_{2}\otimes\partial_{2})](a_{3}\otimes\partial_{1})
&=((a_{1}\circ a_{2})\ast a_{3})\otimes x_{1}x_{2}\partial_{1}
+(a_{3}\ast(a_{1}\circ a_{2}))\otimes x_{1}^{2}\partial_{2}\\[-1mm]
&\quad-((a_{2}\circ a_{1})\ast a_{3})\otimes x_{1}x_{2}\partial_{1}
-(a_{3}\ast(a_{2}\circ a_{1}))\otimes x_{1}x_{2}\partial_{1},
\end{align*}
and
\begin{align*}
(a_{2}\otimes\partial_{2})[(a_{1}\otimes\partial_{1}),\ \ (a_{3}\otimes\partial_{1})]
&=(a_{2}\ast(a_{1}\circ a_{3}))\otimes x_{1}x_{2}\partial_{1}
+((a_{1}\circ a_{3})\ast a_{2})\otimes x_{1}^{2}\partial_{2}\\[-1mm]
&\quad-(a_{2}\ast(a_{3}\circ a_{1}))\otimes x_{1}x_{2}\partial_{1}
-((a_{3}\circ a_{1})\ast a_{2})\otimes x_{1}^{2}\partial_{2}.
\end{align*}
Comparing the coefficients of $x_{1}^{2}\partial_{2}$, we get $(a_{1}\ast a_{3}-
a_{3}\ast a_{1})\ast a_{2}=a_{1}\circ(a_{3}\ast a_{2})-a_{3}\ast(a_{1}\circ a_{2})$.
Similarly, comparing the coefficients of $x_{2}^{2}\partial_{1}$ in equation
$[(a_{1}\otimes\partial_{1}),\; (a_{2}\otimes\partial_{2})(a_{3}\otimes\partial_{2})]
=[(a_{1}\otimes\partial_{1}),\; (a_{2}\otimes\partial_{2})](a_{3}\otimes\partial_{2})
+(a_{2}\otimes\partial_{2})[(a_{1}\otimes\partial_{1}),\; (a_{3}\otimes\partial_{2})]$,
we get $(a_{2}\ast a_{3}+a_{3}\ast a_{2})\circ a_{1}=a_{3}\ast (a_{2}\circ a_{1})
+a_{2}\ast(a_{3}\circ a_{1})$. Thus, by Propositions \ref{pro:aff-Zalg} and
\ref{pro:aff-preli}, we obtain that $(A, \ast, \circ)$ is a Poisson algebra.
\end{proof}

A Poisson coalgebra $(P, \Delta, \delta)$ is called a {\bf completed Poisson
coalgebra} if there is a linear decomposition $P=\oplus_{i\in\bz}P_{i}$ such that
$(P, \Delta)$ is a completed cocommutative associative coalgebra and $(P, \theta)$
is a completed Lie coalgebra. We now give the dual version of Proposition \ref{pro:aff-PPalg}.

\begin{pro}\label{pro:PPco-coP}
Let $(A, \vartheta, \theta)$ be a finite-dimensional pre-Poisson coalgebra
and $(B=\oplus_{i\in\bz}B_{i}, \nu)$ be a completed perm coalgebra.
Define two linear maps $\Delta, \delta: A\otimes B\rightarrow(A\otimes B)
\,\hat{\otimes}\,(A\otimes B)$ by Eqs. \eqref{Zcoass} and \eqref{PLcoLie} respectively.
Then $(A\otimes B, \Delta, \delta)$ is a completed Poisson coalgebra.
Moreover, if $(B=\oplus_{i\in\bz}B_{i}, \nu)$ is the completed perm coalgebra given in
Example \ref{ex:grpermco}, then $(A\otimes B, \Delta, \delta)$ is a completed Poisson
coalgebra if and only if $(A, \vartheta, \theta)$ is a pre-Poisson coalgebra.
\end{pro}

\begin{proof}
If $(A, \vartheta, \theta)$ is a pre-Poisson coalgebra and $(B=\oplus_{i\in\bz}B_{i}, \nu)$
is a completed perm coalgebra, by Propositions \ref{pro:Zco-coass} and \ref{pro:aff-preli},
we get that $(A\otimes B, \Delta)$ is a completed cocommutative coassociative coalgebra
and $(A\otimes B, \delta)$ is a completed Lie coalgebra. Moreover, by Eqs. \eqref{zinbco}
\eqref{PPco1} and \eqref{PPco2}, for any $a\in A$ and $b\in B$, we have
\begin{align*}
&\;(\delta\,\hat{\otimes}\,\id)(\Delta(a\otimes b))
+(\hat{\tau}\,\hat{\otimes}\,\id)((\id\,\hat{\otimes}\,\delta)(\Delta(a\otimes b))\\
=&\;(\theta\otimes\id)(\vartheta(a))\bullet(\nu\,\hat{\otimes}\,\id)(\nu(b))
-(\tau\theta\otimes\id)(\vartheta(a))\bullet(\tau\nu\,\hat{\otimes}\,\id)(\nu(b))\\[-1mm]
&\ \ +(\id\otimes\tau)((\tau\otimes\id)((\id\otimes\theta)(\vartheta(a))))\bullet
(\id\,\hat{\otimes}\,\hat{\tau})((\hat{\tau}\,\hat{\otimes}\,\id)
((\id\,\hat{\otimes}\,\nu)(\nu(b))))\\[-1mm]
&\ \ -(\tau\otimes\id)((\id\otimes\tau)((\tau\otimes\id)((\id\otimes\theta)
(\vartheta(a)))))\bullet(\tau\,\hat{\otimes}\,\id)((\id\,\hat{\otimes}\,\hat{\tau})
((\hat{\tau}\,\hat{\otimes}\,\id)((\id\,\hat{\otimes}\,\nu)(\nu(b)))))\\[-1mm]
&\ \ +(\tau\otimes\id)((\id\otimes\theta)(\vartheta(a)))\bullet
(\hat{\tau}\,\hat{\otimes}\,\id)((\id\,\hat{\otimes}\,\nu)(\nu(b)))
+(\tau\theta\otimes\id)(\vartheta(a))\bullet(\hat{\tau}\nu\,\hat{\otimes}\,\id)(\nu(b))\\[-1mm]
&\ \ -(\tau\otimes\id)((\id\otimes\tau\theta)(\vartheta(a)))\bullet
(\hat{\tau}\,\hat{\otimes}\,\id)((\id\,\hat{\otimes}\,\hat{\tau}\nu)(\nu(b)))\\[-1mm]
&\ \ -(\id\otimes\tau)((\tau\theta\otimes\id)(\vartheta(a)))\bullet
(\id\,\hat{\otimes}\,\hat{\tau})((\hat{\tau}\nu\,\hat{\otimes}\,\id)(\nu(b)))\\
=&\; (\id\otimes\vartheta)(\theta(a))\bullet(\id\,\hat{\otimes}\,\nu)(\nu(a))
+(\id\otimes\tau\vartheta)(\theta(a))\bullet(\id\,\hat{\otimes}\,\hat{\tau}\nu)(\nu(a))\\[-1mm]
&\ \ -(\tau\otimes\id)((\id\otimes\tau)((\vartheta\otimes\id
+(\tau\otimes\id)(\vartheta\otimes\id))(\theta(a))))\bullet(\hat{\tau}\,\hat{\otimes}\,\id)
((\id\,\hat{\otimes}\,\hat{\tau})((\nu\,\hat{\otimes}\,\id)(\nu(b))))\\
=&\; (\id\,\hat{\otimes}\,\Delta)(\delta(a\otimes b)).
\end{align*}
Thus, $(A\otimes B, \Delta, \delta)$ is a completed Poisson coalgebra.

Conversely, if $(A\otimes B, \Delta, \delta)$ is a completed Poisson coalgebra and
$(B=\oplus_{i\in\bz}B_{i}, \nu)$ is the completed perm coalgebra given in Example
\ref{ex:grpermco}, by Propositions \ref{pro:Zco-coass} and \ref{pro:aff-preli},
we get that $(A, \Delta)$ is a Zinbiel coalgebra and $(A, \delta)$ is a pre-Lie coalgebra.
Thus, to prove that $(A, \vartheta, \theta)$ is a pre-Poisson coalgebra, we need to show
the Eqs. \eqref{PPco1} and \eqref{PPco2} hold. In fact, note that
\begin{align*}
&\;(\id\,\hat{\otimes}\,\hat{\tau})((\id\,\hat{\otimes}\,\nu)(\nu(\partial_{1})))\\
=&\;\sum_{i_{1}, i_{2}\in\bz}\sum_{j_{1}, j_{2}\in\bz}
\Big(x_{1}^{i_{1}}x_{2}^{i_{2}}\partial_{1}\otimes
x_{1}^{-i_{1}-j_{1}}x_{2}^{1-i_{2}-j_{2}+1}\partial_{1}\otimes
x_{1}^{j_{1}}x_{2}^{j_{2}}\partial_{1}
-x_{1}^{i_{1}}x_{2}^{i_{2}}\partial_{1}\otimes
x_{1}^{-i_{1}-j_{1}+1}x_{2}^{1-i_{2}-j_{2}}\partial_{1}\otimes
x_{1}^{j_{1}}x_{2}^{j_{2}}\partial_{2}\\[-5mm]
&\qquad\qquad\quad-x_{1}^{i_{1}}x_{2}^{i_{2}}\partial_{2}\otimes
x_{1}^{1-i_{1}-j_{1}}x_{2}^{-i_{2}-j_{2}+1}\partial_{1}\otimes
x_{1}^{j_{1}}x_{2}^{j_{2}}\partial_{1}
+x_{1}^{i_{1}}x_{2}^{i_{2}}\partial_{2}\otimes
x_{1}^{1-i_{1}-j_{1}+1}x_{2}^{-i_{2}-j_{2}}\partial_{1}\otimes
x_{1}^{j_{1}}x_{2}^{j_{2}}\partial_{2}\Big)
\end{align*}
and
\begin{align*}
&\;(\hat{\tau}\,\hat{\otimes}\,\id)((\id\,\hat{\otimes}\,\hat{\tau})
((\id\,\hat{\otimes}\,\nu)(\nu(\partial_{1}))))\\
=&\;\sum_{i_{1}, i_{2}\in\bz}\sum_{j_{1}, j_{2}\in\bz}
\Big(x_{1}^{-i_{1}-j_{1}}x_{2}^{1-i_{2}-j_{2}+1}\partial_{1}\otimes
x_{1}^{i_{1}}x_{2}^{i_{2}}\partial_{1}\otimes
x_{1}^{j_{1}}x_{2}^{j_{2}}\partial_{1}
-x_{1}^{-i_{1}-j_{1}+1}x_{2}^{1-i_{2}-j_{2}}\partial_{1}\otimes
x_{1}^{i_{1}}x_{2}^{i_{2}}\partial_{1}\otimes
x_{1}^{j_{1}}x_{2}^{j_{2}}\partial_{2}\\[-5mm]
&\qquad\qquad\quad-x_{1}^{1-i_{1}-j_{1}}x_{2}^{-i_{2}-j_{2}+1}\partial_{1}\otimes
x_{1}^{i_{1}}x_{2}^{i_{2}}\partial_{2}\otimes
x_{1}^{j_{1}}x_{2}^{j_{2}}\partial_{1}
+x_{1}^{1-i_{1}-j_{1}+1}x_{2}^{-i_{2}-j_{2}}\partial_{1}\otimes
x_{1}^{i_{1}}x_{2}^{i_{2}}\partial_{2}\otimes
x_{1}^{j_{1}}x_{2}^{j_{2}}\partial_{2}\Big).
\end{align*}
Comparing the coefficients of $x_{1}^{i_{1}}x_{2}^{i_{2}}\partial_{1}\otimes
x_{1}^{j_{1}}x_{2}^{j_{2}}\partial_{1}\otimes x_{1}^{-i_{1}-j_{1}}
x_{2}^{2-i_{2}-j_{2}}\partial_{1}$ and $x_{1}^{i_{1}}x_{2}^{i_{2}}\partial_{1}\otimes
x_{1}^{j_{1}}x_{2}^{j_{2}}\partial_{2}\otimes x_{1}^{2-i_{1}-j_{1}}
x_{2}^{-i_{2}-j_{2}}\partial_{2}$ in equation $(\id\,\hat{\otimes}\,\Delta)
(\delta(a\otimes\partial_{1}))=(\delta\,\hat{\otimes}\,\id)(\Delta(a\otimes\partial_{1}))
+(\hat{\tau}\,\hat{\otimes}\,\id)((\id\,\hat{\otimes}\,\delta)
(\Delta(a\otimes\partial_{1}))$, we get $((\theta-\tau\theta)\otimes\id)\vartheta
=(\id\otimes\vartheta)\theta-(\tau\otimes\id)(\id\otimes\theta)\vartheta$ and
$((\vartheta+\tau\vartheta)\otimes\id)\theta=(\id\otimes\theta)\vartheta
+(\tau\otimes\id)(\id\otimes\theta)\vartheta$ respectively.
Thus, $(A, \vartheta)$ is a pre-Poisson coalgebra.
\end{proof}

We now give the notion and results on completed Poisson bialgebras.

\begin{defi}\label{def:CPoisbia}
A {\bf completed Poisson bialgebra} is a quintuple $(P, \cdot, [-,-], \Delta, \delta)$
consisting of a vector space $P=\oplus_{i\in\bz}P_{i}$, two bilinear maps $\cdot, [-,-]:
P\,\hat{\otimes}\,P\rightarrow P$ and two linear maps $\Delta, \delta: P\rightarrow
P\,\hat{\otimes}\,P$ such that
\begin{enumerate}\itemsep=0pt
\item[$(i)$] $(P, \cdot, [-,-])$ is a $\bz$-graded Poisson algebra;
\item[$(ii)$] $(P, \Delta, \delta)$ is a completed Poisson coalgebra;
\item[$(iii)$] $(P, \cdot, \Delta)$ is a completed infinitesimal bialgebra;
\item[$(iv)$] $(P, [-,-], \delta)$ is a completed Lie bialgebra;
\item[$(v)$] for any $p_{1}, p_{2}\in P$,
\begin{align}
&\Delta([p_{1}, p_{2}])=\big(\ad_{P}(p_{1})\,\hat{\otimes}\,\id
+\id\,\hat{\otimes}\,\ad_{P}(p_{1})\big)\Delta(p_{2})+\big(\fu_{P}(p_{2})\,\hat{\otimes}\,\id
-\id\,\hat{\otimes}\,\fu_{P}(p_{2})\big)\delta(p_{1}),  \label{cPbialg1}\\
&\qquad\quad \delta(p_{1}p_{2})=(\fu_{P}(p_{1})\,\hat{\otimes}\,\id)\delta(p_{2})
+(\fu_{P}(p_{2})\,\hat{\otimes}\,\id)\delta(p_{1})      \label{cPbialg2}\\[-1mm]
&\qquad\qquad\qquad\qquad+(\id\,\hat{\otimes}\,\ad_{P}(p_{1}))\Delta(p_{2})
+(\id\,\hat{\otimes}\,\ad_{P}(p_{2}))\Delta(p_{1}).    \nonumber
\end{align}
\end{enumerate}
\end{defi}

\begin{thm}\label{thm:PP-Poisbia}
Let $(A, \ast, \circ, \vartheta, \theta)$ be a finite-dimensional pre-Poisson bialgebra,
$(B=\oplus_{i\in\bz}B_{i}, \diamond, \omega)$ be a quadratic $\bz$-graded perm
algebra and $(A\otimes B, \cdot, [-,-])$ be the induced $\bz$-graded Poisson
algebra from $(A, \ast, \circ)$ by $(B=\oplus_{i\in\bz}B_{i}, \diamond)$.
Define two linear maps $\Delta, \delta: A\otimes B\rightarrow(A\otimes B)\otimes
(A\otimes B)$ by Eqs. \eqref{Zcoass} and \eqref{PLcoLie} for $\nu=\nu_{\omega}$.
Then $(A\otimes B, \cdot, [-,-], \Delta, \delta)$ is a completed Poisson bialgebra,
which is called the {\bf completed Poisson bialgebra induced from $(A, \ast, \circ,
\vartheta, \theta)$ by $(B, \diamond, \omega)$}.

Moreover, if $(B=\oplus_{i\in\bz}B_{i}, \diamond, \omega)$ is the quadratic $\bz$-graded
perm algebra given in Example \ref{ex:qu-perm}, then $(A\otimes B, \cdot, [-,-],
\Delta, \delta)$ is a completed Poisson bialgebra if and only if
$(A, \ast, \circ, \vartheta, \theta)$ is a pre-Poisson bialgebra.
\end{thm}

\begin{proof}
If $(A, \ast, \circ, \vartheta, \theta)$ is a pre-Poisson bialgebra and
$(B=\oplus_{i\in\bz}B_{i}, \diamond, \omega)$ be a quadratic $\bz$-graded perm
algebra, by Propositions \ref{pro:aff-PPalg} and \ref{pro:PPco-coP}, Theorems
\ref{thm:Z-perm-ass} and \ref{thm:pre-lie}, we get that $(A\otimes B, \cdot, [-,-])$
is a $\bz$-graded Poisson algebra, $(A\otimes B, \Delta, \delta)$ is a completed Poisson
coalgebra, $(A\otimes B, \cdot, \Delta)$ is a completed infinitesimal bialgebra and
$(A\otimes B, [-,-], \delta)$ is a completed Lie bialgebra.
Moreover, for any $a, a'\in A$ and $b, b'\in B$, we have
\begin{align*}
&\;\Delta\big([a\otimes b,\; a'\otimes b']\big)
-\big(\ad_{A\otimes B}(a\otimes b)\,\hat{\otimes}\,\id
+\id\,\hat{\otimes}\,\ad_{A\otimes B}(a\otimes b)\big)
\Delta(a'\otimes b')\\[-1mm]
&\qquad-\big(\fu_{A\otimes B}(a'\otimes b')\,\hat{\otimes}\,\id
-\id\,\hat{\otimes}\,\fu_{A\otimes B}(a'\otimes b')\big)\delta(a\otimes b)\\
=&\;\vartheta(a\circ a')\bullet\nu_{\omega}(b\diamond b')
+\tau(\vartheta(a\circ a'))\bullet\hat{\tau}(\nu_{\omega}(b\diamond b'))\\[-1mm]
&\; -\vartheta(a'\circ a)\bullet\nu_{\omega}(b'\diamond b)
-\tau(\vartheta(a'\circ a))\bullet\hat{\tau}(\nu_{\omega}(b'\diamond b))\\[-1mm]
&\; +\sum_{i,j,\alpha}\Big(-(\hat{\fl}_{A}(a)\otimes\id)(\vartheta(a'))
\bullet((b\diamond b'_{1,i,\alpha})\otimes b'_{2,j,\alpha})+(\hat{\fr}_{A}(a)\otimes\id)
(\vartheta(a'))\bullet((b'_{1,i,\alpha}\diamond b)\otimes b'_{2,j,\alpha})\\[-4mm]
&\qquad\ \ -(\hat{\fl}_{A}(a)\otimes\id)(\tau(\vartheta(a')))\bullet
((b\diamond b'_{2,j,\alpha})\otimes b'_{1,i,\alpha})+(\hat{\fr}_{A}(a)\otimes\id)
(\tau(\vartheta(a')))\bullet((b'_{2,j,\alpha}\diamond b)\otimes b'_{1,i,\alpha})\\
&\qquad\ \ -(\id\otimes\hat{\fl}_{A}(a))(\vartheta(a'))\bullet(b'_{1,i,\alpha}\otimes
(b\diamond b'_{2,j,\alpha}))+(\id\otimes\hat{\fr}_{A}(a))(\vartheta(a'))\bullet
(b'_{1,i,\alpha}\otimes(b'_{2,j,\alpha}\diamond b))\\[-1mm]
&\qquad\ \ -(\id\otimes\hat{\fl}_{A}(a))(\tau(\vartheta(a')))\bullet(b'_{2,j,\alpha}
\otimes(b\diamond b'_{1,i,\alpha}))+(\id\otimes\hat{\fr}_{A}(a))(\tau(\vartheta(a')))
\bullet(b'_{2,j,\alpha}\otimes(b'_{1,i,\alpha}\diamond b))\Big)\\[-1mm]
\end{align*}
\begin{align*}
&\; +\sum_{i,j,\alpha}\Big(-(\fl_{A}(a')\otimes\id)(\theta(a))\bullet
((b'\diamond b_{1,i,\alpha})\otimes b_{2,j,\alpha})-(\fr_{A}(a')\otimes\id)(\theta(a))
\bullet((b_{1,i,\alpha}\diamond b')\otimes b_{2,j,\alpha})\\[-4mm]
&\qquad\ \ +(\fl_{A}(a')\otimes\id)(\tau(\theta(a)))\bullet((b'\diamond b_{2,j,\alpha})
\otimes b_{1,i,\alpha})+(\fr_{A}(a')\otimes\id)(\tau(\theta(a)))\bullet((b_{2,j,\alpha}
\diamond b')\otimes b_{1,i,\alpha})\\
&\qquad\ \ +(\id\otimes\fl_{A}(a'))(\theta(a))\bullet(b_{1,i,\alpha}\otimes
(b'\diamond b_{2,j,\alpha}))+(\id\otimes\fr_{A}(a'))(\theta(a))\bullet(b_{1,i,\alpha}
\otimes(b_{2,j,\alpha}\diamond b'))\\[-1mm]
&\qquad\ \ -(\id\otimes\fl_{A}(a'))(\tau(\theta(a)))\bullet(b_{2,j,\alpha}\otimes
(b'\diamond b_{1,i,\alpha}))-(\id\otimes\fr_{A}(a'))(\tau(\theta(a)))\bullet
(b_{2,j,\alpha}\otimes(b_{1,i,\alpha}\diamond b'))\Big).
\end{align*}
Similar to the proof of Theorem \ref{thm:Z-perm-ass}, by the left nondegeneracy of
$\hat{\omega}(-,-)$, every formula of the form $\sum_{i,j,\alpha}(b'_{1,i,\alpha}
\otimes(b\diamond b'_{2,j,\alpha}))$ can be transformed into a combination of
$\nu_{\omega}(b\diamond b')$, $\hat{\tau}(\nu_{\omega}(b\diamond b'))$ and $\Phi$,
where $\Phi\in B\otimes B$ such that $\hat{\omega}(\Phi,\; e\otimes f)
=\omega(b',\; e\diamond(f\diamond b))$. Thus, we get
\begin{align*}
&\;\Delta\big([a\otimes b,\; a'\otimes b']\big)
-\big(\ad_{A\otimes B}(a\otimes b)\,\hat{\otimes}\,\id
+\id\,\hat{\otimes}\,\ad_{A\otimes B}(a\otimes b)\big)
\Delta(a'\otimes b')\\[-1mm]
&\qquad-\big(\fu_{A\otimes B}(a'\otimes b')\,\hat{\otimes}\,\id
-\id\,\hat{\otimes}\,\fu_{A\otimes B}(a'\otimes b')\big)\delta(a\otimes b)\\[-1mm]
=&\; \Big(\vartheta(a\circ a')-\vartheta(a'\circ a)-(\hat{\fl}_{A}(a)\otimes\id)(\vartheta(a'))
-(\id\otimes\hat{\fl}_{A}(a))(\vartheta(a'))+(\id\otimes\hat{\fr}_{A}(a))
(\vartheta(a'))\\[-2mm]
&\qquad-(\fl_{A}(a')\otimes\id)(\theta(a))+(\id\otimes\fl_{A}(a'))(\theta(a))
+(\id\otimes\fr_{A}(a'))(\theta(a))\Big)\bullet\nu_{\omega}(b\diamond b')\\[-1mm]
&\; +\tau\Big(\vartheta(a\circ a')-\vartheta(a'\circ a)-(\hat{\fl}_{A}(a)\otimes\id)
(\vartheta(a'))-(\id\otimes\hat{\fl}_{A}(a))(\vartheta(a'))+(\id\otimes\hat{\fr}_{A}(a))
(\vartheta(a'))\\[-2mm]
&\qquad-(\fl_{A}(a')\otimes\id)(\theta(a))+(\id\otimes\fl_{A}(a'))(\theta(a))
+(\id\otimes\fr_{A}(a'))(\theta(a))\Big)\bullet\tau(\nu_{\omega}(b\diamond b'))\\[-1mm]
&\; -\Big(\vartheta(a'\circ a)+\tau(\vartheta(a'\circ a))-(\hat{\fr}_{A}(a)\otimes\id)
(\tau(\vartheta(a')))-(\id\otimes\hat{\fr}_{A}(a))(\vartheta(a'))
+(\fl_{A}(a')\otimes\id)(\theta(a))\\[-2mm]
&\qquad+(\fl_{A}(a')\otimes\id)(\tau(\theta(a)))-(\id\otimes\fl_{A}(a'))(\theta(a))
+(\id\otimes\fl_{A}(a'))(\tau(\theta(a)))\Big)\bullet\Phi.
\end{align*}
By Eqs. \eqref{PPbialg3} and \eqref{PPbialg4}, we get $\vartheta(a'\circ a)
+\tau(\vartheta(a'\circ a))-(\hat{\fr}_{A}(a)\otimes\id)(\tau(\vartheta(a')))
-(\id\otimes\hat{\fr}_{A}(a))(\vartheta(a'))+(\fl_{A}(a')\otimes\id)(\theta(a))
-(\fl_{A}(a')\otimes\id)(\tau(\theta(a)))-(\id\otimes\fl_{A}(a'))(\theta(a))
+(\id\otimes\fl_{A}(a'))(\tau(\theta(a)))=0$. Thus, from this equation and
Eq. \eqref{PPbialg2}, we can obtain $\Delta\big([a\otimes b,\; a'\otimes b']\big)
-\big(\ad_{A\otimes B}(a\otimes b)\,\hat{\otimes}\,\id+\id\,\hat{\otimes}\,\ad_{A\otimes B}
(a\otimes b)\big)\Delta(a'\otimes b')-\big(\fu_{A\otimes B}(a'\otimes b')\,\hat{\otimes}\,\id
-\id\,\hat{\otimes}\,\fu_{A\otimes B}(a'\otimes b')\big)\delta(a\otimes b)=0$.
Similarly, we also have $\delta\big((a\otimes b)(a'\otimes b')\big)
-\big(\fu_{A\otimes B}(a\otimes b)\,\hat{\otimes}\,\id\big)(\delta(a'\otimes b'))
-\big(\fu_{A\otimes B}(a'\otimes b')\,\hat{\otimes}\,\id\big)(\delta(a\otimes b))
-\big(\id\,\hat{\otimes}\,\ad_{A\otimes B}(a\otimes b)\big)(\Delta(a'\otimes b'))
-\big(\id\,\hat{\otimes}\,\ad_{A\otimes B}(a'\otimes b')\big)(\Delta(a\otimes b))=0$
for any $a, a'\in A$ and $b, b'\in B$. Thus, $(A\otimes B, \cdot, [-,-], \Delta,
\delta)$ is a completed Poisson bialgebra.

Conversely, if $(B=\oplus_{i\in\bz}B_{i}, \diamond, \omega)$ is the quadratic $\bz$-graded
perm algebra given in Example \ref{ex:qu-perm} and $(A\otimes B, \cdot, [-,-], \Delta,
\delta)$ is a completed Poisson bialgebra, by comparing the coefficients of $x_{1}^{i_{1}}
x_{2}^{i_{2}}\partial_{2}\otimes x_{1}^{2-i_{1}}x_{2}^{-i_{2}}\partial_{1}$ in equation
\begin{align*}
&\;\Delta\big([a\otimes\partial_{1},\; a'\otimes\partial_{1}]\big)
-\big(\ad_{A\otimes B}(a\otimes\partial_{1})\,\hat{\otimes}\,\id
+\id\,\hat{\otimes}\,\ad_{A\otimes B}(a\otimes\partial_{1})\big)
\Delta(a'\otimes\partial_{1})\\[-1mm]
&\qquad-\big(\fu_{A\otimes B}(a'\otimes\partial_{1})\,\hat{\otimes}\,\id
-\id\,\hat{\otimes}\,\fu_{A\otimes B}(a'\otimes\partial_{1})\big)
\delta(a\otimes\partial_{1})=0,
\end{align*}
we get $\vartheta(a\circ a'-a'\circ a)-(\fl_{A}(a')\otimes\id)\theta(a)
+(\id\otimes\fl_{A}(a'))\theta(a)+(\id\otimes\fr_{A}(a'))\theta(a)
-(\hat{\fl}_{A}(a)\otimes\id)\vartheta(a')-(\id\otimes\hat{\fl}_{A}(a))\vartheta(a')
+(\id\otimes\hat{\fr}_{A}(a))\vartheta(a')=0$. That is, Eq. \eqref{PPbialg2} holds.
Similarly, Eqs. \eqref{PPbialg1}, \eqref{PPbialg3} and \eqref{PPbialg4} hold.
Thus, by Propositions \ref{pro:aff-PPalg} and \ref{pro:PPco-coP}, Theorems
\ref{thm:Z-perm-ass} and \ref{thm:pre-lie}, we get that $(A, \ast, \circ, \vartheta,
\theta)$ is a pre-Poisson bialgebra.
\end{proof}

For the construction of finite-dimensional Poisson bialgebras, we have the
following corollary.

\begin{cor}\label{cor:ind-Pois}
Let $(A, \ast, \circ, \vartheta, \theta)$ be a finite-dimensional pre-Poisson bialgebra,
$(B, \diamond, \omega)$ be a quadratic perm algebra and $(A\otimes B, \cdot, [-,-])$
be the induced $\bz$-graded Poisson algebra from $(A, \ast, \circ)$ by
$(B=\oplus_{i\in\bz}B_{i}, \diamond)$. Define two linear maps $\Delta, \delta:
A\otimes B\rightarrow(A\otimes B)\otimes(A\otimes B)$ by Eq. \eqref{fZcoass} and
\begin{align}
\delta(a\otimes b)
&=(\id\otimes\id-\tau)\big(\theta(a)\bullet\nu_{\omega}(b)\big)     \label{fplcoli} \\
&:=\sum_{(a)}\sum_{(b)}\Big((a_{(1)}\otimes b_{(1)})\otimes(a_{(2)}\otimes b_{(2)})
+(a_{(2)}\otimes b_{(2)})\otimes(a_{(1)}\otimes b_{(1)})\Big),   \nonumber
\end{align}
for any $a\in A$ and $b\in B$, where $\theta(a)=\sum_{(a)}a_{(1)}\otimes a_{(2)}$ and
$\nu_{\omega}(b)=\sum_{(b)}b_{(1)}\otimes b_{(2)}$ in the Sweedler notation.
Then $(A\otimes B, \cdot, [-,-], \Delta, \delta)$ is a Poisson bialgebra, which is
called the {\bf Poisson bialgebra induced from $(A, \ast, \circ, \vartheta, \theta)$
by $(B, \diamond, \omega)$}.
\end{cor}

\begin{ex}\label{ex:prePbia-Pbia}
Let $(A=\bk\{e_{1}, e_{2}\}, \ast, \circ)$ be the two dimensional pre-Poisson algebra
given in Example \ref{ex:preP-P}. Then $r=e_{1}\otimes e_{2}+e_{2}\otimes e_{1}$ is a
symmetric solution of the $\PPYBE$ in $(A, \ast, \circ)$. Thus, $r$ induces a triangular
pre-Poisson bialgebra structure on $A=\bk\{e_{1}, e_{2}\}$, where the coproducts
$\vartheta$ and $\theta$ on $A$ are given by $\vartheta(e_{1})=e_{2}\otimes e_{2}$,
$\vartheta(e_{2})=0$, $\theta(e_{1})=e_{1}\otimes e_{2}+e_{2}\otimes e_{1}$ and
$\theta(e_{2})=e_{2}\otimes e_{2}$.
Consider the two dimensional quadratic perm algebra $(B=\bk\{x_{1}, x_{2}\}, \diamond,
\omega)$ given in Example \ref{ex:Zbia-Assbia}. By Corollary \ref{cor:ind-Pois},
we have a Poisson bialgebra $(A\otimes B, \cdot, [-,-], \Delta, \delta)$, where the
Poisson algebra $(A\otimes B, \cdot, [-,-])$ is given in Example \ref{ex:preP-P},
the coproduct $\Delta$ is given in Example \ref{ex:Zbia-Assbia}, and the coproduct
$\delta$ is given by
$$
\delta(y_{1})=y_{1}\otimes y_{4}+y_{3}\otimes y_{2}-y_{4}\otimes y_{1}-y_{2}\otimes y_{3},
\quad \delta(y_{4})=y_{3}\otimes y_{4}-y_{4}\otimes y_{3}, \quad \Delta(y_{1})=\Delta(y_{3})=0,
$$
$y_{1}:=e_{1}\otimes x_{1}$, $y_{2}:=e_{1}\otimes x_{2}$, $y_{3}:=e_{2}\otimes x_{1}$,
$y_{4}:=e_{2}\otimes x_{2}$. Moreover, the skew-symmetric solution $\widehat{r}$ of the
$\AYBE$ in $(A\otimes B, \cdot)$ is also a skew-symmetric solution of the $\PYBE$ in
$(A\otimes B, \cdot, [-,-])$, and the Poisson bialgebra structure induced by
$\widehat{r}$ on $A\otimes B$ is just the Poisson bialgebra structure given as above.
\end{ex}

Recall that a pre-Poisson bialgebra $(A, \ast, \circ, \vartheta, \theta)$ is called
{\bf coboundary} if there exists $r\in A\otimes A$ such that $\vartheta=\vartheta_{r}$
and $\theta=\theta_{r}$, where $\vartheta_{r}$ and $\theta_{r}$ are given by Eqs.
\eqref{Zcobo} and \eqref{preli-cobo} respectively.
Let $(A, \ast, \circ)$ be a pre-Poisson algebra and $r=\sum_{i}x_{i}\otimes y_{i}
\in A\otimes A$. We say $r$ is a solution of the {\bf pre-Poisson Yang-Baxter equation}
(or $\PPYBE$) in $(A, \ast, \circ)$ if $\mathbf{Z}_{r}=\mathbf{PL}_{r}=0$. Then by
Propositions \ref{pro:spec-Zbia} and \ref{pro:spec-plbia}, we have

\begin{pro}[\cite{WS}]\label{pro:tri-PPbia}
Let $(A, \ast, \circ)$ be a pre-Poisson algebra, $r\in A\otimes A$ and $\vartheta_{r}:
\theta_{r}: A\rightarrow A\otimes A$ be linear maps defined by Eqa. \eqref{Zcobo} and
\eqref{preli-cobo} respectively. If $r$ is a symmetric solution of the $\PPYBE$ in
$(A, \ast, \circ)$, then $(A, \ast, \circ, \vartheta_{r}, \theta_{r})$ is a
pre-Poisson bialgebra, which is called a {\bf triangular pre-Poisson bialgebra}
associated with $r$.
\end{pro}

Let $(P=\oplus_{i\in\bz}P_{i}, \cdot, [-,-])$ be a $\bz$-graded Poisson algebra.
Suppose that $r=\sum_{i,j,\alpha}x_{i\alpha}\otimes y_{j\alpha}\in P\,\hat{\otimes}\,P$.
If $r$ satisfies $\mathbf{A}_{r}=\mathbf{C}_{r}=0$ as an element in
$P\,\hat{\otimes}\,P\,\hat{\otimes}\,P$, then $r$ is called a {\bf completed solution}
of the $\PYBE$ in $(A=\oplus_{i\in\bz}A_{i}, \cdot, [-,-])$.
The same argument of the proof for \cite[Theorem 2]{NB} extends to the completed case.
We obtain the following proposition.

\begin{pro}\label{pro:poi-tri}
Let $(P=\oplus_{i\in\bz}P_{i}, \cdot, [-,-])$ be a $\bz$-graded Poisson algebra
and $r\in P\,\hat{\otimes}\,P$ is a completed solution of the $\PYBE$ in
$(P=\oplus_{i\in\bz}P_{i}, \cdot, [-,-])$. Define two bilinear maps $\Delta_{r},
\delta_{r}: A\rightarrow A\,\hat{\otimes}\,A$ by Eqs. \eqref{cass-cobo} and
\eqref{cli-cobo} respectively. If $r$ is skew-symmetric, then $(A, \cdot, [-,-],
\Delta_{r}, \delta_{r})$ is a completed Poisson bialgebra, which is called a
{\bf triangular completed Poisson bialgebra} associated with $r$.
\end{pro}

Thus, by Propositions \ref{pro:ZYBE-AYBE} and \ref{pro:PLYBE-CYBE}, we get

\begin{pro}\label{pro:PPYBE-PYBE}
Let $(A, \ast, \circ)$ be a pre-Poisson algebra and $(B=\oplus_{i\in\bz}B_{i}, \diamond,
\omega)$ be a quadratic $\bz$-graded perm algebra, and $(A\otimes B, \cdot, [-,-])$ be the
induced $\bz$-graded Poisson algebra. Suppose that $r=\sum_{i}x_{i}\otimes y_{i}\in
A\otimes A$ is a symmetric solution of the $\PPYBE$ in $(A, \ast, \circ)$. Then
the element $\widehat{r}$ defined by Eq. \eqref{r-indass} is a skew-symmetric
completed solution of the $\PYBE$ in $(A\otimes B, \cdot, [-,-])$.
\end{pro}

By Theorem \ref{thm:indu-triass} and Proposition \ref{pro:indu-trilie}, we have

\begin{pro}\label{pro:indu-tripoi}
Let $(A, \ast, \circ, \vartheta, \theta)$ be a pre-Poisson bialgebra,
$(B=\oplus_{i\in\bz}B_{i}, \diamond, \omega)$ be a quadratic $\bz$-graded perm algebra
and $(A\otimes B, \cdot, [-,-], \Delta, \delta)$ be the induced completed Poisson
bialgebra from $(A, \ast, \circ, \vartheta, \theta)$ by $(B, \diamond, \omega)$.
Then $(A\otimes B, \cdot, [-,-], \Delta, \delta)$ is a triangular completed Poisson bialgebra
if $(A, \ast, \circ, \vartheta, \theta)$ is a triangular pre-Poisson bialgebra.
\end{pro}

In addition, we have the following commutative diagram:
$$
\xymatrix@C=3cm@R=0.5cm{
\txt{$r$ \\ {\tiny a symmetric solution}\\ {\tiny of the $\PPYBE$ in $(A, \ast, \circ)$}}
\ar[d]_{{\rm Pro.}~\ref{pro:PPYBE-PYBE}}\ar[r]^{{\rm Pro.}~\ref{pro:tri-PPbia}} &
\txt{$(A, \ast, \circ, \vartheta_{r}, \theta_{r})$ \\ {\tiny a triangular pre-Poisson
bialgebra}}\ar[d]^{{\rm Thm.}~\ref{thm:PP-Poisbia}}_{{\rm Pro.}~\ref{pro:indu-tripoi}} \\
\txt{$\widehat{r}$ \\ {\tiny a skew-symmetric completed solution}\\ {\tiny of the $\PYBE$
in $(A\otimes B, \cdot, [-,-])$}} \ar[r]^{{\rm Pro.}~\ref{pro:poi-tri}\qquad\quad} &
\txt{$(A\otimes B, \cdot, [-,-], \Delta, \delta)=(A\otimes B, \cdot, [-,-],
\Delta_{\widehat{r}}, \delta_{\widehat{r}})$ \\ {\tiny a completed Poisson bialgebra}}}
$$

Let $(\g, [-,-])$ be a Lie algebra and $(V, \rho)$ be a representation of $(\g, [-,-])$.
Recall that a linear map $\mathcal{T}: V\rightarrow\g$ is called an
{\bf $\mathcal{O}$-operator of $(\g, [-,-])$ associated to $(V, \rho)$} if
for any $v_{1}, v_{2}\in V$,
$$
[\mathcal{T}(v_{1}),\; \mathcal{T}(v_{2})]=
\mathcal{T}\big(\rho(\mathcal{T}(v_{1}))(v_{2})-\rho(\mathcal{T}(v_{2}))(v_{1})\big).
$$
An {\bf $\mathcal{O}$-operator of a Poisson algebra $(P, \cdot, [-,-])$
associated to a representation $(V, \mu, \rho)$} is a linear map $\mathcal{T}:
V\rightarrow P$ such that it is both an $\mathcal{O}$-operator of $(P, \cdot)$
associated to representation $(V, \mu)$ and an $\mathcal{O}$-operator of
$(P, [-,-])$ associated to representation $(V, \rho)$.

\begin{pro}[\cite{NB}]\label{pro:o-poi}
Let $(P, \ast, [-,-])$ be a Poisson algebra, $r\in P\otimes P$ be skew-symmetric.
Then $r$ is a solution of the $\PYBE$ in $(P, \ast, [-,-])$ if and only if $r^{\sharp}:
P^{\ast}\rightarrow P$ is an $\mathcal{O}$-operator of $(P, \ast, [-,-])$ associated
to the coregular representation $(P^{\ast}, -\fu_{P}^{\ast}, \ad_{P}^{\ast})$.
\end{pro}

The $\mathcal{O}$-operator of pre-Lie algebras was considered in \cite{Bai}.
Recall that an {\bf $\mathcal{O}$-operator of a pre-Lie algebra $(A, \circ)$
associated to a representation $(V, \hat{\kl}, \hat{\kr})$} is a linear map
$\mathcal{T}: V\rightarrow A$ such that
$$
\mathcal{T}(v_{1})\circ\mathcal{T}(v_{2})
=\mathcal{T}\big(\hat{\kl}(\mathcal{T}(v_{1}))(v_{2})+\hat{\kr}(\mathcal{T}(v_{2}))(v_{1})\big),
$$
for any $v_{1}, v_{2}\in V$. Let $(P, \ast, \circ)$ be a pre-Poisson and
$(V, \kl, \kr, \hat{\kl}, \hat{\kr})$ be a representation of $(P, \ast, \circ)$.
A linear map $\mathcal{T}: V\rightarrow P$ is called an {\bf $\mathcal{O}$-operator
of pre-Poisson algebra $(A, \ast, \circ)$ associated to representation
$(V, \kl, \kr, \hat{\kl}, \hat{\kr})$} if $\mathcal{T}$ is both an $\mathcal{O}$-operator
of $(P, \ast)$ associated to $(V, \kl, \kr)$ and an $\mathcal{O}$-operator
of $(P, \circ)$ associated to $(V, \hat{\kl}, \hat{\kr})$. By \cite[Theorem 2]{NB}
and \cite[Theorem 2]{NB}, we obtain

\begin{pro}\label{pro:o-preP}
Let $(A, \ast, \circ)$ be a pre-Poisson algebra, $r\in A\otimes A$ be symmetric.
Then $r$ is a solution of the $\PPYBE$ in $(A, \ast, \circ)$ if and only if $r^{\sharp}:
A^{\ast}\rightarrow A$ is an $\mathcal{O}$-operator of $(A, \ast, \circ)$ associated to
the coregular representation $(A^{\ast}, -\fl_{A}^{\ast}-\fr_{A}^{\ast}, \fr_{A}^{\ast},
\hat{\fl}_{A}^{\ast}-\hat{\fr}_{A}^{\ast}, -\hat{\fr}_{A}^{\ast})$.
\end{pro}

Let $(B=B_{0}, \diamond, \omega)$ be a (finite-dimensional) quadratic perm algebra,
$\{e_{1}, e_{2},\cdots,e_{n}\}$ be a basis of $B$ and $\{f_{1}, f_{2},\cdots,f_{n}\}$
be the dual basis of $\{e_{1}, e_{2},\cdots,e_{n}\}$ with respect to $\omega(-,-)$.
Suppose $(A, \ast, \circ)$ is a pre-Poisson algebra. By Proposition
\ref{pro:PPYBE-PYBE}, we get
$$
\widehat{r}=\sum_{i}\sum_{j}(x_{i}\otimes e_{j})\otimes(y_{i}\otimes f_{j})
\in(A\otimes B)\otimes(A\otimes B)
$$
is a skew-symmetric solution of the $\PYBE$ in the induced Poisson algebra $(A\otimes B,
\cdot, [-,-])$ if $r=\sum_{i}x_{i}\otimes y_{i}$ is a symmetric solution
of the $\PPYBE$ in $(A, \ast, \circ)$. Thus, $\widehat{r}^{\sharp}: (A\otimes B)^{\ast}
\rightarrow A\otimes B$  is an $\mathcal{O}$-operator of $(A\otimes B, \cdot, [-,-])$
associated to representation $((A\otimes B)^{\ast}, -\fu^{\ast}_{A\otimes B},
\ad^{\ast}_{A\otimes B})$. In Proposition \ref{pro:ind-oper}, we have shown that
$\widehat{r}^{\sharp}=r^{\sharp}\otimes\kappa^{\sharp}$ as linear maps. Thus, we obtain
the following commutative diagram:
$$
\xymatrix@C=3cm@R=0.5cm{
\txt{$r$ \\ {\tiny a symmetric solution} \\ {\tiny of the $\PPYBE$ in $(A, \ast, \circ)$}}
\ar[d]_-{{\rm Pro.}~\ref{pro:PPYBE-PYBE}}\ar[r]^-{{\rm Pro.}~\ref{pro:o-preP}} &
\txt{$r^{\sharp}$\\ {\tiny an $\mathcal{O}$-operator of $(A, \ast, \circ)$
associated} \\ {\tiny to $(A^{\ast}, -\fl^{\ast}_{A}-\fr^{\ast}_{A}, \fr^{\ast}_{A},
\hat{\fl}^{\ast}_{A}-\hat{\fr}^{\ast}_{A}, -\hat{\fr}^{\ast}_{A})$}}
\ar[d]^-{\mbox{$-\otimes\kappa^{\sharp}$}} \\
\txt{$\widehat{r}$ \\ {\tiny a skew-symmetric solution} \\ {\tiny of the $\PYBE$ in
$(A\otimes B, \cdot, [-,-])$}} \ar[r]^-{{\rm Pro.}~\ref{pro:o-poi}}
& \txt{$\widehat{r}^{\sharp}=r^{\sharp}\otimes\kappa^{\sharp}$ \\
{\tiny an $\mathcal{O}$-operator of $(A\otimes B, \cdot, [-,-])$ } \\
{\tiny associated to $((A\otimes B)^{\ast}, -\fu^{\ast}_{A\otimes B},
\ad^{\ast}_{A\otimes B})$}}}
$$

\subsection{Affinization of quasi-Frobenius pre-Poisson algebras}\label{subsec:affQFPP}
In this subsection, we give a construction of quasi-Frobenius $\bz$-graded
Poisson algebra from quasi-Frobenius pre-Poisson algebras corresponding to a class of
symmetric solutions of the $\PPYBE$.

\begin{defi}\label{def:frob-Z}
Let $(A, \circ)$ be a pre-Lie algebra. If there is a symmetric nondegenerate
bilinear form $\varpi(-,-)$ on $A$ satisfying
$$
\varpi(a_{1}\circ a_{2},\; a_{3})-\varpi(a_{1},\; a_{2}\circ a_{3})
=\varpi(a_{2}\circ a_{1},\; a_{3})-\varpi(a_{2},\; a_{1}\circ a_{3}),
$$
for any $a_{1}, a_{2}, a_{3}\in A$, then $(A, \circ, \varpi)$ is called a
{\bf quasi-Frobenius pre-Lie algebra}.

Let $(A, \ast, \circ)$ be a pre-Poisson algebra and $\varpi(-,-)$ be a symmetric
nondegenerate bilinear form on $A$. If $(A, \ast, \varpi)$ is a quasi-Frobenius
Zinbiel algebra and $(A, \circ, \varpi)$ is a quasi-Frobenius pre-Lie algebra, then
$(A, \ast, \circ, \varpi)$ is called a {\bf quasi-Frobenius pre-Poisson algebra}.
\end{defi}

The quasi-Frobenius pre-Lie algebras is closely related to the $L$-dendriform
algebras (see \cite{BLN} for details). For the quasi-Frobenius pre-Poisson algebras,
we can show that it is closely related to symmetric solutions of the $\PPYBE$.

\begin{pro}\label{pro:quasi-PP}
Let $(A, \ast, \circ)$ be a pre-Poisson algebra with a nondegenerate bilinear form
$\varpi(-,-)$, $\{e_{1}, e_{2},\cdots,e_{n}\}$ be a basis of $A$ and $\{f_{1},
f_{2},\cdots,f_{n}\}$ be the dual basis of $\{e_{1}, e_{2},\cdots,e_{n}\}$ with
respect to $\varpi(-,-)$. Denote $r=\sum_{i}e_{i}\otimes f_{i}\in A\otimes A$.
Then $r$ is a symmetric solution of the $\PPYBE$ in $(A, \ast, \circ)$ if and only
if $(A, \ast, \circ, \varpi)$ is a quasi-Frobenius pre-Poisson algebra.
\end{pro}

\begin{proof}
Similar to the proof of Proposition \ref{pro:quasi-Z}, if $r$ is symmetric,
we can obtain that $r$ is a solution of the $\mathcal{S}$-equation if and only if
$\varpi(a_{1}\circ a_{2},\; a_{3})+\varpi(a_{1},\; a_{2}\circ a_{3})
=\varpi(a_{2}\circ a_{1},\; a_{3})+\varpi(a_{2},\; a_{1}\circ a_{3})$
for any $a_{1}, a_{2}, a_{3}\in A$. In addition to the conclusion
of Proposition \ref{pro:quasi-Z}, we get this proposition.
\end{proof}

\begin{defi}
Let $(\g=\oplus_{i\in\bz}\g_{i}, [-,-])$ be a $\bz$-graded Lie algebra. If there is an
antisymmetric nondegenerate graded bilinear form $\mathcal{B}(-,-)$ on
$\g=\oplus_{i\in\bz}\g_{i}$ satisfying
\begin{align*}
\mathcal{B}([g_{1}, g_{2}],\; g_{3})+\mathcal{B}([g_{2}, g_{3}],\; g_{1})
+\mathcal{B}([g_{3}, g_{1}],\; g_{2})=0,
\end{align*}
for any $g_{1}, g_{2}, g_{3}\in\g$, then $(\g=\oplus_{i\in\bz}\g_{i}, [-,-], \mathcal{B})$
is called a {\bf quasi-Frobenius $\bz$-graded Lie algebra}. 

Let $(P=\oplus_{i\in\bz}P_{i}, \cdot, [-,-])$ be a $\bz$-graded Poisson algebra
and $\mathcal{B}(-,-)$ be a antisymmetric nondegenerate graded bilinear form on
$P=\oplus_{i\in\bz}P_{i}$. If $(P, \cdot, \mathcal{B})$ is a $\bz$-graded commutative
associative algebra with Connes cocycle and $(P, [-,-], \mathcal{B})$ is a
quasi-Frobenius $\bz$-graded Lie algebra, then $(P, \cdot, [-,-], \mathcal{B})$
is called a {\bf quasi-Frobenius $\bz$-graded Poisson algebra}.
\end{defi}

Recently, Poisson algebras with some special bilinear forms and their related algebra
structures, especially quasi-Frobenius Poisson algebras, have been studied in \cite{LB}.

\begin{pro}\label{pro:Frob-PP}
Let $(A, \ast, \circ, \varpi)$ be a quasi-Frobenius pre-Poisson algebra,
$(B=\oplus_{i\in\bz}B_{i}, \diamond, \omega)$ be a quadratic $\bz$-graded perm algebra
and $(A\otimes B, \cdot, [-,-])$ be the induced Poisson algebra.
Define a bilinear form $\mathcal{B}(-,-)$ on $A\otimes B$ by Eq. \eqref{bilinear}.
Then $(A\otimes B, \cdot, [-,-], \mathcal{B})$ is a quasi-Frobenius
$\bz$-graded Poisson algebra.

Furthermore, if the quadratic $\bz$-graded perm algebra is given in Example \ref{ex:qu-perm},
then $(A\otimes B, \cdot, [-,-], \mathcal{B})$ is a quasi-Frobenius $\bz$-graded Poisson
algebra if and only if $(A, \ast, \circ, \varpi)$ is a quasi-Frobenius pre-Poisson algebra.
\end{pro}

\begin{proof}
For any $a_{1}, a_{2}, a_{3}\in A$ and $b_{1}, b_{2}, b_{3}\in B$, since
\begin{align*}
&\;\mathcal{B}\big([a_{1}\otimes b_{1},\; a_{2}\otimes b_{2}],\ \ a_{3}\otimes b_{3}\big)
+\mathcal{B}\big([a_{2}\otimes b_{2},\; a_{3}\otimes b_{3}],\ \ a_{1}\otimes b_{1}\big)\\[-1mm]
&\qquad+\mathcal{B}\big([a_{3}\otimes b_{3},\; a_{1}\otimes b_{1}],\ \
a_{2}\otimes b_{2}\big)\\[-1mm]
=&\;\Big(\varpi(a_{2}\circ a_{3},\; a_{1})-\varpi(a_{1}\circ a_{2},\; a_{3})
-\varpi(a_{1}\circ a_{3},\; a_{2})+\varpi(a_{2}\circ a_{1},\; a_{3})\Big)
\omega(b_{2}\diamond b_{3},\; b_{1})\\[-2mm]
&\quad+\Big(\varpi(a_{1}\circ a_{2},\; a_{3})
+\varpi(a_{1}\circ a_{3},\; a_{2})-\varpi(a_{3}\circ a_{2},\; a_{1})
-\varpi(a_{3}\circ a_{1},\; a_{2})\Big)\omega(b_{3}\diamond b_{2},\; b_{1})\\
=&\; 0.
\end{align*}
Thus, by Proposition \ref{pro:Con-ass}, we obtain that $(A\otimes B, \cdot, [-,-],
\mathcal{B})$ is a quasi-Frobenius $\bz$-graded Poisson algebra.
Conversely, suppose that the quadratic $\bz$-graded perm algebra $(B=\oplus_{i\in\bz}B_{i},
\diamond, \omega)$ is given in Example \ref{ex:qu-perm} and $(A\otimes B, \cdot, [-,-],
\mathcal{B})$ is a quasi-Frobenius $\bz$-graded Poisson algebra. By Proposition
\ref{pro:Con-ass}, to prove that $(A, \ast, \circ, \varpi)$ is a quasi-Frobenius pre-Poisson
algebra, we only need to show that $(A, \circ, \varpi)$ is a quasi-Frobenius
pre-Lie algebra. In fact, for any $a_{1}, a_{2}, a_{3}\in A$, note that
\begin{align*}
0&=\mathcal{B}\big([a_{1}\otimes \partial_{1},\; a_{2}\otimes \partial_{1}],\ \
a_{3}\otimes x_{1}^{-1}\partial_{2}\big)+\mathcal{B}\big([a_{2}\otimes \partial_{1},\;
a_{3}\otimes x_{1}^{-1}\partial_{2}],\ \ a_{1}\otimes \partial_{1}\big)\\[-1mm]
&\qquad+\mathcal{B}\big([a_{3}\otimes x_{1}^{-1}\partial_{2},\;
a_{1}\otimes \partial_{1}],\ \ a_{2}\otimes \partial_{1}\big)\\
&= -\varpi(a_{1}\circ a_{2}-a_{2}\circ a_{1},\; a_{3})+\varpi(a_{1},\; a_{2}\circ a_{3})
-\varpi(a_{2},\; a_{1}\circ a_{3}).
\end{align*}
We get that $(A, \circ, \varpi)$ is a quasi-Frobenius pre-Lie algebra and
$(A, \ast, \circ, \varpi)$ is a quasi-Frobenius pre-Poisson algebra.
\end{proof}

Combining Theorem \ref{thm:quasi-ass-eq}, Propositions \ref{pro:PPYBE-PYBE},
\ref{pro:quasi-PP} and \ref{pro:Frob-PP}, we obtain the equivalences.

\begin{thm}\label{thm:quasi-P-eq}
Let $(A, \ast, \circ)$ be a pre-Poisson algebra with a nondegenerate bilinear form
$\varpi(-,-)$, $\{e_{1}, e_{2},\cdots,e_{n}\}$ be a basis of $A$ and $\{f_{1},
f_{2},\cdots,f_{n}\}$ be the dual basis of $\{e_{1}, e_{2},\cdots,e_{n}\}$ with
respect to $\varpi(-,-)$. Denote $r=\sum_{i}e_{i}\otimes f_{i}\in A\otimes A$
and suppose $(A\otimes B, \cdot, [-,-])$ be the Poisson algebra induced by
$(A, \ast, \circ)$ and the $\bz$-graded perm algebra $(B, \diamond)$ given in
Example \ref{ex:grperm}. Define a bilinear form $\mathcal{B}(-,-)$ on $A\otimes B$
by Eq. \eqref{bilinear} for $\omega(-,-)$ given in Example \ref{ex:qu-perm}.
Then the following conditions are equivalent.
\begin{enumerate}\itemsep=0pt
\item[$(i)$] $(A, \ast, \circ, \varpi)$ is a quasi-Frobenius pre-Poisson algebra;
\item[$(ii)$] $r$ is a symmetric solution of the $\PPYBE$ in $(A, \ast, \circ)$;
\item[$(iii)$] $\widehat{r}:=\sum_{i}\sum_{i_{1}, i_{2}\in\bz}\big((e_{i}\otimes
  x_{1}^{i_{1}}x_{2}^{i_{2}}\partial_{1})\otimes(f_{i}\otimes x_{1}^{-i_{1}}x_{2}^{-i_{2}}
  \partial_{2})+(e_{i}\otimes x_{1}^{i_{1}}x_{2}^{i_{2}}\partial_{2})\otimes(f_{i}\otimes
  x_{1}^{-i_{1}}x_{2}^{-i_{2}}\partial_{1})\big)\in(A\otimes B)\,\hat{\otimes}\,(A\otimes B)$
  is a skew-symmetric completed solution of the $\PYBE$ in $(A\otimes B, \cdot, [-,-])$;
\item[$(iv)$] $((A\otimes B, \cdot, [-,-], \mathcal{B})$ is a quasi-Frobenius $\bz$-graded
  Poisson algebra.
\end{enumerate}
\end{thm}

\begin{ex}\label{ex:FPP-FP}
Consider the two dimensional quasi-Frobenius pre-Poisson algebra $(A=\bk\{e_{1}, e_{2}\},
\ast$, $\circ, \varpi)$, where the pre-Poisson algebra $(A=\bk\{e_{1}, e_{2}\}, \ast, \circ)$
is given in \ref{ex:Zbia-Assbia} and the bilinear form $\varpi(-,-)$ is given by
$\varpi(e_{1}, e_{2})=1$. Then $r=e_{1}\otimes e_{2}+e_{2}\otimes e_{1}$ is a
symmetric solution of the $\PPYBE$ in $(A, \ast, \circ)$. Let $(B=\bk\{x_{1}, x_{2}\},
\diamond, \omega)$ be the two dimensional quadratic perm algebra given in Example
\ref{ex:Zbia-Assbia}. Then the element $\widehat{r}$ defined by Eq. \eqref{r-indass} is a
skew-symmetric solution of the $\PYBE$ in $(A\otimes B, \cdot, \circ)$, and the bilinear
form $\mathcal{B}(-,-)$ on $A\otimes B$ is given by $\mathcal{B}(e_{2}\otimes x_{2},\;
e_{1}\otimes x_{1})=1=\mathcal{B}(e_{1}\otimes x_{2},\; e_{2}\otimes x_{1})$.
\end{ex}


\begin{thebibliography}{abc}

\bibitem{Agu} M. Aguiar,
   Pre-Poisson algebras,
   {\it Lett. Math. Phys.} {\bf 54} (2000), 263--277.

\bibitem{Arn} V.I. Arnol'd,
   {\it Mathematical methods of classical mechanics}, volume 60,
   Springer Science \& Business Media, 2013.

\bibitem{Bai} C. Bai,
   Left-symmetric bialgebras and an analogue of the classical Yang-Baxter equation,
   {\it Commun. Contemp. Math.} {\bf 10} (2008), 221--260.

\bibitem{Bai1} C. Bai,
   Double constructions of Frobenius algebras, Connes cocycles and their duality,
   {\it J. Noncommut. Geom.} {\bf 4} (2010), 475--530.

\bibitem{Baa} C. Bai,
   An introduction to pre-Lie algebras,
   {\it In: Algebra and Applications 1: Nonssociative
   Algebras and Categories}, pages 245--273, Wiley Online Library 2021.

\bibitem{BLN} C. Bai, L. Liu, X. Ni,
   Some results on $L$-dendriform algebras,
   {\it J. Geom. Phys.} {\bf 60} (2010), 940--950.

\bibitem{BN} A. Balinsky, S. Novikov,
   Poisson brackets of hydrodynamic type, Frobenius algebras and Lie algebras,
   {\it Sov. Math. Dokl.} {\bf 32} (1985), 228--231.

\bibitem{Bax} R.J. Baxter,
    Partition function of the eight-vertex lattice model,
    {\it Ann. Phys.} {\bf 70} (1972), 193--228.

\bibitem{BD}  A.A. Belavin, V. G. Drinfeld,
    Solutions of the classical Yang-Baxter equation for simple Lie algebras,
    {\it Functional Anal. Appl.} {\bf 16} (1982), 1--29.

\bibitem{Bur} D. Burde,
   Left-symmetric algebras and pre-Lie algebras in geometry and physics,
   {\it Cent. Eur. J. Math.} {\bf 4} (2006), 323--357.

\bibitem{CH} Z. Cui, B. Hou,
    Quasi-triangular Novikov bialgebras and induced quasi-triangular Lie bialgebras,
    arXiv: 2505.19579

\bibitem{Dir} P.A. Dirac,
   Lectures on quantum mechanics, Courier Corporation, 2013.

\bibitem{Dri} V. Drinfeld,
   Quantum groups,
   {\it In: Proc. Int. Congr. Math.}, volume 1, pages 798--820,
   Amer. Math. Soc., Providence RI, 1987.

\bibitem{GK} V. Ginzburg, D. Kaledin,
   Poisson deformations of symplectic quotient singularities,
   {\it Adv. Math.} {\bf 186} (2004), 1--57.

\bibitem{HBG} Y. Hong, C. Bai, L. Guo,
   Infinite-dimensional Lie bialgebras via affinization of Novikov bialgebras and
   Koszul duality,
   {\it Commun. Math. Phys.} {\bf 401} (2023), 2011--2049.

\bibitem{HBG1} Y. Hong, C. Bai, L. Guo,
    Deformation families of Novikov bialgebras via differential antisymmetric
    infinitesimal bialgebras,
    arXiv: 2402.16155.

\bibitem{Hou} B. Hou,
    Extending structures for perm algebras and perm bialgebras,
    {\it J. Algebra} {\bf 649} (2024), 392--432.

\bibitem{Hue} J. Huebschmann,
   Poisson cohomology and quantization,
   {\it J. Reine Angew. Math.} {\bf 408} (1990), 57--113.

\bibitem{JR} S.A. Joni, G.C. Rota,
    Coalgebras and bialgebras in combinatorics.
    {\it Stud. Appl. Math.} {\bf 61} (1979), 93--139.

\bibitem{Kac} V.G. Kac,
   {\it Infinite-dimensional Lie algebras, 3rd ed.},
   Cambridge University Press, Cambridge, 1990.

\bibitem{Kon} M. Kontsevich,
   Deformation quantization of Poisson manifolds,
   {\it Lett. Math. Phys.} {\bf 66} (2003), 157--216.

\bibitem{LLB} Y. Lin, X. Liu, C. Bai,
    Differential antisymmetric infinitesimal bialgebras, coherent derivations and
    Poisson bialgebras,
    Symmetry Integrability Geom. Methods Appl. {\bf 19} (2023), 018, 47pp.

\bibitem{LZB} Y. Lin, P. Zhou, C. Bai,
   Infinite-dimensional Lie bialgebras via affinization of perm bialgebras and pre-Lie
   bialgebras,
   {\it J. Algebra} {\bf 663} (2025), 210--258.

\bibitem{LB} G. Liu, C. Bai,
   Anti-pre-Lie algebras, Novikov algebras and commutative 2-cocycles on Lie algebras,
   {\it J. Algebra} {\bf 609} (2022), 337--379.

\bibitem{LL} Y. Lin, D. Lu,
   Quasi-triangular and factorizable Poisson bialgebras,
   arXiv: 2506.21870.

\bibitem{LB} G. Liu, C. Bai,
   New splittings of operations of Poisson algebras and transposed Poisson algebras and
   related algebraic structures,
   {\it In Algebra without Borders-Classical and Constructive Nonassociative Algebraic
   Structures: Foundations and Applications}, pages 49--96, Springer, 2023.


\bibitem{Liv} M. Livernet,
   Rational homotopy of Leibniz algebras,
   {\it Manuscripta Math.} {\bf 96} (1998), 295--315.

\bibitem{Lod1} J.L. Loday,
   Cup product for Leibniz cohomology and dual Leibniz algebras,
   {\it Math. Scand.} {\bf 77}, (1995), 189--196.

\bibitem{Lod2} J. L. Loday,
   Dialgebras,
   {\it In Dialgebras and related operads}, Lecture Notes in Math. 1763,
    pages 7--66, Springer, Berlin, 2001.

\bibitem{NB} X. Ni, C. Bai,
   Poisson bialgebras,
   {\it J. Math. Phys.} {\bf 54} (2013), 023515.

\bibitem{Pol} A. Polishchuk,
   Algebraic geometry of Poisson brackets,
   {\it J. Math. Sci.} {\bf 84} (1997), 1413--1444.

\bibitem{Sto} A. Stolin,
    Rational solutions of the classical Yang-Baxter equation and quasi-Frobenius Lie algebras,
    {\it J. Pure Appl. Algebra} {\bf 137} (1999), 285--293.

\bibitem{Vai} I. Vaisman,
   {\it Lectures on the geometry of Poisson manifolds},
   volume 118, Birkh{\"a}user, 2012.

\bibitem{Wan} Y. Wang,
   Zinbiel bialgebras, relative Rota-Baxter operators and the related Yang-Baxter equation,
   {\it J. Algebra} {\bf 689} (2026), 656--689.

\bibitem{WS} Y. Wang, Y. Sheng,
   Phase space of a Poisson algebra and the induced Pre-Poisson bialgebra,
   arXiv: 2504.2093.

\bibitem{Wei} A. Weinstein,
   Lectures on symplectic manifolds,
   {\it CBMS Regional Conference Series in Mathematics} 29,
   Amer. Math. Soc., Providence 1979.

\bibitem{Yang} C.N. Yang,
    Some exact results for the many-body problem in one dimension with repulsive
    deltafunction interaction,
    {\it Phys. Rev. Lett.} {\bf 19} (1967), 1312--1314.

\bibitem{Zhe} V.N. Zhelyabin,
    Jordan bialgebras and their relation to Lie bialgebras,
    {\it Algebra Logic}, {\bf 36} (1997), 1--15.





\end{thebibliography}
 \end{document}